\documentclass[english, 10pt]{amsart}
\usepackage{euscript,graphicx,epstopdf,amscd,amsgen,amsfonts,amssymb,latexsym,amsmath,amsthm,graphicx,mathrsfs,times,color,overpic,fourier,array}
\usepackage[all]{xy}
\usepackage{enumitem}
 \usepackage[colorlinks=true, backref=page]{hyperref}
\usepackage{xcolor}
\definecolor{forestgreen}{rgb}{0.0, 0.27, 0.13}

\usepackage[T1]{fontenc}

\newtheorem{theorem}{Theorem}[section]
\newtheorem{theoremalph}{Theorem}

\newtheorem{coroalph}[theoremalph]{Corollary}
\newtheorem{lemma}[theorem]{Lemma}
\newtheorem{claim}[theorem]{Claim}
\newtheorem{subclaim}[theorem]{Subclaim}
\newtheorem*{claim*}{Claim}
\newtheorem{corollary}[theorem]{Corollary}
\newtheorem{proposition}[theorem]{Proposition}

\newtheorem{assumption}[theorem]{Assumption}

\theoremstyle{definition}
\newtheorem{definition}[theorem]{Definition}
\newtheorem{notation}[theorem]{Notation}
\newtheorem{fact}[theorem]{Fact}

\newtheorem{remark}[theorem]{Remark}
\newtheorem*{maintheorem*}{Theorem}

\newcommand{\eqdef}{\stackrel{\scriptscriptstyle\rm def}{=}}

\DeclareMathOperator{\diam}{diam}
\DeclareMathOperator{\dist}{dist}
\DeclareMathOperator{\card}{card}
\DeclareMathOperator{\interior}{int}
\DeclareMathOperator{\Lip}{Lip}

\DeclareMathOperator{\Leb}{Leb}
\DeclareMathOperator{\length}{length}
\DeclareMathOperator{\Diff}{Diff}
\def\vr{{\mathbf r}}

\def\Sub{\mathcal{L}}

\def\cW{\EuScript{B}}
\def\cS{\EuScript{S}}

\def\cD{\mathcal{D}}
\def\cR{\mathcal{R}}
\def\cL{\EuScript{L}}

\def\fm{\mathfrak{m}}

\def\fB{\mathfrak{B}}

\DeclareMathOperator{\Mod}{Mod}

\def\ua{\underline{a}}
\def\ub{\underline{b}}

\def\ux{\underline{x}}
\def\uy{\underline{y}}

\def\ba{\mathbf{a}}

\def\bt{\mathbf{t}}
\def\bC{\mathbf{C}}

\def\bfK{\mathbf{K}}

\def\bfC{\mathbf{C}}
\def\bfH{\mathbf{H}}

\def\rmb{\mathrm{b}}

\def\BM{{\rm BM}}
\def\PH{{\rm PH}}

\def\c{{\rm c}}
\def\csu{{\rm scu}}
\def\s{{\rm s}}
\def\ss{{\rm ss}}
\def\uu{{\rm uu}}
\def\u{{\rm u}}
\def\cs{{\rm cs}}
\def\cu{{\rm cu}}

\def\bfC{\mathbf{C}}

\def\bH{\mathbb{H}}
\def\bN{\mathbb{N}}

\def\bZ{\mathbb{Z}}

\def\bR{\mathbb{R}}

\def\cA{\EuScript{A}}

\def\cC{\EuScript{C}}

\def\cW{\mathcal{W}}

\def\cM{\EuScript{M}}

\def\c{{\rm c}}
\def\s{{\rm s}}
\def\u{{\rm u}}
\def\cs{{\rm cs}}
\def\cu{{\rm cu}}

\def\FK{{\rm FK}}

\def\bfa{\mathbf a}

\numberwithin{equation}{section}

\DeclareMathSymbol{\varnothing}{\mathord}{AMSb}{"3F}
\renewcommand{\emptyset}{\varnothing}

\title[Full flexibility of entropies]{Full flexibility of entropies among ergodic measures\\ for partially hyperbolic diffeomorphisms}

\author[L.~J.~D\'iaz]{Lorenzo J. D\'\i az}
\address{Departamento de Matem\'atica PUC-Rio, Marqu\^es de S\~ao Vicente 225, G\'avea, Rio de Janeiro 22451-900, Brazil}
\email{lodiaz@puc-rio.br}

\author[K.~Gelfert]{Katrin~Gelfert}
\address{Instituto de Matem\'atica Universidade Federal do Rio de Janeiro, Av. Athos da Silveira Ramos 149, Cidade Universit\'aria - Ilha do Fund\~ao, Rio de Janeiro 21945-909,  Brazil}\email{gelfert@im.ufrj.br}

\author[M.~Rams]{Micha\l~Rams}
\address{Institute of Mathematics, Polish Academy of Sciences, ul. \'Sniadeckich 8, 00-656 Warsaw, Poland.} 
\email{rams@impan.pl}

\author[J.~Zhang]{Jinhua~Zhang}
\address{School of Mathematical Sciences, Beihang University\\ 100191, Beijing, P.R. China}
\email{jinhua$\_$zhang@buaa.edu.cn}

\thanks{J.Zhang was partially supported by National Key R\&D Program of China (2022YFA1005801),  NSFC 12471179, the Fundamental Research Funds for the Central Universities, and E-16/2014 INCT/FAPERJ (Brazil).  
M.Rams was supported by National Science Centre grant 2019/33/B/ST1/0027 (Poland).
L.J.D\'iaz and K.Gelfert have been supported [in part] by 
CAPES -- Finance Code 001, by 
CNPq-grants  310069/2020-3 and 
305327/2022-4,  
CNPq Projeto Universal 430154/2018-6  and 
404943/2023-3. 
  This study was funded by FAPERJ – Carlos Chagas Filho Foundation for Research Support of the State of Rio de Janeiro, Processes SEI E-26/200.371/2023 and 
    E-16/2014 INCT/FAPERJ.  
The authors thank their home institutions for their hospitality while preparing this paper. 
} 
\begin{document}

\begin{abstract}
We study nonhyperbolic and transitive partially hyperbolic diffeomorphisms having a one-dimensional center. We prove joint flexibility with respect to entropy and center Lyapunov exponent for a broad class of these systems. Flexibility means that for any given value of the center Lyapunov exponent and any value of entropy less than the supremum of entropies of ergodic measures with that exponent, there is an ergodic measure with exactly this entropy and exponent. Our hypotheses involve minimal foliations and blender-horseshoes, they formalize the interplay between two regions of the ambient space, one of center expanding and the other of center contracting type. The list of examples our results apply is rather long,  a non-exhaustive list includes fibered by circles, flow-type, some Derived from Anosov diffeomorphisms,  and some anomalous (non-dynamically coherent) diffeomorphisms.
\end{abstract}

\maketitle

\section{Introduction}

In the context of nonhyperbolic partially hyperbolic diffeomorphisms of dimension three or higher, we prove
flexibility with respect to the pair [entropy, center Lyapunov exponent] for a broad class of systems. Flexibility, in this case, means that for any given value of the center Lyapunov exponent and any value of entropy less than the supremum of entropies of ergodic measures with that exponent, there exists an ergodic measure with exactly this entropy and exponent. Our setting involves one-dimensional center bundle which allows to consider the associated Lyapunov exponent, minimal strong foliations, and blender-horseshoes. This formalizes the interplay between two regions of the ambient space, one of center expanding and one of center contracting type. 
A short, non-exhaustive, list of cases to which our results apply is as follows:
  \begin{itemize}[leftmargin=0.6cm ]
  \item fibered by circles, \cite{BonDia:96}, 
  \item  flow-type,  \cite{BonDia:96, BuzFisTah:22},
  \item some Derived from Anosov diffeomorphisms, \cite{RodUreYan:22}, and
  \item  some anomalous diffeomorphisms in \cite{BonGogHamPot:20}.
  \end{itemize}
We now go into the details.

\subsection{Flexibility program}
The ``flexibility program'' in dynamics was initiated by Katok (for the original formulation and further discussion see \cite{BocKatRod:22,ErcKat:19}). It refers to verifying if
\begin{itemize}[leftmargin=0.6cm ]
\item[] {\em{within a certain realm of dynamical systems, a given dynamical quantity takes arbitrary values.}}
\end{itemize}  
This program has been implemented in various contexts%
\footnote{The flexibility program in smooth dynamics with special focus on Lyapunov exponents was outlined and discussed, for example, in \cite{BocKatRod:22} , where they prove results for volume-preserving $C^2$ Anosov diffeomorphisms admitting dominated splitting into one-dimensional bundles.}. Given a dynamical system, one can ask whether its ergodic measures realize the full spectrum of, for example, entropies. For a continuous map $f$ on a compact metric space, the variational principle for topological entropy states that 
\[
	h_{\rm top}(f)=\sup_\mu h_\mu(f),
\]	 
where  $h_\mu(f)$ denotes the entropy of an $f$-invariant probability measure. This leads to the question of whether, for every $h\in[0,h_{\rm top}(f))$ there exists an ergodic measure with entropy $h$. In the context of smooth dynamics, this property has often been referred to as a conjecture of Katok. In fact, it follows from his results in \cite{Kat:80}, which imply that for $C^{1+\alpha}$-surface diffeomorphisms with positive entropy,  any ergodic measure with entropy arbitrarily close to $h_{\rm top}(f)$ can be approximated (simultaneously in entropy and in the weak$\ast$ topology) by measures supported on horseshoes.

This ``intermediate entropy property''  can be also formulated within a given subclass of ergodic measures. For example, given a continuous function $\varphi\colon X\to\bR$ and $\alpha$ a number within the range of averages $\{\int\varphi\,d\mu\colon\mu\text{ $f$-invariant probability}\}$, can every value
\[
	h\in[0,h(\alpha)],\quad\text{ where }\quad
	h(\alpha)\eqdef \sup\Big\{h_\mu(f)\colon\int\varphi\,d\mu=\alpha\Big\},
\]
be attained by some ergodic measure? This question can be seen as a mixture of flexibility program with multifractal analysis. It was verified in \cite{HouTia:24} for basic sets which, in particular, satisfy the specification property%
\footnote{The specification property allows, roughly speaking, an arbitrary concatenation of orbit pieces. After contributions by Sigmund \cite{Sig:74}, this property allows to describe many properties of the space of invariant measures.}. 
We are interested in this conjecture in settings where specification properties do not hold, and where ergodic measures with ``different types of hyperbolicity'' coexist and interact within the same piece of transitive dynamics. 
Indeed this gives rise to ergodic measures with some exponent equal to zero 
(called {\em{nonhyperbolic}})
and having positive entropy. Observe that there are no general tools to deal with such measures (such as Pesin theory)
and those measures can have very rich dynamics.
Theorem~\ref{thm.main} proves Katok's conjecture for a large class of partially hyperbolic systems.
 We also show  that the maximal complexity of the nonhyperbolic part can be expressed in terms of entropy of nonhyperbolic measures, see Corollary~\ref{c.vp}. 
 
 We now proceed to describe the systems we consider.

\subsection{Setting}
Let $M$ be a compact Riemannian manifold without boundary, and denote by $\Diff^1(M)$ the space of $C^1$ diffeomorphisms on $M$ endowed with the $C^1$ topology. A diffeomorphism $f\in\Diff^1(M)$ is \emph{partially hyperbolic}, if there exists a $Df$-invariant dominated splitting 
\begin{equation}\label{defsplitt} 
	TM=E^\ss \oplus E^\c\oplus E^\uu 
\end{equation}	 
such that $E^\ss$ (resp. $E^\uu$) is uniformly contracting (resp. expanding), see Section \ref{ss.prescribed} for details. Denote by $\PH^1_{\c =1}(M)$ the (open) set of $C^1$ partially hyperbolic diffeomorphisms on $M$ with a one-dimensional \emph{center bundle} $E^\c$. Due to  \cite{HirPugShu:77}, the \emph{strong stable bundle} $E^\ss$  and the \emph{strong unstable bundle} $E^\uu$ are uniquely integrable. We denote the strong stable and unstable foliations 
that integrate these bundles by $\cW^\ss$ and $\cW^\uu$,   respectively.  As the center bundle is one-dimensional, it (locally) defines a continuous vector field, by Peano's theorem, at each point, there are $C^1$ curves tangent to $E^\c$, which we simply call \emph{center curves}. In general, these curves may not form a foliation.

In what follows, we consider diffeomorphisms $f\in \PH^1_{\c =1}(M)$ and denote by $\cM_{\rm erg}(f)$ and $\cM(f)$ the sets of $f$-ergodic and $f$-invariant measures, respectively. For any ergodic measure $\mu$ of $f\in\PH_{\c=1}^1(M)$, its \emph{center Lyapunov exponent} is defined by 
\begin{equation}\label{defchic}
	\chi^\c(\mu)
	\eqdef \int\log\|D^\c f\|\,d\mu,
	\quad\text{ where }\quad
	D^\c f(x)\eqdef Df|_{E^\c(x)}.
\end{equation}
Note that, as $E^\c$ is a continuous bundle,  the latter is a continuous function. We only focus on those exponents. Define 
\begin{equation}\label{defminmax}
	\chi_{\rm max}
	\eqdef \max\big\{\chi^\c(\mu)\colon\mu\in\cM_{\rm erg}(f) \big\}
	\quad\text{ and }\quad
	\chi_{\rm min}
	\eqdef\min\big\{\chi^\c(\mu)\colon\mu\in\cM_{\rm erg}(f) \big\},
\end{equation}
omitting the dependence on $f$ in this notation, and for each  $\chi\in[\chi_{\rm min},\chi_{\rm max}]$, consider the level set of ergodic measures
$$
\cM_{\rm erg,\chi}(f)
	\eqdef \big\{\mu\colon \mu\in\cM_{\rm erg}(f),\chi^\c(\mu)=\chi\big\}.
$$	
This allows to split the set of ergodic measures of $f$ into three subsets according to the sign of the exponent:   {\em{nonhyperbolic measures}} if $\chi=0$, {\em{center contracting measures}}
if $\chi<0$, and {\em{center expanding measures}} if $\chi>0$. The last two classes are called
{\em{hyperbolic.}}
Under our assumptions, each of these sets is nonempty, and there exists a genuine interaction between them that we explore. 

Let us introduce the class of \emph{robustly transitive and nonhyperbolic diffeomorphisms}%
\footnote{Both properties, transitivity and nonhyperbolicity, hold simultaneously robustly.}, focusing  on those in $\mathrm{PH}^1_{\c=1}(M)$, denoting this set by $\mathrm{RNT}^1_{\c=1}(M)$. This class includes those satisfying our hypotheses. A diffeomorphism is {\em{transitive}} if it has a dense orbit, and \emph{$C^1$-robustly transitive} if it has a $C^1$-neighbor\-hood consisting of transitive diffeomorphisms. The first examples of $C^1$-robustly transitive diffeomorphisms were constructed by Anosov \cite{Ano:67}, and carry now his name. While Anosov diffeomorphisms are uniformly hyperbolic, the class of robustly transitive ones is broader and includes nonhyperbolic examples, as shown in the pioneering works of Shub and Ma\~n\'e \cite{Shu:71, Man:78}. This leads to the subcategory of {\em{robustly transitive and nonhyperbolic diffeomorphisms,}} where diffeomorphisms have a neighborhood of maps that are not hyperbolic. 
A detailed discussion of  $\mathrm{RNT}^1_{\c=1}(M)$ is beyond the scope of this paper, for further details on this extensively studied class see \cite[Section 2.1.1]{DiaGelZha:}.

Let us mention here that as our techniques are semi-local, our results could be formulated simply assuming that the dynamics is confined to a bounded invariant region of a manifold. This kind of formulation would, however, force us to define some assumptions (like minimality) in a more complicated way, and we are not going to go this way.
  
\subsection{Main results}

To state our main results, the first ingredient is {\em{blender-horseshoes,}} a special type of hyperbolic sets that appear in dimensions three or higher. For a heuristic explanation of a blender, we refer to \cite{BonCroDiaWil:16}. 
Blenders were introduced to study robust transitivity in  \cite{BonDia:96}. Thereafter
they have been formulated in various ways, adapted to specific settings and goals, yet all formulations preserve the same essential characteristics%
\footnote{These variations include the
blender-horseshoes \cite{BonDia:12} used in this paper,
symbolic blenders \cite{BarKiRai:12},
symplectic blenders \cite{NasPuj:12}, 
dynamical blenders \cite{BocBonDia:16},
parablenders \cite{Ber:16}, 
almost blenders \cite{Bie:17},
superblenders \cite{AviCroWil:21},
blender machines \cite{Asa:22}, 
conformal blenders \cite{FakNasRaj:25},
and adaptation of blenders to the complex setting  \cite{Duj:17}.
They have proven to be  useful in many contexts beyond the study of transitivity.
The applications  include, besides applications to study entropy properties of measures and the space of ergodic measures, robust cycles  and tangencies, fast growth of periodic orbits, historical behavior,
stable ergodicity (including actions of groups of diffeomorphisms),  instability problems (Arnold diffusion) in symplectic dynamics, and properties of the study of complex automorphisms  and polynomial and holomorphic maps. These ample list of applications and the fact that our arguments are semi-local suggest that the techniques involved in this paper could be ``translated'' to other settings,
 see the comment
at the end of Section~\ref{seBM1}.}.
We call a blender-horseshoe {\em{unstable}} if its center bundle  is uniformly expanding and {\em{stable}} if its center bundle is uniformly contracting. See Section~\ref{ss.blenderhorseshoes} for details. 

The second ingredient is the strong stable and unstable foliations and their minimality. Recall that a foliation is \emph{minimal} if every leaf of it is dense in the ambience. 

Finally, we consider, for every $\chi\in[\chi_{\rm min},\chi_{\rm max}],$  the level set of Lyapunov exponents
\begin{equation}\label{defLchi}
	\cL(\chi)
	\eqdef\Big\{x\in M\colon \lim_{|n|\rightarrow+\infty}\frac{1}{n}\log\|Df^n|_{E^\c(x)}\|=\chi\Big\}.
\end{equation}
Note that this set is, in general, a noncompact ``fractal'' set. To measure its ``complexity'', we rely on the concept of topological entropy $h_{\rm top}(f,\cdot)$ as defined in \cite{Bow:73}.

\begin{theoremalph}\label{thm.main}
	Let  $f\in \mathrm{PH}^1_{\c=1}(M)$ such that 
\begin{itemize}[leftmargin=0.6cm ]
\item[(1)] $f$ has an unstable and a stable blender-horseshoe,
\item[(2)] the foliations $\cW^\ss$ and $\cW^\uu$ are both  minimal.
\end{itemize}
Then for every $\chi\in(\chi_{\rm min},\chi_{\rm max})$ and every $\varepsilon>0$ small, there exists a continuous path 
\[\{\mu_t\}_{t\in[0,1]}\subset  \cM_{\rm erg,\chi}(f)\]
such that 
\begin{itemize}[leftmargin=0.6cm ]
\item $t\mapsto h_{\mu_t}(f)$ is continuous;
\item $h_{\mu_0}(f)=0$ and $h_{\mu_1}(f)>h_{\rm top}(f,\cL(\chi))-\varepsilon.$
\end{itemize}
\end{theoremalph}

\begin{definition}[The class $\mathrm{BM}^1(M)$]
We denote by $\mathrm{BM}^1(M)$ the set of diffeomorphisms satisfying hypotheses (1) and (2) of Theorem~\ref{thm.main}. 
\end{definition}

In Section \ref{seBM1} we further discuss this large class of diffeomorphisms. This set was introduced independently in \cite{BonZha:19, DiaGelSan:20},   motivated by \cite{DiaGelRam:17}. For diffeomorphisms in that class we have now a quite complete description of the space of ergodic measures in terms of weak$\ast$ and entropy approximation.

\begin{remark}\label{r.BBDlevels}
For every $f\in \mathrm{BM}^1(M)$ and every $\chi\in[\chi_{\rm min},\chi_{\rm max}]$, one has $\cL(\chi)\neq \emptyset$, see for example \cite[Section 6.1]{BocBonDia:16}. Moreover, $h_{\rm top}(f,\cL(\chi))>0$, for every $\chi\in(\chi_{\rm min},\chi_{\rm max})$. The challenging part is to prove the inequality for $\chi=0$, which  was proved in \cite{BocBonDia:16}. To investigate the entropy at the extremes $\chi=\chi_{\rm min}$ and $\chi=\chi_{\rm max}$ remains an open problem related to ergodic optimization. 
\end{remark}

\begin{remark}\label{r.DGZ}
Given any $f\in \mathrm{BM}^1(M)$ it is shown in \cite{DiaGelZha:} that	
\begin{itemize}[leftmargin=0.6cm ]
\item the map $\chi\mapsto h_{\rm top}(f,\cL(\chi))$ varies continuously with $\chi\in [\chi_{\rm min},\chi_{\rm max}]$,
\item  for every $\chi\in[\chi_{\rm min},\chi_{\rm max}]\setminus\{0\}$, one has the \emph{restricted variational principle}
\[
	h_{\rm top}(f,\cL(\chi))=\sup\big\{h_\mu(f)\colon \mu\in\cM_{\rm erg,\chi}(f)\big\} .
\]
\end{itemize} 
\end{remark}

From Theorem~\ref{thm.main}, we derive the following. 

\begin{coroalph}[Restricted variational principle for $\chi=0$]\label{c.vp}
Let $f\in\mathrm{BM}^1(M)$. Then 
\[
	h_{\rm top}(f,\cL(0))=\sup\big\{h_\mu(f)\colon\mu\in\cM_{\rm erg,0}(f)\big\}.
\]
\end{coroalph}

First, by Corollary \ref{c.vp}, the maximal complexity of the nonhyperbolic part of the dynamics can be expressed  in terms of entropies of ergodic measures. Moreover, Theorem~\ref{thm.main} sheds some more light on the space of nonhyperbolic ergodic measures proving that it contains path connected components varying continuously in entropy. For a dynamical system satisfying the specification property, it follows from \cite{Sig:77,LinOlsSte:78} that the space of its ergodic measures is path connected. Under the hypotheses considered in this paper, this property was shown in \cite{DiaGelSan:20,YanZha:20}. That each of the sets $\cM_{\rm erg,<0}(f)$ and $\cM_{\rm erg,>0}(f)$ is path connected is also a consequence of \cite{GorPes:17}.
  
\subsection{The set $\mathrm{BM}^1(M)$}\label{seBM1}
  
At first glance, the conditions defining this set may seem quite restrictive. However,
this set is actually rather large. Indeed,
 Table~\ref{thetable} highlights key features in the study of partially hyperbolic dynamics: existence of compact (both periodic and non-periodic) center leaves and dynamical coherence%
\footnote{A diffeomorphism $f\in \mathrm{PH}^1_{\c=1}(M)$ is {\em{dynamically coherent}} if there are simultaneously $f$ invariant foliations tangent to $E^\ss\oplus E^\c$ and to $E^\c\oplus E^\uu$, which give rise to a center foliation tangent to $E^\c$.}. 
Some examples in this table are classical (for example, fibered by circles, flow-type, and some derived from Anosov or DA systems) while being dynamically coherent. However, the table also includes examples that are non-dynamically coherent and exhibit non-periodic center leaves. These belong to the recently discovered and developed class of anomalous partially hyperbolic systems, which is vast and likely to see further developments in the near future.

\begin{tiny}
\begin{table}[h]\caption{Properties and examples in $\mathrm{BM}^1(M)$}
\begin{center}
\begin{tabular}{|m{1.5cm}|c|c|c|m{2.5cm}|}
\hline
 & compact center leaves & periodic compact center leaves & dynamically coherent & some references \\
\hline
some DA 	& no & no & yes & \cite{RodUreYan:22} \\
\hline
fibered by circles & yes & yes & yes & \cite{BonDiaUre:02,RodRodUre:07}\\
\hline
flow-type & yes & yes & yes & \cite{BonDiaUre:02,RodRodUre:07,BuzFisTah:22}\\
\hline
anomalous & yes & no & no & \cite{BonGogHamPot:20}\\
\hline
\end{tabular}
\end{center}
\label{thetable}
\end{table}
\end{tiny}
  
\begin{remark}[Occurrence of blender-horseshoes is open and dense in $\mathrm{RNT}^1_{\c=1}(M)$]\label{r.open}
First, blender-horseshoes, like any hyperbolic set, have continuations, making their existence a $C^1$-robust property. We also note that the diffeomorphisms with both a stable and an unstable blender-horseshoe are dense in $\mathrm{RNT}^1_{\c=1}(M)$. Indeed, as a consequence of \cite{Hay:97,BonCro:04}, the diffeomorphisms having a {\em{heterodimensional cycle}}
 (i.e., there are saddles of center contracting and center expanding types whose invariant manifolds intersect 
 cyclically\footnote{The term heterodimensional refers to the fact that the stable bundles of these saddles have different dimensions and thus the sets have different type of hyperbolicity. Note that, due to dimension deficiency, some of these intersections cannot be transverse.}) form a dense subset of $\mathrm{RNT}^1_{\c=1}(M)$. Moreover, such cycles lead, after small $C^1$-perturbations, to blender-horseshoes, that can be chosen of either type, see \cite{BonDia:08}\footnote{The results in \cite{BonDia:08} are
  stated in terms of robust cycles. Blender-horseshoes were later introduced in \cite{BonDia:12} to provide the appropriate framework for studying robust cycles, formalizing \cite{BonDia:08}, where the concept of blender-horseshoes appeared without an explicit name.}.
  \end{remark}  
 
\begin{remark}[Robust minimality of the strong foliations]\label{r.robmin}
Note that the minimality of the foliations $\cW^\ss$ and  $\cW^\uu$ is not an open property. However, if both $\cW^\ss$ and $\cW^\uu$ are minimal and there are a stable blender-horseshoe and an unstable blender-horseshoe, then these foliations are indeed robustly minimal: for every $g$ close to $f$, its associated foliations $\cW^\ss_g$ and $\cW^\uu_g$ remain minimal. This fact is a key argument in \cite{BonDiaUre:02}\footnote{It is worth noting that \cite{BonDiaUre:02} was written before the concept of a blender-horseshoe was formally introduced in \cite{BonDia:12}. However, the object involved in the proofs  in \cite{BonDiaUre:02} is indeed the unstable blender-horseshoe.}.
  \label{r.minimalitystrong}
  \end{remark}
  
For a broader discussion of blender-horseshoes and minimality of the foliations, we refer to \cite[Sections 2.1.2 and 2.1.3]{DiaGelZha:}. As the minimality of the foliation $\cW^\uu$ (or $\cW^\ss$) immediately implies the transitivity of $f$, from Remarks~\ref{r.open} and \ref{r.minimalitystrong} we get the following.
    
\begin{remark}
$\mathrm{BM}^1(M)$ is an open  subset of $\mathrm{RNT}^1_{\c=1}(M)$.
\end{remark}

In view of our techniques here, let us close the discussion of $\mathrm{BM}^1(M)$ with a philosophical observation. On one hand, what is above presents a general framework for the interaction between two types of hyperbolicity. On the other hand, considering the versatility of blenders in various settings, that our constructions are semi-local,  and  that cyclic configurations are in the realm of nonhyperbolic dynamics, we believe that the conditions in this paper have the potential to be adapted to other dynamical contexts.

\subsection{Further results}
\label{ss.further}

To prove Theorem~\ref{thm.main}, we use a more general result that involves the minimality of  only one strong foliation 
and the corresponding type of blender-horseshoe.

\begin{theoremalph}\label{thm.maingeneral}
	Let $f\in \mathrm{PH}^1_{\c=1}(M)$ such that
\begin{itemize}[leftmargin=0.6cm ]
\item[(1)] the foliation $\cW^\uu$ is  minimal,
\item[(2)] $f$ has an unstable blender-horseshoe,
\item[(3)] $f$ has a saddle of contracting type.
\end{itemize}
	 Then, for every $\chi\in(\chi_{\rm min},0]$ and every $\varepsilon>0$, there exists a continuous path $\{\mu_t\}_{t\in[0,1]}\subset \cM_{\rm erg,\chi}(f)$ such that 
\begin{itemize}[leftmargin=0.6cm ]
	\item $t\mapsto h_{\mu_t}(f)$ is continuous,
	\item it holds 
\[ 	
	h_{\mu_0}(f)=0,
	\quad\text{ and }\quad
	h_{\mu_1}(f)\geq \limsup_{\chi'\nearrow\chi}
	\left(
	\sup\Big\{h_\mu(f)\colon\mu\in\cM_{\rm erg,\chi'}(f)\Big\} \right) - \varepsilon.
\]
\end{itemize}
\end{theoremalph}

\begin{definition}[The classes $\mathrm{BM}^1_\u(M)$ and $\mathrm{BM}^1_\s(M)$]
We denote by $\mathrm{BM}^1_\u(M)$ the set of diffeomorphisms satisfying the hypotheses (1)--(3) of Theorem~\ref{thm.maingeneral}. Analogously, we can define the set  $\mathrm{BM}^1_\s(M)$ considering the foliation $\cW^\ss$,  a stable blender-horse\-shoe, and a saddle of expanding type.
\end{definition}

Applying Theorem \ref{thm.maingeneral} to $f^{-1}$ instead of $f$, we obtain its analogous version for $f\in \mathrm{BM}^1_\s(M)$. 
It follows from Remark \ref{r.robmin} that both subsets $\mathrm{BM}^1_\s(M)$ and $\mathrm{BM}^1_\u(M)$ are open sets. Note also that 
\[
	\mathrm{BM}^1 (M) =   \mathrm{BM}^1_\s(M) \cap \mathrm{BM}^1_\u(M).
\]	 
For $f\in\BM^1(M)$,  the hypothesis (3) on the saddles become superfluous, as they are already provided by the blender-horseshoes.

\begin{remark}
The set $\mathrm{BM}_\u^1(M) \cup \mathrm{BM}_\s^1(M)$ is an open and dense subset of $\mathrm{RNT}^1_{\c=1} (M)$, see \cite{BonDiaUre:02,RodRodUre:07}.
\end{remark}

\subsection{The space of ergodic measures}

We can frame the results of this paper within the study of the space of invariant and ergodic measures in settings beyond uniform hyperbolicity, where specification properties do not hold. In our setting, measures with different types of hyperbolicity (specifically, center contracting and center expanding) coexist and interact. 
Thus, considering convex combinations of ergodic measures with different type of hyperbolicity is natural.  A natural problem is to decide which measures in those combinations are in the closure of the space of ergodic measures. There are some situations where it does not happen, in particular,  \cite[Corollary 3]{DiaGelRam:19}
 provides examples of two ergodic measures (maximal entropy measures with positive and negative center Lyapunov exponents) whose convex combinations are not approximated by ergodic measures in weak$*$-topology and in entropy.
Here we obtain some results in the opposite direction. Theorem~\ref{nonerglims} below, derived from our constructions, shows how sequences of ergodic measures can converge to a nonergodic measure whose ergodic decomposition includes both center expanding and center contracting measures. 

For what is below, we split the set of hyperbolic measures into two parts:
  
\begin{equation}\label{def:Msubspaces}\begin{split}
	\cM_{\rm erg, <0}(f)
	&\eqdef\big\{\mu\colon \mu\in\cM_{\rm erg}(f),\chi^\c(\mu)<0\big\}
	=\bigcup_{\chi\in[\chi_{\rm min},0)}\cM_{\rm erg,\chi}(f),\\
	\cM_{\rm erg, >0}(f)
	&\eqdef\big\{\mu\colon \mu\in\cM_{\rm erg}(f),\chi^\c(\mu)>0\big\}
	=\bigcup_{\chi\in(0,\chi_{\rm max}]}\cM_{\rm erg,\chi}(f).
\end{split}\end{equation}
Define also $\cM_{\rm erg, \leq 0}(f)$ and $\cM_{\rm erg, \geq 0}(f)$ in the obvious ways.

For the following, we say that measures are \emph{close in the weak$\ast$ and entropy topology} if they are close in the Wasserstein distance (see \eqref{eqWasser} for definition) and in entropy.

\begin{theoremalph}\label{nonerglims}
Let $f\in\PH_{\c=1}^1(M)$ such 
\begin{itemize}[leftmargin=0.6cm ]
\item[(1)] $f$ has an unstable and a stable blender-horseshoe,
\item[(2)] the foliations $\cW^\ss$ and $\cW^\uu$ are both  minimal.
\end{itemize}
	Let $\vartheta\in\cM_{\rm erg,\le0}(f)$ and $\Lambda_\fB$ be an unstable blender-horseshoe. Then for every neighborhood $U$ of $\vartheta$ in the weak$\ast$-topology and every $\varepsilon>0$, there is a sequence $(\mu_k)_k\subset\cM_{\rm erg}(f)$ converging to $\mu_\infty\in\cM(f)\cap U$ such that 
\begin{enumerate}[leftmargin=0.6cm ]
	\item  $|h_{\mu_{\infty}}(f)-h_{\vartheta}(f)|<\varepsilon$,
	\item there exist  $\mu_{<0}\in\cM_{\rm erg,<0}(f)\cap U$, $\mu_{>0}\in\cM(f|_{\Lambda_\fB})$, and $\alpha\in(0,1)$ such that 
	\[ \mu_\infty=\alpha\mu_{<0}+(1-\alpha)\mu_{>0}.\]
\end{enumerate}
\end{theoremalph}

An analogous result gives a nonergodic limit measure combining an ergodic measure of expanding type with a stable blender-horseshoe. 

\begin{remark}
The    ergodic decomposition of $\mu_\infty$ consists of measures in both $\cM_{\rm erg,<0}(f)$ and $\cM_{\rm erg,>0}(f)$, but none in $\cM_{\rm erg,0}(f)$.
\end{remark}

A result with the same flavor was obtained \cite[Theorem C]{BonZha:19},  which states that the measures supported on the vicinity of a heterodimensional cycle of a saddle of contracting type and an unstable blender-horseshoe%
\footnote{With the terminology in \cite{BocBonDia:16}, this configuration is called {\em{split flip-flop.}}}  are limit of {\em{periodic measures}} (ergodic measures supported on periodic orbits). We note that in  \cite{BonZha:19} no entropy estimate is obtained, while here we construct measures of every possible entropy. 

\subsection{Techniques}
Let us discuss our techniques. Given a hyperbolic ergodic measure with positive entropy, we construct horseshoes which ``ergodically'' approximate it, following classical ideas of Katok \cite{Kat:80} and using the Brin-Katok theorem \cite{BriKat:83}. Our partially hyperbolic context allows us to construct these horseshoes, for which we introduce the concept of \emph{$\csu$-cubes}. An $\csu$-cube is obtained by, roughly speaking, taking a ``base'' tangent to $E^\cs=E^\ss\oplus E^\c$ and then attaching fibers given by strong-unstable disks. The dynamics on each such horseshoe can be described conveniently in terms of symbolic dynamics and a discrete-time suspension.  

In what follows, let us assume the minimality of $\cW^\uu$, the existence of a horseshoe of contracting type, and an unstable blender-horseshoe. The minimality of the strong foliation implies that the horseshoe ``connects'' to the unstable blender-horseshoe and \emph{vice versa}. This  provides a (heterodimensional) cycle between the contracting horseshoe and the unstable blender-horseshoe%
\footnote{This means that the invariant manifolds of the horseshoe and the blender-horseshoe intersect cyclically.}.
We identify subcubes which first stay in the domain of the contracting horseshoe for a long time, then transit to the blender-horseshoe, and finally return to the horseshoe. A careful control of those transitions allows to construct new horseshoes of contracting type, with slightly altered center Lyapunov exponent and almost the same entropy as the initial one. This initiates a sequence of horseshoes, each horseshoe cyclically related with the blender. Inductively, this defines a cascade of horseshoes, whose exponents gradually approach some target value, while their entropy is controlled.   

The next step is to show that this cascade of horseshoes can be constructed in such a way that the sequence of spaces of ergodic measures associated with each horseshoe weakly converges to some space of ergodic measures with prescribed center Lyapunov exponent. To achieve this, we use an approach pioneered by \cite{GorIlyKleNal:05} and \cite{Wei:00} and further developed by \cite{KwiLac:}. In \cite[Lemma 2]{GorIlyKleNal:05} they provide a criteria guaranteeing that a sequence of ergodic measures  converges to an ergodic measure. In the symbolic setting, there is the $\bar f$-distance defined among ergodic measures, with the property that a $\bar f$-Cauchy sequence of ergodic measures converges weak$\ast$ and in entropy to an ergodic measure (see, for example, \cite{ORW:82,Shi:96} for details and the summary Section~\ref{secdbardistance}). In  \cite{KwiLac:}, this is generalized by introducing the {\em{Feldman-Katok}} quasidistance $\bar F_\FK$, a generalization of $\bar f$ to topological dynamical systems. In \cite{KwiLac:} it is proven that if a sequence of ergodic measures is $\bar F_\FK$-Cauchy then it converges in weak$\ast$ and in entropy, and the limit measure is ergodic. The convergence defined in \cite{GorIlyKleNal:05} turns out to be a special case of the $\bar F_\FK$-convergence.

To prove our main results, we show the $\bar F_\FK$-Cauchy property for properly defined measures on the cascade of horseshoes, which proves ergodicity of the corresponding limit measures and allows to estimate its entropy. Indeed, it turns out that in the class $\BM^1(M)$ the $\bar F_\FK$-convergence is a natural and very versatile tool. 

\subsection{Organization} In Section \ref{ss.prescribed} we develop the language of $\csu$-cubes and study their properties. In Section \ref{sec:tophorse}, we define $\csu$-horseshoes and prove that there exists some horseshoe which ergodically mimic a given ergodic measure of contracting type and with positive entropy. In Section \ref{sectophorssus}, we recall that discrete-time suspensions conveniently describe ergodic measures on a horseshoe. In Section \ref{ss.twolemmas}, we discuss blender-horseshoes and their dynamics. In Section \ref{s.connecting}, we study the interaction between the contracting-type $\csu$-horseshoe and an expanding blender-horseshoe. This initiates a cascade of inductively defined $\csu$-horseshoes. Section \ref{sedbarFK} is dedicated to the $\bar d$- and Feldman-Katok distances. In Section \ref{sec:existence} we  prove the convergence in the Feldman-Katok distance across the cascade of horseshoes, proving the ergodicity of corresponding limit measures. Finally, in Section \ref{secfinal}, we complete the proofs of Theorems \ref{thm.maingeneral} and \ref{thm.main}, and in Section \ref{secnonerglims}, we prove Theorem \ref{nonerglims}.

\section{$\csu$-cubes and their properties}\label{ss.prescribed} 

Throughout this section, let $f\in\PH^1_{\c=1}(M)$. Besides the partially hyperbolic splitting in \eqref{defsplitt}, we consider the subbundles 
\[
 	E^\cs\eqdef E^\ss\oplus E^\c,\quad
	E^\cu\eqdef E^\c \oplus E^\uu. 
\]
We introduce so-called \emph{$\csu$-cubes} which provide one of the key ingredients for our construction of  horseshoes. Roughly speaking, such a cube is obtained from a disk tangent to $E^\cs$ and saturated by strong unstable disks.

To be a bit more precise, recall that for a partially hyperbolic diffeomorphism $f\in\PH^1_{\c=1}(M)$, by definition, there is $N\in\mathbb{N}$ such that for every $x\in M$,   
\begin{itemize}[leftmargin=0.6cm ]
	\item $\|Df^N|_{E^\ss (x)}\|<1$ and $\|Df^{-N}|_{E^\uu (x)}\|<1$;
	\item $\max\big\{\|Df^N|_{E^\ss (x)}\| \|Df^{-N}|_{E^\c(f^N(x))}\|, \|Df^N|_{E^\c(x)}\| \|Df^{-N}|_{E^\uu (f^N(x))}\|\big\}<1.$
\end{itemize}
Up to changing the metric, one can assume that $N=1$ (see~\cite{Gou:07}) and the bundles $E^\ss$, $E^\c$ and $E^\uu$ are orthogonal to each other, and we do so throughout this paper. Recall that the splitting in \eqref{defsplitt} is continuous. See \cite[Chapter B1]{BonDiaVia:05} for details.

In Section \ref{sec:csucubes}, we define $\csu$-cubes, introduce some special classes of such cubes, and study their interaction by $f$. In Section \ref{sec:csucubesquanti}, we introduce some quantifiers of $\csu$-cubes. In Section \ref{sec:csucubedistortion}, we state some distortion results.

\subsection{$\csu$-cubes: covering relations and iterations}\label{sec:csucubes}

Recall that a \emph{center curve} is a $C^1$ curve tangent to $E^\c$. We call a set $D^\cs$ as a \emph{$\cs$-manifold} if 
\begin{itemize}
	\item $D^\cs$  is a $\dim E^\cs$-dimensional  $C^1$ submanifold  that is everywhere tangent ot $E^\cs$;
	\item $D^\cs$ is the image (under the exponential maps) of a $C^1$ map\footnote{There exist $x_0\in D^\cs$, $r>0$ small and a $C^1$ map $\psi\colon E^\cs(x_0)\to E^\uu(x_0)$ such that $\psi(0)=0$, $D\psi(0)=0$, $\|D\psi\|<1$ and $D^\cs=\exp_{x_0}((v,\psi(v)))$ with $\|v\|\le r.$};
	\item $D^\cs$ is   homeomorphic to a closed ball.
\end{itemize} 
The following remark guarantees, in particular, the existence of $\cs$-manifolds.

Given $r>0$ and $\dagger\in\{\ss, \uu\}$, denote by $\cW^\dagger(x,r)$ the ball centered at $x$ and of radius $r$ (relative to the induced distance on $\cW^\dagger(x)$). 

\begin{remark}[Weak integrability]\label{remweaint}
 By \cite{BriBurIva:04}, the bundles $E^\cu$ and $E^\cs$ are \emph{weakly integrable}, in the sense that for every $x\in M$ there exist two immersed complete $C^1$ submanifolds $\cW_{\rm loc}^\cu(x)$ and $\cW_{\rm loc}^\cs(x)$ which contain $x$ in their interior and are tangent to $E^\cu$ and $E^\cs$, respectively. Moreover, there exists $\varepsilon_1>0$ small such that for every center curve $\gamma$ with length at most $\varepsilon_1$, the sets
\[
	\cW^\uu(\gamma,\varepsilon_1)\eqdef\bigcup_{z\in\gamma}\cW^\uu(z,\varepsilon_1) 
	\quad\textrm{~and~}  \quad
	\cW^\ss(\gamma,\varepsilon_1)\eqdef\bigcup_{z\in\gamma}\cW^\ss(z,\varepsilon_1) 
\]
are $C^1$ submanifolds tangent to $E^\cu$ and $E^\cs$, respectively (\cite[proof of Proposition 3.4]{BriBurIva:04}). 
Indeed, these manifolds are images (under the exponential maps) of  graphs of  $C^1$ maps. 
\end{remark}

\begin{definition}[$\ss$- and $\uu$-disks]
	A \emph{$\uu$-disk} is a closed disk of dimension $\dim E^\uu$ contained some leave of the foliation $\cW^\uu$; analogously for \emph{$\ss$-disk}.
\end{definition}

\begin{definition}[$\csu$-cube, base, and fibers]
An {\em{$\csu$-cube}} (for short, a {\em{cube}}) is a set of the form 
\begin{equation}\label{csucube}
	C = \bigcup_{z\in D^\cs} D^\uu (z),
\end{equation}
where $D^\cs=D^\cs(C)$ is a $\cs$-manifold and $\{D^\uu (z)\colon z\in D^\cs\}$ is a family of pairwise disjoint $\uu$-disks depending continuously on $z$ (the point $z$ may not belong to the disk $D^\uu(z)$ and hence $D^\cs$ may be not contained in $C$). The set $D^\cs$ is the \emph{base} of the cube and
\begin{equation}\label{csucubebis}
	\cD(C)
	\eqdef 
	\big\{D^\uu(z)\colon z\in D^\cs\big\}
\end{equation}
 its family of \emph{fibers}.
 \end{definition}

\begin{remark}[Bases of a cube]\label{r.csucubebis}
A cube can have several bases. A base is not necessarily a subset of the cube. Indeed, there are examples of cubes whose bases are not contained in it, although the cubes considered in our constructions always contain a base. 

Given a cube $C$  with a base $D^\cs$, any $\cs$-manifold $\widetilde D^\cs$ such that the projection from $\widetilde D^\cs$ to $D^\cs$ along the strong unstable foliation is a homeomorphism also provides a base for $C$. On the other hand, if $D^\cs$ and $\widetilde D^\cs$ are bases of the same cube, then they are homeomorphic by the projection along the strong unstable foliation.
 Note that the fibers (considered as sets) do not depend on the ``parametrization'' given by a base.
\end{remark}

\begin{definition}[Subcubes]\label{defsubcubes}
A cube $\widetilde C$ is a \emph{$\csu$-subcube} (or, for short, a \emph{subcube}) of a cube $C$ if there exists a choice of bases $\widetilde D^\cs$ and $D^\cs$ for $\widetilde C$ and $C$, respectively, such that $\widetilde D^\cs\subset D^\cs$ and $\widetilde D^\uu (z)\subset D^\uu (z)$ for every $z\in\widetilde D^\cs$.
We consider two particular types. A subcube $\widetilde C$ is \emph{$\cs$-complete} if $\widetilde D^{\cs}= D^\cs$ and \emph{$\uu$-complete} if $\widetilde D^\uu (z) =D^\uu (z)$ for every $z\in \widetilde D^{\cs}$.

A cube $C_1$ \emph{$\uu$-covers} a cube $C_2$ if there are bases $D^\cs_1$ and $D^\cs_2$ for $C_1$ and $C_2$, respectively, such that there is a connected component $C_{12}$ of $ C_1\cap  C_2$ which is a $\uu$-complete subcube of $C_2$ with base $D^\cs_{12}$ and fibers $\{D^\uu_{12}(z)\colon z\in D^\cs_{12}\}$ satisfying:
\begin{itemize}[leftmargin=0.6cm ]
	\item $D^\cs_{12}\subset\interior D^\cs_2$ and $D^\cs_{12}=D^\cs_1$,
	\item $D^\uu_{12}(z)=D^\uu_2(z)\subset\interior D^\uu_1(z)$ for every $z\in D^\cs_{12}$.
\end{itemize}
\end{definition}

\begin{remark}\label{r.uucompcover}
Consider cubes $C_1, C_2, C_3$ such that
\begin{itemize}[leftmargin=0.6cm ]
\item $C_2$ is a $\uu$-complete subcube of $C_1$,
\item $C_3$ is a $\cs$-complete subcube of $C_1$.
\end{itemize}
Then every $\uu$-complete subcube of
$C_2$ is $\uu$-complete for $C_1$ and $\uu$-covers $C_3$.
\end{remark}

The following is an immediate consequence of partial hyperbolicity, see Remark \ref{remweaint}.
 
\begin{lemma}\label{lemcubeit}
	For every cube $C$, its image   $f(C)$  and pre-image $f^{-1}(C)$ are cubes. Moreover, if $C'$ is a $\cs$-subcube of $C$, then $f(C')$ and $f^{-1}(C')$ are $\cs$-subcubes of $f(C)$  and $f^{-1}(C)$, respectively.
\end{lemma}	
 
The following is an immediate consequence of the above definitions.

\begin{lemma}\label{lr.uucovering}
Given cubes $C_1,C_2$ such that $f(C_1)$ $\uu$-covers $C_2$,  there is a connected component $C_3$ of $C_1\cap f^{-1}(C_2)$ such that 
\begin{itemize}[leftmargin=0.6cm ]
	\item $C_3$ is a $\cs$-complete subcube of $C_1$ 
	\item  $f(C_3)$ is a $\uu$-complete subcube of $C_2$.
\end{itemize}

\end{lemma}

We now introduce a special type of $\csu$-cubes.

\begin{definition}[Center curves and $\csu$-cubes centered at a center curve]\label{defGammas}
Given $x\in M$ and $r>0$, denote by $\gamma^\c_r(x)$ a center curve centered at $x$ and of radius $r$. Denote by $\Gamma^\c_r(x)$ the collection of all such curves. Let $\Gamma^\c_r \eqdef \{ \Gamma^\c_r(x)\colon x\in M\}$. Given $\gamma\in \Gamma^\c_r$, let
\begin{equation}\label{e.cubecenteredatx}
	C^\csu(\gamma, r)\eqdef  \bigcup_{x\in \cW^{\ss}(\gamma,r)} \cW^\uu(x,r),
	\quad\text{ where }\quad
	\cW^{\ss}(\gamma,r)\eqdef \bigcup_{x\in \gamma} \cW^\ss(x,r).
\end{equation}
Note that, if $r\in(0,\varepsilon_1]$ for $\varepsilon_1$ as in Remark \ref{remweaint}, $C^\csu(\gamma, r)$ is indeed an $\csu$-cube with base $D^\cs=\cW^{\ss}(\gamma,r)$ and family of fibers $D^\uu(\cdot)=\cW^\uu(\cdot,r)$, as defined in \eqref{csucube}. In this case, we say that this cube is \emph{centered at $\gamma$ and has size $r$}. Compare Figure \ref{fig.csucube}. 
\begin{figure}[h] 
 \begin{overpic}[scale=.5]{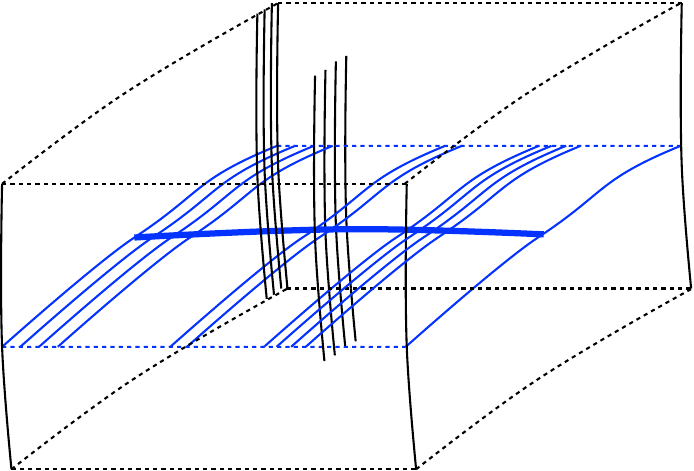}
 	\put(80,32.5){$\gamma$}
	\put(51,56){$\cW^\uu(x,r)$}
 	\put(99,43){\textcolor{blue}{$\cW^{\ss}(\gamma,r)=D^\cs$}}
	\put(65,0){$C^\csu(\gamma, r)$}
\end{overpic}
  \caption{A $\csu$-cube centered at some center curve $\gamma$ of size $r$}
 \label{fig.csucube}
\end{figure}
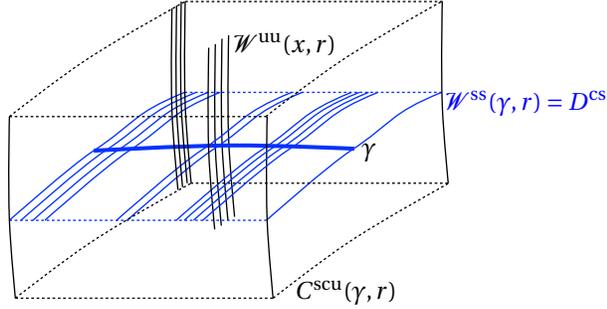
\end{definition}

\subsection{Inner sizes and distances}\label{sec:csucubesquanti}

\begin{remark}\label{remdeftauc}
	Given a cube $C$ with base $D^\cs$, there is $\tau_{\rm c}>0$ such that for every $z\in D^\cs$ the corresponding fiber $D^\uu(z)$ contains some $\uu$-disk $\cW^\uu(x_z,\tau_{\rm c})$ 
	for some $x_z\in C$. We call such number $\tau_{\rm c}$ an \emph{inner $\uu$-radius} of $C$. 
\end{remark}

\begin{definition}[Distance and diameter]\label{d.safedisk-biss}
Given $\tau>0$, two $\uu$-disks $\Delta,\Delta'$ are \emph{$\tau$-close} if
\[
	\dist_{\rm H}(T\Delta,T\Delta')<\tau,
\]
where $\dist_{\rm H}$ denotes the Hausdorff distance on the corresponding Grassmannian manifold, and $T\Delta$ denotes the tangent bundle of $\Delta$. Consider the ``ball''
\[
	B(\Delta,\tau) 
	\eqdef \bigcup\big\{ \Delta'\colon \Delta'\text{ is a $\uu$-disk and } 
	 \dist_{\rm H} (T\Delta, T\Delta') <\tau\big\}.
\]
Given a collection $\cD$ of $\uu$-disks, we define its \emph{diameter}
\[
	\diam_{\rm H}(\cD)
	\eqdef \sup\big\{\tau>0\colon \dist_{\rm H} (T\Delta, T\Delta')<\tau\text{ for every }\Delta,\Delta'\in\cD\big\}.
\]
\end{definition}

\begin{definition}[Cube-safe disks] \label{def.Kisafe}
Given a cube $C$ and $\tau>0$, a $\uu$-disk $\Delta$ is \emph{$\tau$-$C$-safe} if the interior of every $\uu$-disk which is $\tau$-close to $\Delta$ contains a disk in the collection $\cD(C)$ of fibers of $C$. A family of $\uu$-disks is \emph{$\tau$-$C$-safe} if it is open and only contains $\uu$-disks with this property.
\end{definition}

\begin{remark}\label{r.safeness-of-cube}
	For every cube $C$, there are $\tau=\tau(C)>0$ and a $\tau$-$C$-safe family $\cD_{\rm c}$ of $\uu$-disks.
\end{remark}

\begin{remark}\label{r.uniform-cs-bounds-for-csu-cube}
The following is a consequence of the uniform continuity of the center stable bundle. Consider $\varepsilon_1$ from Remark \ref{remweaint}. Given $r\leq \varepsilon_1$,  there exists $\rho(r)>0$ such that for any $\csu$-cube $C$ with $\diam(C)\le r$,   every $\cs$-disc $\widetilde D^\cs\subset C$  has diameter (under its intrinsic metric) bounded from above by $\rho(r)$. Notice that $\rho(r)\leq \rho(r^\prime)$ for $0<r\leq r^\prime\leq \varepsilon_1.$
\end{remark}

The assertions in the following remark are consequences of the uniform continuity and transversality of the bundles of the partially hyperbolic splitting. 

\begin{remark}[``Local product structure'']\label{r.zeta}
For every $r\in(0,\varepsilon_1]$, with $\varepsilon_1$ as in Remark \ref{remweaint}, there is $\delta(r)>0$ such that for every $\delta\in(0,\delta(r)]$ there is $\zeta=\zeta(\delta,r)\in(0,r)$ such that the following holds. For every $x$ and $z$  satisfying $d(x,z)<\zeta$, 
\[
	z'\eqdef 	\cW^\ss(\gamma^\c_r(x),r)\pitchfork\cW^\uu(z,r)
	\quad\text{is well defined and $d(z',x)<\delta.$ }	
\]
Moreover, 
\[
	\cW^\uu(z,r)\subset\cW^\uu(z',2r),\quad
	\cW^\uu(z',r)\subset\cW^\uu(z,2r).
\]
Furthermore, for every $x$ and $y$ satisfying $d(y,x)<\zeta$, the set
\[
	\widetilde D^\cs
	\eqdef
	\cW^\ss(\gamma^\c_r(y),r)\cap C^\csu(\gamma^\c_\zeta(x),\zeta)
\]
is a base for the cube $C^\csu(\gamma^\c_\zeta(x),\zeta)$.
\end{remark}

\subsection{Cubes and distortion}\label{sec:csucubedistortion}

Recall notation $D^\c f$ in \eqref{defchic} and that $x\mapsto\lVert D^\c f(x)\rVert$ is a continuous function. Given $\varepsilon>0$, consider the following ``center-modulus of continuity''
\[
	\Mod_f(\varepsilon)
	\eqdef \max\Big\{\big|\log\|D^\c f(x)\|-\log\|D^\c f(y)\|\big|\colon d(x,y)\le\varepsilon\Big\}.
\]
Note that, by uniform continuity of $\log\|D^\c f\|$, $\Mod_f(\varepsilon)\to0$ as $\varepsilon\to0$.

\begin{lemma}[Center distortion]\label{l.centerdistortion}
For every $\chi<0$ and $\varepsilon_{\rm L} \in (0, |\chi|/2)$ there is $\varepsilon>0$ such that the following holds. Suppose that for $x\in M$ there are some $K_{\rm L}\in\bN$ and $m\in\bN$ such that for every $n=1,\ldots, m$,
	\[
	\lvert D^\c f^n(x)\rvert
	\le K_{\rm L} e^{n\chi} .
	\]
	Then for every center curve $\gamma$ centered at $x$ of length bounded by $\varepsilon/K_{\rm L}$, 
	it holds
	\[
	\length(f^n(\gamma)) \leq K_{\rm L}\cdot e^{n(\chi+\varepsilon_{\rm L})}\length(\gamma)
	\quad\text{for every $n=0,\ldots, m$.}
	\]
\end{lemma}

\begin{proof}
Given $\chi<0$ and $\varepsilon_{\rm L} \in (0, |\chi|/2)$, choose $\varepsilon>0$ such that $\Mod_f(\varepsilon)<\varepsilon_{\rm L}$. Let $x$,  $m$, and $\gamma$ be as in the assumption. The proof is by induction. First note that
\[
	\length(f(\gamma))
	< \|D^\c f(x)\|\cdot e^{\varepsilon_{\rm L}}\cdot \length(\gamma)
	\le K_{\rm L} e^{\chi+\varepsilon_{\rm L}}\cdot \varepsilon/K_{\rm L}
	<  \varepsilon. 
\]
Assume now that for some $0\leq i\leq m-1$ and  for every $j\leq i$, one has  
\[
	\length(f^j(\gamma))\leq K_{\rm L} e^{j(\chi+\varepsilon_{\rm L})}\length(\gamma)
	<K_{\rm L}e^{j(\chi+\varepsilon_{\rm L})}\frac{\varepsilon}{K_{\rm L}}
	<\varepsilon.
\]	 
Using $\Mod_f(\varepsilon)<\varepsilon_{\rm L}$ again, we get that the length of $f^{i+1}(\gamma)$ is bounded  by 
\[ 
	\mathrm{length}(f^{i+1}(\gamma))
	\le \|D^\c f^{i+1}(x)\|\cdot e^{(i+1)\varepsilon_{\rm L}}\cdot\length(\gamma)
	\leq K_{\rm L} e^{(i+1)(\chi+\varepsilon_{\rm L})}\cdot\length(\gamma)
,\]
	proving the lemma.
\end{proof}

\begin{corollary}\label{c.distortion-reference-cube}
Let $\varepsilon_1>0$ be as in Remark~\ref{remweaint}. For every  $K\geq 1$, $\chi<0$, and $\varepsilon>0$,  there exists $L=L(K,\chi,\varepsilon)\in\bN$ with the following properties. 	Let $C$ be an $\csu$-cube having the following properties:
\begin{itemize}[leftmargin=0.6cm ]
\item[(1)] there exists $n\in\bN$ such that   every $x\in C$ and every $k=1,\ldots,n$,
\[
	\lVert D^\c f^k(x)\rVert 
	\le Ke^{k\chi};
\]	
\item[(2)] ${\rm max}\big\{\diam(C), \diam(f^n(C))\big\}\leq \varepsilon_1$.
\item[(3)] there is a base $D^{\cs}$ of $C$ contained in $C$.
\end{itemize} 
Then 
\[
	\diam (f^k(C)) < \varepsilon
	\quad\text{ for every }\quad
	k=L,\ldots,n-L.
\]
\end{corollary}

\proof  Given $\varepsilon>0$, let $L_1=L_1(\varepsilon)\in\bN$ be the smallest integer such that 
\begin{equation}\label{choicesL1}
	-L_1\inf_{x\in M}\log\|Df|_{E^\uu(x)}\| +\log\varepsilon_1
	<\log(\varepsilon/3).
\end{equation}	
By Remark \ref{r.uniform-cs-bounds-for-csu-cube}, one has $\diam(D^\cs)\le\rho(\varepsilon_1).$
As $\chi<0$, there exists $L_2=L_2(K,\chi,\varepsilon)\in\bN$ such that for every $k\geq L_2$, one has 
\[ 
	K\cdot e^{k\chi}
	< \frac{\varepsilon}{3\rho(\varepsilon_1)}.
\]  
By item (1) together with the domination between $E^\ss$ and $E^\c$,  for $k\geq L_2$, one has 
\begin{equation}\label{e.byhypotheses}
	\diam(f^k(D^\cs))
	\leq \max_{x\in D^\cs}\|D^\c f^k(x)\|\cdot \diam(D^\cs)
	<\varepsilon/3.
\end{equation}
 Let $L\eqdef\max\{L_1,L_2\}$. 

Let $C$ be an $\csu$-cube and $D^\cs\subset C$ be a base of $C$ as in the assumption. Note that $f^n(D^\cs)$ is also a base of $f^n(C)$. As, by item (2), $\diam(f^n(C))\leq \varepsilon_1$, for every  $x\in C$, there exists $z\in D^\cs$ such that $f^n(x)\in\cW^\uu(f^n(z),\varepsilon_1)$. The choice \eqref{choicesL1} implies  
\begin{equation}\label{e.desdes}
	d(f^k(x),f^k(z))\le \varepsilon \quad \mbox{for every $k=0,\ldots, n-L$.}
\end{equation}
Take any pair of points $x_1,x_2\in C$ and consider the corresponding points $z_1,z_2\in D^\cs$ as above. 
Since $\diam(f^k(D^\cs))<\varepsilon/3$ for $k=L,\ldots, n$ due to  Equation \eqref{e.byhypotheses},   we get \[d(f^k(z_1),f^k(z_2))<\varepsilon/3\quad \textrm{ for $k=L,\ldots, n$.}\] 
Now, for $k=L,\ldots, n-L$, from \eqref{e.desdes} one has 
\[
d(f^k(x_1),f^k(x_2))
\leq d(f^k(x_1),f^k(z_1))+d(f^k(z_1),f^k(z_2))+d(f^k(z_2),f^k(x_2))
<\varepsilon.
\]
This proves the corollary.
\endproof 

\begin{remark}
Let us briefly explain how Corollary \ref{c.distortion-reference-cube} is going to be implemented. The assertion here is that for any $\csu$-cube with a long block of trajectory presenting a uniformly hyperbolic dynamics,  the size of the cube becomes very small after a few iterations and stays small almost to the very end of the block. The number of iterations  for which the size is large   is uniformly small, it does not depend on the length of the block, only on the hyperbolicity properties (in particular, Lyapunov exponents).

In the following sections we are going to construct horseshoe-like invariant hyperbolic sets, the domains of which are going to  live in certain $\csu$-cubes. Our construction will allow us to construct those horseshoes with arbitrarily large return times, while having uniform hyperbolicity. Applying Corollary \ref{c.distortion-reference-cube}, we get that the trajectories inside one such $\csu$-cube  
 stay together all the time except for a small number of iterations, and this  number  can be arbitrarily small compared with the return time of the horseshoe.  
\end{remark}

\section{$\csu$-horseshoes}\label{sec:tophorse}

In this section, given $f\in \PH_{\c=1}^1 (M)$ and $\nu\in \cM_{\rm erg,<0}(f)$ (recall notation \eqref{def:Msubspaces}), we construct a ``horseshoe'' $\Lambda$ such that there is $n\in\bN$ so that $f^n|_\Lambda$ is topologically conjugate to the full shift in $N$ symbols, where $\log N\approx nh_\nu(f)$, and the center Lyapunov exponents on $\Lambda$ are close to $\chi^\c(\nu)$; see Theorem \ref{t.smarthorseshoes}. For that, we follow the ideas in \cite{Kat:80} and put the results \cite{Cro:11,Gel:16} into the framework of $\csu$-cubes. To make this more precise, let us introduce some notations.

\begin{definition}\label{defhorseshoe}
Given an $\csu$-cube $C$, a triple $\bfH=(C,\{\bfK_i\}_{i=1}^N,f^n)$ is a \emph{$\csu$-horseshoe (relative to $C$)} if the following holds:
\begin{itemize}[leftmargin=0.6cm ]
\item $\bfK_1, \ldots , \bfK_N$ are pairwise disjoint connected components of $C\cap f^{-n} (C)$;
\item each $\bfK_i$, $i=1,\ldots,N$, is a $\cs$-complete subcube of $C$;
\item   each  $f^n(\bfK_i)$ , $i=1,\ldots,N$, is a $\uu$-complete subcube of $C$.
\end{itemize} 	
We call $\bfK_1,\ldots,\bfK_N$ the \emph{rectangles} and $C$ the \emph{reference cube}  of $\bfH$, respectively. Associated to $\bfH$, we define the set
\begin{equation}\label{defLambdaH}
	\Lambda_\bfH \eqdef \bigcap_{k\in\bZ}f^{nk}(\bfK_1\cup\ldots\cup\bfK_N).
\end{equation}
By a slight abuse of notation, we also call $\Lambda_\bfH$ a \emph{$\csu$-horseshoe}. Note that $\Lambda_\bfH$ is $f^n$-invariant.
\end{definition}

Note that  in Definition \ref{defhorseshoe} the set $C\cap f^{-n}(C)$ can, in general, have more connected components.

\begin{definition}\label{defhorseshoecontype}
A $\csu$-horseshoe $\bfH=(C,\{\bfK_i\}_{i=1}^N,f^n)$ is of \emph{contracting type} if there exist $\chi<0$ and $\varepsilon\in(0,|\chi|)$ such that for every $i=1,\ldots,N$ and every $x\in\bfK_i$, 
\[
	e^{n(\chi-\varepsilon)}< \|D^\c f^n(x)\|<e^{n(\chi+\varepsilon)}.
\]
\end{definition}

\begin{theorem}\label{t.smarthorseshoes}
Let $f\in \PH_{\c=1}^1 (M)$ and $\nu\in \cM_{\rm erg,<0}(f)$ be a measure with positive entropy.
Let $\varphi_1,\ldots,\varphi_\ell\colon M\to\bR$ be a collection of continuous functions. For every $\varepsilon>0$ and $r>0$ sufficiently small, there are  an integer $n\in \bN$ and a center curve $\gamma$, such that there exist 
\[
	N>\exp(n(h_\nu(f)-\varepsilon))
\]	 
and a $\csu$-horseshoe $\bfH=(C,\{\bfK_i\}_{i=1}^N,f^n)$ relative to the $\csu$-cube $C=C^\csu(\gamma, r)$ such that 
\begin{itemize}[leftmargin=0.6cm ]
\item[(1)] $\bfH$ is of contracting type: for every $i=1,\ldots,N$ and every $x\in\bfK_i$, 
\[
	e^{n(\chi^\c(\nu)-\varepsilon)}< \|D^\c f^n(x)\|<e^{n(\chi^\c(\nu)+\varepsilon)},
\]
\item[(2)] $\Lambda_\bfH$ is a hyperbolic basic set of $f^n$ and  $f^n|_{\Lambda_\bfH}$ is conjugate to the full shift on $N$ symbols.
\end{itemize} 
Moreover, there is $K>1$ such that for every $x\in \Lambda_\bfH$ and $k\in\bN$,
\begin{equation}\label{e.lastestimate}
	K^{-1}e^{k (\chi^\c(\nu) - \varepsilon)} 
	< \rVert D^\c f^k(x)\lVert
	< Ke^{k (\chi^\c(\nu) + \varepsilon)} 
\end{equation}
and for every $i=1,\ldots,\ell$ and $\mu'\in\cM(f|_{\Lambda_\bfH})$,
\begin{equation}\label{e.lastestimatephis}
	\Big| \int\varphi_i\,d\mu'-\int\varphi_i\,d\mu\Big|
	<\varepsilon .
\end{equation}		
\end{theorem}

We postpone the proof of Theorem~\ref{t.smarthorseshoes} to Section \ref{ss.applyingKatok}. In Section \ref{secKatok}, we prove the existence of a ``skeleton'' of orbit segments which ergodically approximate the measure $\nu$. Finally, in Section \ref{secHorseshoefurther}, we define cylinders of higher order and explore their properties.

\subsection{Approximating and generating basic sets}\label{secKatok}

Given a hyperbolic ergodic measure of contracting type $\nu\in\cM_{\rm erg,<0}(f)$, in Proposition \ref{proKatok} below, we construct ``skeletons'' which ``approximate'' this measure in entropy and in terms of finite-time Lyapunov exponents. Our proof follows the classical ideas in \cite[Supplement S]{KatHas:95}. 

In what follows, we consider \emph{Bowen balls}, defined by
\[
	B_n(x,\varepsilon_0)
	\eqdef \bigcap_{k=0}^{n-1}f^{-k}\big(B(f^k(x)),\varepsilon_0\big),
	\quad\text{ where }\quad
	B(y,\varepsilon_0)
	\eqdef \{z\in M\colon d(z,y)<\varepsilon_0\}.
\]

\begin{remark}[Entropy expansiveness]\label{remexpansive}
By \cite{LiViYa:13,DiFiPaVi:12}, every $f\in\PH_{\c=1}^1(M)$ is \emph{$\varepsilon_0$-entropy expansive} for some $\varepsilon_0>0$, that is, for every $x\in M$,%
\[\begin{split}
	h_{\rm top}(f,B_\infty(x,\varepsilon_0))=0, 
	\quad\text{ where }\quad
	B_\infty(x,\varepsilon_0)
	&\eqdef\bigcap_{n\in\bN}
	B_n(x,\varepsilon_0).
\end{split}\]
\end{remark}
Given $n\in\mathbb{N}$ and $\varepsilon>0$, a subset $E\subset M$ is called an \emph{$(n,\varepsilon)$-separated set} if for any two different points $x,y\in E$, there exists $i\in\{0,\ldots,n-1 \}$ such that $d(f^i(x),f^i(y))>\varepsilon$.

\begin{proposition}[Approximating ``skeleton'']\label{proKatok}
Let $f\in\PH_{\c=1}^1(M)$. Let $\varphi_1,\ldots,\varphi_\ell\colon M\to\bR$ be a collection of continuous functions. Let $\varepsilon_0$ be as in Remark \ref{remexpansive}. For every $\nu\in\cM_{\rm erg,<0}(f)$ with positive entropy $h_\nu(f)$ and every $\delta>0$, there exist $K_{\rm L},K_{\rm B}>0$ such that for any $\rho>0$, there exist $n_0\in\bN$ arbitrarily large and an $(n_0,\varepsilon_0)$-separated set $E$ such that 
\begin{enumerate}[leftmargin=0.6cm ]
\item[(i)] $\card E> e^{n_0(h_\nu(f)-\delta)}$;
\item[(ii)] $\diam(E\cup f^{n_0}(E))\leq \rho;$
\item[(iii)] for every $x\in E$ and every $n\in\mathbb{N}$, one has 
		\[K_{\rm L}^{-1}\cdot e^{n(\chi^\c(\nu)-\delta)}\leq \|D^\c f^n(x)\|\leq K_{\rm L}\cdot e^{n(\chi^\c(\nu)+\delta)}.\]
\item[(iv)] for every $x\in E$, every $n\in\mathbb{N}$, and every $i=1,\ldots,\ell$, one has 
		\[-K_{\rm B}+n\big(\int\varphi_i\,d\mu-\delta\big)
		\le \sum_{k=0}^{n-1}\varphi_i(f^k(x))
		\le K_{\rm B}+n\big(\int\varphi_i\,d\mu+\delta\big)
		.\] 		
\end{enumerate}
\end{proposition}

\begin{proof}
Let $\nu\in\cM_{\rm erg,<0}(f)$ and $\delta\in\big(0,\min\{h_\nu(f),|\chi^\c(\nu)|\}\big)$. We start by selecting sets $\Gamma_{\rm E}$ and $\Gamma_{\rm L}$ with convenient properties.

By Brin-Katok's theorem \cite{BriKat:83} together with Egorov's theorem,  there are a set $\Gamma_{\rm E}\subset M$ and a number $n_{\rm E}\in\bN$ such that $\nu(\Gamma_{\rm E})>3/4$ and for every $x\in\Gamma_{\rm E}$ and $n\ge n_{\rm E}$,
\begin{equation}\label{eqBriKat}
	e^{-n(h_\nu(f)+\delta/2)}
	\le 
	\nu(B_n(x,\varepsilon_0))
	\le
	 e^{-n(h_\nu(f)-\delta/2)}.
\end{equation}

By Birkhoff's ergodic theorem together with Egorov's theorem, there are a set $\Gamma_{\rm L}\subset M$ and a number $n_{\rm L}\in\bN$ such that $\nu(\Gamma_{\rm L})>3/4$ and for every $x\in\Gamma_{\rm L}$ and $n\ge n_{\rm L}$,
\[
	e^{n(\chi^\c(\nu)-\delta)}
	\le
	 \lVert D^\c f^n(x)\rVert
	\le 
	e^{n(\chi^\c(\nu)+\delta)}.
\]
Hence, there is a constant  $K_{\rm L}>1$ such that for all $n\in\bN$ and $x\in\Gamma_{L}$
\begin{equation} \label{eqn:dist}
	K_{\rm L}^{-1}e^{n(\chi^\c(\nu)-\delta)}
	\le
	 \lVert D^\c f^n(x)\rVert
	\le
	 K_{\rm L}e^{n(\chi^\c(\nu)+\delta)}.
\end{equation}
Analogously, applying Birkhoff's ergodic theorem to the continuous functions $\varphi_1,\ldots,\varphi_\ell$, respectively, one gets a set $\Gamma_{\rm B}\subset M$, a number $n_{\rm B}\in\bN$, and a constant $K_{\rm B}>1$ such that $\nu(\Gamma_{\rm B})>3/4$ and for every $x\in\Gamma_{\rm B}$ and $n\ge n_{\rm B}$,
\begin{equation} \label{eqn:distbirk}
	-K_{\rm B}+n\big(\int\varphi_i\,d\mu-\delta\big)
		\le \sum_{k=0}^{n-1}\varphi_i(f^k(x))
		\le K_{\rm B}+n\big(\int\varphi_i\,d\mu+\delta\big).
\end{equation}
Note that 
\[
	\nu(\Gamma_{\rm E}\cap\Gamma_{\rm L}\cap\Gamma_{\rm B})
	>\frac14.
\]

Take now any $\rho>0$. Take $A\subset M$ of diameter less than $\rho$ such that $B\eqdef A\cap \Gamma_{\rm E}\cap \Gamma_{\rm L}$ satisfies $\nu(B)>0$. Note that the ergodicity of $\nu$ implies that 
\[
	\lim_{n\to\infty}\frac1n\sum_{k=0}^n\nu(f^{k}(B)\cap B)=\nu(B)^2.
\] 
The following fact is a straightforward consequence of the above equality.

\begin{claim}\label{clergo}
There exists an increasing sequence $(k_i)_i\subset\bN$ so that for every $i\in\bN$, 
\[\nu(f^{k_i}(B)\cap B)>\nu(B)^2(1-\delta).\] 
\end{claim}

Fix some $i\in\bN$ sufficiently large, such that
\begin{equation}\label{eqchoini}
	n_{\rm E}<k_i
	\quad \textrm{~and~}	\quad
	 \frac{1}{k_i}\big|\log\big(\nu(B)^2(1-\delta)\big)\big|< \frac \delta2.
\end{equation}
Take $n_0\eqdef k_i$. Let 
\[
	B'
	\eqdef
	 f^{-n_0}\big(f^{n_0}(B)\cap B\big).
\]
Claim \ref{clergo} together with $f$-invariance of $\nu$ implies
\begin{equation}\label{eqlargemeasu}
	\nu(B')>\nu(B)^2(1-\delta).
\end{equation}

We now construct the set $E$ of points which are $(n_0,\varepsilon_0)$-separated, as claimed in the proposition. Let $N$ be the smallest integer such that
\begin{equation}\label{eqchoiceN}
	N
	\ge \nu(B)^2(1-\delta)e^{n_0(h_\nu(f)-\delta/2)}
	> e^{n_0(h_\nu(f)-\delta)},
\end{equation}
where the second inequality follows from \eqref{eqchoini}. Let $x_1\in B'$. From  \eqref{eqlargemeasu}, \eqref{eqBriKat}, and \eqref{eqchoini} we get
\[
	\nu\big(B'\setminus B_{n_0}(x_1,\varepsilon_0)\big)
	\ge 
	\nu(B') - \nu(B_{n_0}(x_1,\varepsilon_0))
	> 
	\nu(B)^2(1-\delta) - e^{-n_0(h_\nu(f)-\delta/2)}
	>
	 0. 
\]
Hence, $B'\setminus B_{n_0}(x_1,\varepsilon_{0})$ has positive measure and we can proceed by (finite) induction on $n$ to choose further points in $B'$. Assume that $n\in\{1,\ldots,N-1\}$ is such that the points $x_1,\ldots,x_n$ are already chosen. Note that 
\[\begin{split}
	\nu\Big(B'\setminus\bigcup_{k=1}^{n}B_{n_0}(x_k,\varepsilon_0)\Big)
	&\ge
	 \nu(B') - n\max_{k=1,\ldots,n}\nu(B_{n_0}(x_k,\varepsilon_0))
	 \\
	\text{\tiny{by \eqref{eqlargemeasu} and \eqref{eqBriKat}}}\quad
	&>
	\nu(B)^2(1-\delta)-ne^{-n_0(h_\nu(f)-\delta/2)}\\
	\text{\tiny{by \eqref{eqchoiceN}}}\quad
	&>0.
\end{split}\]
Hence, we can pick some point 
\[
	x_{n+1}
	\in B'\setminus\bigcup_{k=1}^{n}B_{n_0}(x_k,\varepsilon_0).
\]
Note that, by construction, each pair of points $x_{n+1}$, $x_k$, for $k=1,\ldots,n$, is $(n_0,\varepsilon_0)$-se\-pa\-rated. This finishes the induction step, which can be repeated $N$ times, and defines the set of points $E\eqdef\{x_1,\ldots,x_N\}$. Together with \eqref{eqchoiceN}, it follows that $E$ satisfies item (i).

By construction, we have $E\cup f^{n_0}(E)\subset A$, proving item (ii). As $E\subset B'\subset \Gamma_{\rm L}\cap\Gamma_{\rm B}$, items (iii)  and (iv) follow from \eqref{eqn:dist} and \eqref{eqn:distbirk}, respectively. This proves Proposition \ref{proKatok}.
\end{proof}

\subsection{Proof of Theorem~\ref{t.smarthorseshoes}}\label{ss.applyingKatok}

We first fix some auxiliary numbers in Section \ref{ssec:quantif}. In Section \ref{ssec:rectangles} we construct the horseshoe rectangles and finish the proof in Section \ref{ssec:finish}. 

\subsubsection{Choice of quantifiers}\label{ssec:quantif}
Let 
\begin{equation}\label{deflambdauu}
	\tau^\ss
	= \tau^\ss(f)
	\eqdef \max_{x\in M}\log\|Df|_{E^\ss(x)}\|
	<0<
	\tau^\uu
	= \tau^\uu(f)
	\eqdef \min_{x\in M}\log\|Df|_{E^\uu(x)}\|.
\end{equation}
Let $\varepsilon_0$ and $\varepsilon_1$ be as in Remarks \ref{remexpansive} and \ref{remweaint}, respectively. Fix
\[
 	\varepsilon_{\rm L}\in\big(0,\min\big\{h_\nu(f),|\chi^\c(\nu)|/2\big\}\big).
\]
Apply Lemma \ref{l.centerdistortion} to $\chi^\c(\nu)+\varepsilon_{\rm L}/4$ and $\varepsilon_{\rm L}/4$,  and consider the corresponding $\varepsilon$. 

Apply Proposition~\ref{proKatok} to $\varepsilon_0$, the measure $\nu$, and $\varepsilon_{\rm L}/4$ and consider the corresponding constants $K_{\rm L}\ge1$ and $K_{\rm B}\ge1$. 

In what follows, we shrink $\varepsilon>0$ and fix $r>0$ sufficiently small such that
\begin{equation}\label{e.threeconditions1}\begin{split}
	\varepsilon
	&<\frac{\varepsilon_0}{4},\\
	\Mod_f(6r+\varepsilon)
	&< \frac{\varepsilon_{\rm L}}{4},\\
	r 
        &<\min\Big\{\frac{\varepsilon_1}{2}, \frac{\varepsilon_0}{8},\frac{\varepsilon}{4K_{\rm L}}\Big\}.
\end{split}\end{equation}

Recall the sets of $\Gamma^\c_r(x)$ and $C^\csu(\gamma,r)$ in Definition \ref{defGammas}. 

\begin{remark}[Choice of $\zeta$]\label{remquantifier}
By the uniform continuity of the bundle $E^\cs$ and compactness of $M$, there exists $\zeta>0$ such that 
\begin{equation}\label{eqchoicezeta}
	B(x,4\zeta)\subset C^\csu(\gamma,r)
	\quad\text{ for every $	x\in M$, $\gamma\in \Gamma^\c_r(x)$. }
\end{equation}
Moreover, for every  point  $y\in M$ with $d(x,y)<\zeta$ and every $\gamma_1\in\Gamma^\c_{2r}(y)$, every fiber of $C^\csu(\gamma,r)$ has nonempty transverse intersection with the interior of $\cW^\ss(\gamma_1,2r)$.
\end{remark}

Fix now 
\begin{equation}\label{e.choice-of-rho}
	\rho\in\big(0, \frac{1}{2}\Leb (\zeta) \big),
\end{equation}
where $\Leb(\zeta)$ is the Lebesgue number of the open covering $\{ B(x,\zeta)\colon x\in M\}$.

\begin{remark}[Choices of $n$ and $E$]\label{rem-n-E}
Given $\rho$, by Proposition~\ref{proKatok}, there exist $n\in\bN$ (which can be chosen arbitrarily large) and   an $(n, \varepsilon_0)$-separated set $E=\{x_1, \ldots , x_N\}$ such that: 
\begin{enumerate}[leftmargin=0.6cm ]
	\item[(i)] $\card E\geq e^{n(h_\nu(f)-\varepsilon_{\rm L}/4)}$;
	\item[(ii)] for every $x\in E$ and every $m\in\mathbb{N}$, one has 
	\begin{equation}\label{itemii}
		K_{\rm L}^{-1}\cdot e^{m(\chi^\c(\nu)-\varepsilon_{\rm L}/4)}
		\leq \|D^\c f^m(x)\|\leq K_{\rm L}\cdot e^{m(\chi^\c(\nu)+\varepsilon_{\rm L}/4)} 
	\end{equation}	
	and for every $x\in E$, every $m\in\mathbb{N}$, and every $i=1,\ldots,\ell$, one has
	\[
	-K_{\rm B}+n\big(\int\varphi_i\,d\mu-\varepsilon_{\rm L}/4\big)
		\le \sum_{k=0}^{n-1}\varphi_i(f^k(x))
		\le K_{\rm B}+n\big(\int\varphi_i\,d\mu+\varepsilon_{\rm L}/4\big);	
	\]
	\item[(iii)] $\diam(E\cup f^{n}(E))\leq \rho.$
\end{enumerate}
Without loss of generality, we can assume that $n$ satisfies
\begin{equation}\label{e.largen}
	K_{\rm L} \le e^{n\varepsilon_{\rm L}/2},\quad\quad 
	\varepsilon\cdot e^{n(\chi^\c(\nu)+\varepsilon_{\rm L})}+4r\cdot e^{n\tau^\ss}<\zeta.
\end{equation}
\end{remark}

This finishes the choice of auxiliary quantifiers.

\subsubsection{Construction of the rectangles of the horseshoe.}\label{ssec:rectangles}

By property (iii) in Remark \ref{rem-n-E} and our choice of $\rho$ in \eqref{e.choice-of-rho}, there  exists $x_0\in M$ such that 
\begin{equation}\label{eqpropE}
	E\cup f^n(E)\subset B(x_0,\zeta).
\end{equation}
Fix a center curve $\gamma$ centered at $x_0$ and of radius $r$, and consider the corresponding cube $C\eqdef C^\csu(\gamma,r)$ with base $D^\cs\eqdef\cW^\ss(\gamma,r)$ as in \eqref{e.cubecenteredatx}. This will provide the cube in the assertion of the theorem. It remains to check its properties. 
 
Given $E=\{x_1,\ldots,x_N\}$ chosen above, for $ i=1, \ldots, N$ let
\[
	\bfK_i \eqdef
	\text{ the connected component of }f^{-n}(C) \cap C\text{ which contains the point }x_i.
\]
Note that, by our choice of $\zeta$ in \eqref{eqchoicezeta} together with \eqref{eqpropE}, it holds $x_i,f^n(x_i)\in C$ for every  $ i=1, \ldots, N$ and hence $\bfK_i$ is nonempty. 

Let $\bfH\eqdef(C,\{\bfK_i\}_{i=1}^N,f^n)$. Claims \ref{c.completeness-of-each-Ki} and \ref{lemcl.disjoint} below imply that $\bfH$ indeed satisfies all properties of a $\csu$-horseshoe (relative to $C$), recall Definition \ref{defhorseshoe}. Note that item (i) in Remark \ref{rem-n-E} gives already
\[
	N
	=\card E\geq e^{n(h_\nu(f)-\varepsilon_{\rm L}/4)}
	\geq e^{n(h_\nu(f)-\varepsilon_{\rm L})}.
\]

\begin{claim}\label{c.completeness-of-each-Ki}
For every $i=1,\ldots, N$, the set $\bfK_i$ is a $\cs$-complete subcube of $C$ and $f^n(\bfK_i)$ is a $\uu$-complete subcube of $C$.
\end{claim}

\proof 
Fix $x_i$ and consider a center curve $\gamma_i\in\Gamma^\c_{2r}(x_i)$. By the choice of $\zeta$ in Remark \ref{remquantifier}, we get that 
\[
	D^\cs_i\eqdef\cW^\ss(\gamma_i,2r)\cap C
\]	  
intersects every fiber of the cube $C$. In fact, $D^\cs_i$ is the image of $D^\cs$ under the strong unstable holonomies between $D^\cs$ and $\cW^\ss(\gamma_i,2r)$. Therefore $D^\cs_i$ is also a base of $C$. 

By  \eqref{itemii}, for every $k\in\bN$, one has 
\[ 
	\|D^\c f^k(x_i)\|\leq K_{\rm L}\cdot e^{k(\chi^\c(\nu)+\varepsilon_{\rm L}/4)}.
\]
From \eqref{e.threeconditions1}, we get $\length(\gamma_i)= 4r<\varepsilon/K_{\rm L}$. By Lemma \ref{l.centerdistortion} applied to  $\chi^\c(\nu)+\varepsilon_{\rm L}/4$ and $\varepsilon_{\rm L}/4$, for every $k\in\bN$ one has 
\[
	\length(f^k(\gamma_i))
	\leq K_{\rm L}\cdot e^{k(\chi^\c(\nu)+\varepsilon_{\rm L}/2)}\cdot \length(\gamma_i)
	< \varepsilon\cdot e^{k(\chi^\c(\nu)+\varepsilon_{\rm L}/2)}.
\]
As $D^\cs_i$ is tangent to $E^\cs$, by the domination between $E^\ss$ and $E^\c$, for each $k\in\bN$,
\begin{equation}\label{eq:diameter-of-base-Dcsi}
	\diam(f^k(D^\cs_i))
	<\varepsilon\cdot e^{k(\chi^\c(\nu)+\varepsilon_{\rm L}/2)}+4r\cdot e^{k\tau^\ss}.
\end{equation}
Recalling that $f^n(x_i)\in B(x_0,\zeta)$, and using \eqref{eq:diameter-of-base-Dcsi},  \eqref{e.largen} and \eqref{eqchoicezeta}, one has 
\[
	\diam(f^n(D^\cs_i))
	\subset B(x_0,2\zeta)\subset B(x_0,4\zeta)\subset C.
\]
Since $f$ is uniformly expanding along the strong unstable bundle, one deduces that the connected component $f^n(\bfK_i)$ of $f^n(C)\cap C$ containing $f^n(x_i)$ is a $\uu$-complete subcube of $C$ and $f^n(D^\cs_i)\subset f^n(\bfK_i)$ is its base. Once again using the uniform expansion of $f$ along the strong unstable bundle, $\bfK_i$ is a cube with base $D^\cs_i$, and thus is a $\cs$-complete subcube of $C$. This ends the proof  of Claim~\ref{c.completeness-of-each-Ki}.
\endproof 

\begin{claim}\label{lemcl.disjoint}
The sets $\bfK_i$, $i=1,\ldots,N$, are pairwise disjoint. Moreover,  
\begin{equation}\label{eqformula}
	d(f^k(x),f^k(x_i))<6r+\varepsilon
	\quad\text{ for every }x\in \bfK_i\text{ and }k=0,\ldots,n.
\end{equation}
\end{claim}

\begin{proof}
By Claim \ref{c.completeness-of-each-Ki}, each $f^n(\bfK_i)$ is $\uu$-complete subcube of $C$. Hence, for every $x\in\bfK_i\subset C$ the set $\cW^\uu(f^n(x),2r)$ intersects $f^n(D^\cs_i)$ in a unique point $w_n$. By the uniform  contraction of $f^{-1}$ along the  bundle $E^\uu$,  for each $0\le k\leq n$,
 \[\begin{split}
 	d(f^{-k}(f^n(x)),f^{-k}(f^n(x_i)))
 	&\le  d(f^{-k}(f^n(x)),f^{-k}(w_n))+ d(f^{-k}(w_n),f^{-k}(f^n(x_i)))\\
	\text{\tiny by \eqref{deflambdauu} and \eqref{eq:diameter-of-base-Dcsi}}\quad
 	&< 2r\cdot e^{-k\tau^\uu}+ \varepsilon\cdot e^{(n-k)(\chi^\c(\nu)+\varepsilon_{\rm L})}+4r\cdot e^{(n-k)\tau^\ss}\\
	\text{\tiny by  \eqref{e.threeconditions1}}\quad
	&< 6r+\varepsilon <\varepsilon_0.
\end{split}\]
In particular, this proves  \eqref{eqformula}.
 
The proof of disjointness is now by contradiction. Suppose that $\bfK_i\cap \bfK_j\neq\emptyset$ for some $i\neq j$. Hence, we get $\bfK_i=\bfK_j$ since both sets are connected components of $f^{-n}(C)\cap C$. This implies that $f^n(x_j)\in f^n(\bfK_i)$. By the above, we get that for each $0\le k\le n$, one has 
\[ 
	d(f^k(x_i),f^k(x_j))
	<\varepsilon_0.
\]
This contradicts the fact that $x_j,x_j$ are $(n,\varepsilon_0)$-separated.  
\end{proof}

\subsubsection{Verifying assertions (1) and (2) of Theorem \ref{t.smarthorseshoes}}\label{ssec:finish}

The following proves (1).

\begin{claim}\label{cl.finite-time-Lyapunov-exponent-for-Ki}
	For each $i=1,\ldots, N$ and for every $x\in\bfK_i$, one has 
\[ 
	e^{n(\chi^\c(\nu)-\varepsilon_{\rm L})}< \|D^\c f^n(x)\|< e^{n(\chi^\c(\nu)+\varepsilon_{\rm L})}. 
\]
\end{claim}
\proof 
Fix $1\le i\leq N$. Inequality \eqref{eqformula} in Claim \ref{lemcl.disjoint} allows us to use the modulus of continuity in \eqref{e.threeconditions1} getting that for every $x\in \bfK_i$ and $0\le k\le n$,
\[ 
	\varepsilon_{\rm L}/4<\log\|D^\c f(f^k(x))\|-\log\|D^\c f(f^k(x_i))\|<\varepsilon_{\rm L}/4.
\]
Combining with \eqref{itemii}, for every  $x\in \bfK_i$ and $0\le k\le n$ one has
\begin{equation}\label{eq:center-Lyapunov-exponent-for-Ki}
	K_{\rm L}^{-1} e^{k(\chi^\c(\nu)-\varepsilon_{\rm L}/2)}
	< \|D^\c f^k(x)\|
	<K_{\rm L} e^{k(\chi^\c(\nu)+\varepsilon_{\rm L}/2)}.
\end{equation} 
By \eqref{e.largen}, for every $x\in \bfK_i$ one  gets 
\[	  
	e^{n(\chi^\c(\nu)-\varepsilon_{\rm L})} 
	< \|D^\c f^n(x)\|
 	 < e^{n(\chi^\c(\nu)+\varepsilon_{\rm L} )},
\]
proving the claim.
\endproof 
	
As we deal with a partially hyperbolic diffeomorphism with one-dimensional center, Claim \ref{cl.finite-time-Lyapunov-exponent-for-Ki} implies assertion (2): the set $\Lambda=\Lambda_\bfH$ in \eqref{defLambdaH} is a hyperbolic basic set of $f^n$. Moreover, by construction, $f^n|_{\Lambda}$ is topologically conjugate to a full shift with $N$ symbols. 

It remains to prove the estimate of the finite-time Lyapunov exponents in \eqref{e.lastestimate} and the Birkhoff sums in \eqref{e.lastestimatephis}. For every $x\in\Lambda$ and every $k\in\bN$,  write $k=\ell n+j$ with $\ell\in\bN_0$\footnote{Let $\bN_0=\bN\cup\{0\}$.}  and $0\le j<n$. With this notation, $f^{mn}(x)\in\bigcup_{i=1}^N\bfK_i$ for each $m=0,\ldots,\ell$. One gets 
\[\begin{split}
	\|D^\c f^k(x)\|
	&=\left(\prod_{m=0}^{l-1}\log\|D^\c f(f^{mn}(x))\|\right)\cdot \|D^\c f^j(f^{mn}(x))\|\\
	\text{\tiny by Claim~\ref{cl.finite-time-Lyapunov-exponent-for-Ki} and \eqref{eq:center-Lyapunov-exponent-for-Ki}}\quad
	&< e^{mn(\chi^\c(\nu)+\varepsilon_{\rm L} )}\cdot K_{\rm L}\, e^{j(\chi^\c(\nu)+\varepsilon_{\rm L}/2)}
	\le K_{\rm L}\, e^{k(\chi^\c(\nu)+\varepsilon_{\rm L} )}.
\end{split}\]
This lower estimate is analogous. The estimates of the Birkhoff sums are analogous. This completes the proof of Theorem~\ref{t.smarthorseshoes}.
\qed

\subsection{Further properties of the horseshoes in Theorem \ref{t.smarthorseshoes}}\label{secHorseshoefurther}

The properties obtained in this section will be important when studying the interaction between a $\csu$-horseshoe and a blender-horseshoe, see Section \ref{s.connecting}.

\begin{remark}\label{remcover}
	It follows from the definition of the collection of rectangles $\{\bfK_i\}_i$ in Theorem \ref{t.smarthorseshoes} that every $\uu$-complete subcube of $C$ $\uu$-covers $\bfK_\ell$, for every $\ell=1,\ldots,N$.
\end{remark}

\begin{notation}[Cylinders]\label{notCylinders}
	Given a $\csu$-horseshoe $\bfH=(C,\{\bfK_i\}_{i=1}^N,f^n)$, for every $m$-word $\ba=(a_0,\ldots,a_{m-1})\in\{1,\ldots,N\}^m$, let
	\[
	\bfK_\ba
	\eqdef \bigcap_{i=0}^{m-1}f^{-in}(\bfK_{a_i}).
	\]	
\end{notation}

\begin{corollary} \label{cl.cubeshorseshoe}
Let $(C, \{\bfK_i\}_{i=1}^N, f^n)$ be a horseshoe provided by Theorem~\ref{t.smarthorseshoes}. Then for every $m\in\bN$ and every $m$-word $\ba=(a_0, \ldots, a_{m-1})\in\{1,\ldots,N\}^m$, the cylinder  $\bfK_\ba$  is a $\cs$-complete subcube of $\bfK_{a_0}$ such  that
\begin{enumerate}
\item $f^{j n} (\bfK_\ba) \subset \bfK_{a_j}$  \text{for every} $j=0, \ldots, m-1,$
\item $f^{mn}(\bfK_\ba)$ \text{ is a $\uu$-complete subcube of } $C$,  
\item there exists a base $D^\cs_{\bfa}\subset \bfK_\ba$ of $\bfK_\ba$.
\end{enumerate} 
Furthermore, there exists $m_0\in\bN$ such that for every $m\geq m_0$ and every $m$-word $\ba\in\{1,\ldots,N\}^m$,   
 every  base $D^\cs_{\bfa}\subset \bfK_\ba$ of $\bfK_\ba$ satisfies that 
\[
	\diam ( f^{mn}(D^{\cs}_{\bfa})) <e^{mn(\chi^\c(\nu)+2\varepsilon_{\rm L})}.
\]
In particular, $f^{mn}(\bfK_\ba)$ $\uu$-covers every $\bfK_\ell$, $\ell=1,\ldots,N$.
\end{corollary}

\begin{proof}
We argue by induction on $m$. For $m=1$, the assertion is stated in Theorem \ref{t.smarthorseshoes}. Assume that the assertion holds for $m\ge1$. Let $(a_0,\ldots,a_{m-1},a_m)\in\{1,\ldots,N\}^{m+1}$. By induction hypothesis, $\bfK_{(a_0,\ldots,a_{m-1})}$ is a $\cs$-complete subcube of $\bfK_{a_0}$ and $f^{mn}(\bfK_{(a_0,\ldots,a_{m-1})})$ is a $\uu$-complete subcube of $C$. Hence, by Remark \ref{remcover},  $f^{mn}(\bfK_{(a_0,\ldots,a_{m-1})})$ $\uu$-covers $\bfK_{a_m}$. Therefore, we can consider a $\uu$-complete subcube $\bfK'_{a_m}$ of $\bfK_{a_m} \cap f^{mn}(\bfK_{(a_0,\ldots,a_{m-1})})$. We have that 
\[
	f^{-mn} (\bfK'_{a_m})
	= \bfK_{(a_0,\ldots,a_{m-1},a_m)}
\]	 
is a $\cs$-complete cube of $\bfK_{a_0}$ contained in $\bfK_{(a_0,\ldots,a_{m-1})}$ such that $f^{(m+1)n} (\bfK_{(a_0,\ldots,a_{m-1},a_m)})$ is a $\uu$-complete subcube of $C$. This finishes the induction and proves items 1  and 2.
	
	To conclude the proof, we need  to recall some facts from the proof of Theorem~\ref{t.smarthorseshoes}.
For every  $i=1,\ldots, N$, the subcube $\bfK_i$ has a base $D^\cs_i\subset \bfK_i$. Since $f^{mn}(\bfK_\ba)$ is a $\uu$-complete subcube of  $C$, thus for each $i$, $f^{mn}(\bfK_\ba)\cap D^\cs_i$ is a base for  $f^{mn}(\bfK_\ba)$, then 
\[ 
	f^{-mn}\big(f^{mn}(\bfK_\ba)\cap D^\cs_i\big)
	\subset \bfK_\ba
\]
gives a base for $\bfK_\ba$ which is contained in $\bfK_\ba$, proving item 3. 

Note that each base $D^\cs$ of $C$ contained in $C$ can be viewed  as a graph (recall Remark \ref{r.csucubebis}) and thus there exists $r'>0$, independent of the  choice of $D^\cs$, so that  $\diam (D^\cs)<r'$.

By the construction of $\bfK_\ba$, $f^{jn}(\bfK_\ba)\subset \bfK_{a_{j}} ,$ for each $j=0,\ldots,m-1$.  
Take any base $D^\cs_\ba\subset \bfK_\ba$, then $\diam(D^\cs_\ba)<r'$. By Claim~\ref{cl.finite-time-Lyapunov-exponent-for-Ki} and the domination between $E^\c$ and $E^\ss$, one gets 
\[
	\diam(f^{mn}(D^\cs_\ba))
	\leq r' (e^{mn(\chi^\c(\nu)+\varepsilon_{\rm L})}+e^{mn\tau^\ss}).
\]
Take $m_0\in\bN$ such that for every $m\geq m_0$,
\[ 
	r' (e^{mn(\chi^\c(\nu)+\varepsilon_{\rm L})}+e^{mn\tau^\ss})
	<e^{mn(\chi^\c(\nu)+2\varepsilon_{\rm L})}, 
\]
ending the proof of the corollary.
\end{proof}

\begin{notation}\label{notrectangle}
In analogy to the notation of base and fibers of a general $\csu$-cube, for $\bfK_i$, $i\in\{1,\ldots,N\}$, fix
\[
	D^\cs_i \subset \bfK_i
	\quad\text{ and }\quad
	\cD_i\eqdef\cD(\bfK_i)=\{D^\uu_i(x)\colon x\in D^\cs_i\}
\]
a base and its associate family of fibers, respectively. Moreover, consider the associated safe-number $\tau_i\eqdef\tau(\bfK_i)$ in Remark \ref{r.safeness-of-cube}.

Analogously, given $k\in\bN$ and $\bfa=(a_0,\ldots,a_{k-1})\in\{1,\ldots,N\}^k$, for the cube $\bfK_\bfa$, fix
\[
	D^\cs_\bfa
	\subset \bfK_\bfa
	\quad\text{ and }\quad
	\cD_\bfa\eqdef\cD(\bfK_\bfa)=\{D^\uu_\bfa(x)\colon x\in D^\cs_{a_0}\}
\]
a base as provided by Corollary \ref{cl.cubeshorseshoe} and its associate family of fibers, respectively. 
Note that, as $\bfK_\ba$ is a $\cs$-complete subcube of $C$, we have that $D^\cs_\ba$ is also a base of $C$ contained in $C$. Note also that every disk in $\cD_\bfa$ is contained some disk in $\cD_{a_0}$.
\end{notation}

\begin{remark}\label{remtaui-remDh} 
	For every $i=1,\ldots,N$, there is a family of $\tau_i$-$\bfK_i$-safe $\uu$-disks. Moreover, there are $\tau_{\rm h}>0$ and a family $\cD_{\rm h}$ of $\uu$-disks which is $\tau_{\rm h}$-$\bfK_i$-safe simultaneously for every $i=1,\ldots,N$.
\end{remark}

Note that all above constructions only required partial hyperbolicity. 

\begin{remark}\label{remmini}
If the strong unstable foliation is minimal, there is $\nu_{\rm h}>0$ so that $\cW^\uu(x,\nu_{\rm h})$ contains a disk in $\cD_{\rm h}$ for every $x\in M$.
\end{remark}

\section{$\csu$-horseshoes and suspension spaces}\label{sectophorssus}

In order to conveniently define ergodic measures supported on a $\csu$-horseshoe, in this section we introduce discrete-time suspensions of associated shift maps.

\subsection{Suspension of an abstract shift}

We start by fixing some notation. Consider a finite alphabet $\cA$. Given $R\in\bN$,  the \emph{discrete-time suspension system} $(\cS_{\cA,R},\Phi_{\cA,R})$ of the full shift $(\cA^\bZ,\sigma_\cA)$ with constant roof function $R$ is defined as 
\begin{itemize}[leftmargin=0.6cm ]
\item   (Suspension space)
\[
	\cS_{\cA,R}\eqdef (\cA^\bZ \times \bZ)/\sim,
\]
where $\sim$ is the equivalence relation identifying $(\ua,s)$ with $(\sigma_\cA (\ua), s- R)$;
\item (Suspension map)
\[
\Phi_{\cA,R}\colon\cS_{\cA,R}\to\cS_{\cA,R},\quad
\Phi_{\cA,R}(\underline a,s)
\eqdef\begin{cases}
	(\underline a,s+1)&\text{ if }s\in\{0,\ldots, R-2\},\\
	(\sigma_\cA(\underline a),0)&\text{ if }s=  R -1.
\end{cases}
\]
\end{itemize}

Unless stated otherwise, we represent each equivalence class by its \emph{canonical representation} $(\underline a,s)$ with $s\in\{0,\ldots, R-1\}$.
Equip $\cA^{\bZ}$ with the metric 
\[
	\dist_{\cA^{\bZ}}(\underline a,\underline b)\eqdef \exp\big(-\inf\{\lvert k\rvert\colon a_k\ne b_k\}\big).
\] 
We equip the suspension  space with a metric $\rho$ given as follows:  
\[
	\rho\big((\ua,s),(\ub,t)\big)\eqdef
	\begin{cases}\dist_{\cA^\bZ}(\ua,\ub)&\text{ if }s=t,\\
	1&\text{ otherwise}.
	\end{cases} 
\]

Denote by $\fm$ the counting measure on $\bN$. Given a measure $\nu\in\cM_{\rm erg}(\sigma_\cA)$, consider the \emph{suspension} of $\nu$ by $R$ defined by
\begin{equation}\label{eqtildelambda}
	\lambda_{\cA,R,\nu}
	\eqdef \frac{(\nu\times\fm)|_{\cS_{\cA,R}}}{(\nu\times\fm)(\cS_{\cA,R})}
	= \frac{(\nu\times\fm)|_{\cS_{\cA,R}}}{R}.
\end{equation}

\begin{remark}\label{rem:some}
	The measure $\lambda_{\cA,R,\nu}$ is a $\Phi_{\cA,R}$-invariant and ergodic probability measure. Moreover, $\nu\mapsto\lambda_{\cA,R,\nu}$ defines a homeomorphism between the spaces of ergodic measures $\cM_{\rm erg}(\sigma_\cA)$ and $\cM_{\rm erg}(\Phi_{\cA,R})$.
\end{remark}	

\subsection{Factors of suspension spaces associated to $\csu$-horseshoes}\label{ss.symbolic-horseshoe}

Throughout this subsection, let $f\in \PH^1_{\c=1}(M)$. Consider a $\csu$-horseshoe $\bfH=(C, \{\bfK_i\}_{i=1}^N, f^R)$  relative to an $\csu$-cube $C$. Consider the associated maximal $f^R$-invariant set 
\begin{equation}\label{defLambda}
\Lambda_\bfH=\bigcap_{k\in\bZ}f^{kR}\left(\bigcup_{i=1}^N\bfK_i \right).
\end{equation}
Consider also the $f$-invariant set
\begin{equation}\label{defLambdatilde}
	\widetilde\Lambda_\bfH
	\eqdef\bigcup_{k=0}^{R-1}f^k(\Lambda_\bfH).
\end{equation}
Assume that $f^R|_{\Lambda_\bfH}$ is of contracting type. Fix the finite alphabet $\cA\eqdef\{ 1,\ldots, N\}$ and note that $f^R|_{\Lambda_\bfH}$ is topologically conjugate to the full shift $(\cA^\bZ,\sigma_\cA)$. 

Given $\underline{a}=(a_i)_{i\in\bZ}\in\cA^\bZ$, one defines 
\begin{equation}\label{defKai}
	\bfK(\underline{a})\eqdef\bfK_{a_0}. 
\end{equation}
	
Now, we define the projection from suspension space to the manifold $M$.
\begin{definition}[Projection of the suspension space]\label{d.projeciton-of-suspension-space}
Let  $(\cS_{\cA,R},\Phi_{\cA,R})$ be the suspension of the symbolic dynamics $(\cA^\bZ,\sigma_{\cA})$ with constant roof function $R$. The projection of the suspension space to the manifold $M$ is defined as 
\[
 	\Pi \colon \cS_{\cA,R}\to \widetilde\Lambda_\bfH=\bigcup_{i=0}^{R-1}f^i(\Lambda_\bfH),\qquad
 	(\underline{a},s)\mapsto f^s \Big(\bigcap_{i\in\bZ}f^{-iR}(\bfK_{a_i}) \Big),
\]
where $(\underline{a},s)$ is in its canonical representation, and $\underline{a}=(a_i)_{i\in\bZ}$.
 \end{definition}
 
Let us collect some basic properties of the map $\Pi$. 

\begin{lemma}[The factor map $\Pi$]\label{r.factor}
The map $\Pi$ satisfies the following properties.
\begin{enumerate}[leftmargin=0.6cm ]
\item The map $\Pi$ is continuous, surjective, and satisfies
\[
	\Pi\circ \Phi_{\cA,R}=f\circ \Pi. 
\]
\item For every $s=0,\ldots, R-1$, the restricted map 
\[
	\Pi|_{\cA^{\bZ}\times\{s \}}\colon\cA^{\bZ}\times\{s \}\to f^s(\Lambda_\bfH)
\] 
is a homeomorphism and satisfies
\[
	f^R|_{f^s(\Lambda_\bfH)}\circ \Pi|_{\cA^{\bZ}\times\{s \}}
	=\Pi|_{\cA^{\bZ}\times\{s \}}\circ \Phi_{\cA,R}^R|_{\cA^{\bZ}\times\{s \}}.
\] 
\item $\Pi$ preserves the entropy in the sense that
\[
	h_\lambda(\Phi_{\cA,R})=h_{\Pi_\ast\lambda}(f) \quad\text{ for every }\quad
	\lambda\in\cM(\Phi_{\cA,R}).
\]
\end{enumerate} 
\end{lemma}

\proof
Note that for each $\underline{a}\in\cA^\bZ$, the set $\Pi(\underline{a},0)=\bigcap_{i\in\bZ}f^{-iR}(\bfK_{a_i})$ contains a unique point in $\Lambda_\bfH$, and this correspondence $(\underline{a},0)\mapsto\Pi(\underline{a},0)$ is a topological conjugacy between $(\cA^\bZ,\sigma_\cA)$ and $f^R|_{\Lambda_\bfH}$. Combining this fact with the definitions of $\Pi$ and the suspension space, items 1 and 2 of the assertion follow. 

Given $\lambda\in\cM(\Phi_{\cA,R})$, let $\nu\in\cM(\sigma_\cA)$ be its push forward under the natural projection $(\underline{a},s)\mapsto\underline{a}$.  For each $s\in\{0,\ldots, R-1 \}$, the measure $\nu$ naturally gives a probability measure $\nu_s$ on $\cA^\bZ\times\{s\}$ which is $\Phi_{\cA,R}^R$-invariant. Note that $\lambda=\frac{1}{R}\sum_{s=0}^{R-1}\nu_s$. Consider the measure $\Pi_\ast(\lambda)=\frac{1}{R}\sum_{s=0}^{R-1}\Pi_\ast\nu_s$. Then by item 2, each $\Pi_\ast\nu_s$ is an $f^R$-invariant measure, and 
$ h_{\Pi_\ast\nu_s}(f^R)=h_{\nu_s}(\Phi_{\cA,R}^R)$. Taking the average over $s$, one gets 
\[ 
	h_{\Pi_\ast\lambda}(f^R)=h_{\lambda}(\Phi_{\cA,R}^R).
\]
As $\Pi_\ast\lambda$ and $\lambda$ are $f$-invariant and $\Phi_{\cA,R}$-invariant respectively, one gets  
\[
	h_{\Pi_\ast\lambda}(f)=h_{\lambda}(\Phi_{\cA,R}).
\]
This proves item 3.
\endproof

\section{Dynamics inside blender-horseshoes}\label{ss.twolemmas}

In this section, we introduce blender-horseshoes. They have associated a special family of $\uu$-disks called \emph{in-between}. Using these disks, we introduce a special class of $\csu$-cubes, called \emph{$\mathcal{B}$-cubes}, and study their iterations.

\subsection{Blender-horseshoes: Definition and preliminaries}
\label{ss.blenderhorseshoes}

We briefly recall the basic properties of blender-horseshoes given in \cite{BonDia:12}.

An \emph{unstable blender-horseshoe} is a pair $\fB=\fB(\bfC, g=f^k)$,  for some $k\in\bN$, such that $\bfC$ is a $C^1$-embedded rectangle and the set
\begin{equation}\label{basicset}
	\Lambda_\fB
	\eqdef \bigcap_{n\in\bZ}g^n(\bfC)
	\subset \interior(\bfC)
\end{equation}
is the maximal invariant set of  $g$ in $\bfC$ and is a hyperbolic basic set whose tangent space has a partially hyperbolic splitting $T_{\Lambda_\fB} M=E^{\ss}\oplus E^{\c}\oplus E^{\uu}$, where $E^{\c}$ is one-dimensional and uniformly expanding. Some further conditions (BH1)--(BH9) are required, and we just sketch these properties. See the detailed definition in \cite[Section 3.2]{BonDia:12} and also \cite[Section 2]{DiaGelSan:20}. As in \eqref{defLambdatilde}, we also consider the full $f$-invariant set
 \begin{equation}\label{basicsettilde}
	\widetilde\Lambda_\fB
	\eqdef \bigcup_{n\in\bZ}f^n(\Lambda_\fB).
\end{equation}

The restriction $g|_{\Lambda_\fB}$ is topologically  conjugate to the full shift with two symbols. More precisely, conditions (BH1) and (BH3) provide the existence of a Markov partition with two disjoint ``sub-rectangles'' $\bfC_1$ and $\bfC_2$ of $\bfC$ which are the connected components of $g^{-1}(\bfC)\cap \bfC$. There are exactly two fixed points $P\in \bfC_1$ and $Q\in \bfC_2$ of $g|_{\Lambda_\fB}$. We define $\cW^\dagger_{\rm loc}(P)$ as the connected component of $\cW^\dagger(P)\cap \bfC$ which contains $P$, for $\dagger=\ss,\u,\uu$; analogously for $\cW^\dagger_{\rm loc}(Q)$.

Condition (BH2) requires the existence of continuous cone fields $\cC^{\ss},\cC^\cu$, and $\cC^{\uu}$ around the bundles $E^\ss, E^\cu$, and $E^\uu$, respectively, defined on $\bfC$ such that 
\begin{itemize}[leftmargin=1.2cm ]
\item[(BH2a)] $\cC^{\ss}$ restricted on $g(\bfC_1 \cup  \bfC_2)$ is $Dg^{-1}$-invariant;
\item[(BH2b)]  $\cC^\cu$ and $\cC^{\uu}$ restricted on  $\bfC_1\cup \bfC_2$ are $Dg$-invariant;
\item[(BH2c)] the vectors in the cone fields $\cC^\cu,\cC^{\uu}$ restricted to $\bfC_1\cup \bfC_2$ are uniformly expanded by $Dg$;
\item[(BH2d)] the vectors in $\cC^{\ss}$ restricted to $g(\bfC_1 \cup  \bfC_2)$ are uniformly expanded by $Dg^{-1}$.
\end{itemize} 
 
A disc tangent to the cone field $\cC^{\uu}$ is \emph{$\cC^{\uu}$-complete}, if it crosses $\bfC$ ``completely''. Particular examples are given by $\cW^{\uu}_{\rm loc}(P)$ and $\cW^{\uu}_{\rm loc}(Q)$. 

Condition (BH4) states that every $\cC^{\uu}$-complete disk containing a point of $\cW^{\ss}_{\rm loc}(P)$ and every $\cC^{\uu}$-complete disk containing a point of $\cW^{\ss}_{\rm loc}(Q)$ are disjoint. Condition (BH4) allows us to defined the $\cC^{\uu}$-complete disks which are at the right or the left of $\cW^{\ss}_{\rm loc}(P)$ and of $\cW^{\ss}_{\rm loc}(Q)$, respectively. In particular, this allows us to define a \emph{$\cC^{\uu}$-complete disks in-between} $\cW^{\ss}_{\rm loc}(P)$ and  $\cW^{\ss}_{\rm loc}(Q)$ (or \emph{$\cC^{\uu}$-complete disks in-between}, for short). The set of all $\cC^{\uu}$-complete disks in-between is called the \emph{superposition region} of the blender-horseshoe. 

Conditions (BH5) concerns about the iterations of a $\cC^{\uu}$-complete disk and its positions with respect to $\cW^{\ss}_{\rm loc}(P)$ and $\cW^{\ss}_{\rm loc}(Q)$.
 
Condition (BH6) states that for any $\cC^{\uu}$-complete disk $D$ in-between, $g(D)$ contains a $\cC^{\uu}$-complete in-between. This is indeed one of the key properties of a blender-horseshoe in our setting.

\begin{remark}\label{iteratesB}
Given an unstable blender-horseshoe $\fB=\fB(\bfC, g)$, for any $\ell\in\bN$ the pair $(\bfC,g^\ell)$ also defines an unstable blender-horseshoe. 
\end{remark}

\subsection{Iterations of $\uu$-disks in-between}\label{sss.in-between}

In what follows, we fix an unstable blender-horseshoe $\fB =\fB(\bfC,g)$ and its superposition region. We begin by collecting some properties of this family of disks. Here, we will focus only on $\cC^\uu$-complete disks which are contained in the strong unstable manifolds $\cW^{\uu}_{\rm loc}(\cdot)$ (and not general disks tangent to the cone field $\cC^\uu$). We only require to study these special collection of disks.

\begin{definition}[Strong superposition regions $\mathcal{B}$ and $\mathcal{B}_1,\mathcal{B}_2$]\label{defBinbetween}
	 We denote by $\mathcal{B}$ the set of all $\uu$-disks in the superposition region of the blender-horseshoe $\fB$, that is, of all $\cC^{\uu}$-complete disks in-between which are contained in the strong unstable manifolds $\cW^{\uu}_{\rm loc}(\cdot)$. Let
\[
	\mathcal{B}_i\eqdef \mathcal{B}\cap\bfC_i,
	\quad i=1,2.
\]
\end{definition}

\begin{definition}[Superposition cube]
A \emph{superposition cube} or \emph{$\mathcal{B}$-cube} is an $\csu$-cube whose fibers are in the strong superposition region $\mathcal{B}$. Analogously, define \emph{$\mathcal{B}_i$-cubes}, $i=1,2$.
\end{definition}

\begin{remark}[Uniform size of disks in $\mathcal{B}$]\label{r.uniformsizeblender}
There is $\nu_{\mathrm{b}}>0$ such that every $\uu$-disk in $\mathcal{B}$ contains a disk of the form $\cW^\uu (z,\nu_{\mathrm{b}})$ for some $z\in M$.
\end{remark}

\begin{definition}[Distance and diameter]\label{defdiamH}
Analogously to Definition \ref{d.safedisk-biss}, given $\Delta\in\mathcal{B}$ and $\tau>0$, consider $B(\Delta,\tau)$ and given $\mathcal{B}'\subset\mathcal{B}$, we define its \emph{diameter}
\[
	\diam_{\rm H}(\mathcal{B}')
	\eqdef \sup\big\{\tau>0\colon 
		\dist_{\rm H} (T\Delta, T\Delta')<\tau\text{ for every }\Delta,\Delta'\in\mathcal{B}'\big\}.
\]
By a slight abuse of notation, viewing a $\mathcal{B}$-cube $C$ as a family of disks, we define analogously its diameter $\diam_{\rm H}(C)$.
\end{definition}
 
Lemma \ref{l.safedisk} restates  that the property of a $\uu$-disk to be in-between is a robust property. 

\begin{lemma}[Safe disks in-between]\label{l.safedisk}
There are $\tau_\rmb>0$ and an open  family $\mathcal{B}_\rmb\subset\mathcal{B}$ which is $\tau_\rmb$-$\mathcal{B}$-safe.
\end{lemma}

The following is a consequence of the key property (BH6) of blender-horseshoes in our context.

\begin{lemma}[Images of $\mathcal{B}_i$-cubes]\label{l.imabicubes}
The image of any $\mathcal{B}_i$-cube contains a $\mathcal{B}$-cube.
\end{lemma}

Let us now define ``successors'' of $\mathcal{B}$-cubes with sufficiently small diameter. Note that, by condition (BH6) and the very definition of $\mathcal{B}$, every disk $\Delta\in\mathcal{B}$ always has some ``successor'' $\widetilde\Delta\in\mathcal{B}$ in the sense that $\widetilde\Delta\subset g(\Delta)$. 
Given two disks $\Delta,\Theta$ which are ``close'', we want to define their corresponding ``successors'' $\widetilde\Delta$ and $\widetilde\Theta$ in such a way that they are also ``close''. The next lemma provides an algorithm for defining such successors.

\begin{lemma}[Successor of a $\mathcal{B}$-cube]\label{l.safeblender}
There is $\rho_{\mathrm{b}}>0$ such that for every $\mathcal{B}$-cube $\bfK$ with base $D^\cs$ and fibers $\{D(x)\colon x\in D^\cs\}$ satisfying $\diam_{\mathrm{H}} (\bfK) <\rho_{\mathrm{b}}$ the following holds:
\begin{enumerate}[leftmargin=0.6cm ]
\item[(1)] either $g(\bfK)$ contains a $\mathcal{B}$-cube $\bfK_1$ with base $g(D^\cs)$ which itself contains a $\mathcal{B}_1$-subcube with base $g(D^\cs)$, 
\item[(2)] or $g(\bfK)$ contains a $\mathcal{B}$-cube $\bfK_2$ with base $g(D^\cs)$ which itself contains a $\mathcal{B}_2$-subcube with base $g(D^\cs)$,
\end{enumerate}
We let 
$\bfK^{[1]}\eqdef \bfK_1$ if case (1) occurs and 
$\bfK^{[1]}\eqdef \bfK_2$ otherwise (that is, if (1) does not occur and hence (2) does) and we call 
$\bfK^{[1]}$ the \emph{successor} of $\bfK$. 
The resulting cube is of the form
$$
	\bfK^{[1]}
	\eqdef \bigcup_{g(x)\in g(D^\cs)} D^{[1]} (g(x)), 
	\quad \mbox{where} \quad D^{[1]} (g(x)) \subset g( D(x)).
$$
\end{lemma}

In the above lemma, we emphasize that $\bfK^{[1]}$ is uniquely defined.

We emphasize that the definition of \emph{successor} in Lemma \ref{l.safeblender} requires a $\mathcal{B}$-cube to have diameter less than $\rho_{\mathrm{b}}$ to guarantee that its successor is well defined. The next step is to consider successors of higher order, inductively applying Lemma~\ref{l.safeblender}.
      
\begin{definition}[Higher-order successors]\label{d.c-herentitineraries}
A $\mathcal{B}$-cube $\bfK$ with $\diam_{\mathrm{H}} (\bfK) <\rho_{\mathrm{b}}$ \emph{has a successor of order $j$} if there are $\mathcal{B}$-cubes $\bfK^{[i]}$, $i=0,\ldots, j$, such that 
\begin{itemize}[leftmargin=0.6cm ]
\item $\bfK^{[0]}=\bfK$,
\item $\diam_{\mathrm{H}} (\bfK^{[i]}) < \rho_{\mathrm{b}}$ for every $i=0, \ldots, j-1$,
\item $\bfK^{[i+1]} =  (\bfK^{[i]})^{[1]} $ for every $i=0, \ldots, j-1$.
\end{itemize}
We call $\bfK^{[j]}$ the \emph{$j$th successor of $\bfK$}. 
\end{definition}

\begin{remark}
The $k$th successor of a $\mathcal{B}$-cube (if it exists) is uniquely defined. 
\end{remark}

\begin{remark}[{The number $\ell(\bfK)$}]\label{r.lastsucessor}
Given a  $\mathcal{B}$-cube $\bfK$ with $\diam_{\mathrm{H}} (\bfK) <\rho_{\mathrm{b}}$
there is a number $\ell (\bfK)\in\bN_0$ such that the sucessors  $\bfK^{[j]}$ of $\bfK$ are defined for every
 $j=1, \dots, \ell (\bfK)+1$ and $\diam_{\mathrm{H}} (\bfK^{[\ell (\bfK)+1]}) \ge \rho_{\mathrm{b}}$.
\end{remark}

\section{Extending a horseshoe using a blender-horseshoe}\label{s.connecting}

In this section, we explore the cyclic interaction between horseshoes of contracting type and unstable blender-horseshoes. Extracting the essential requirements of our constructions, we precisely state a hypothesis about the cyclic relations between an enveloping $\csu$-cube and a  blender-horseshoe. This sort of configuration provides a special heterodimensional cycle between a  horseshoe of contracting type and a blender-horseshoe, see Remark~\ref{r.heterocycle}. Studying such cycles in full generality is beyond our goals.%
\footnote{This type of cycle is a variation of  the split flip-flop configuration from \cite[Section 4.1]{BocBonDia:16}, here a contracting saddle is replaced by a contracting horseshoe. Split flip-flop configurations were employed in \cite{BocBonDia:16} to construct non-hyperbolic ergodic measures with positive entropy. They also were analyzed in \cite{BonZha:19} from the perspective of the space of ergodic measures.}

We assume that the reference cube  of the $\csu$-horseshoe of contracting type  and the blender-horseshoe are cyclically related, see Definition~\ref{d.trelation} below. This holds, in particular, if the strong unstable foliation is minimal, see Lemma~\ref{lemr.transitiontime}. It holds also in many more general settings, which are not our focus here. Given a $\csu$-horseshoe of contracting type whose reference cube is cyclically related to a blender-horse\-shoe, we construct another $\csu$-horseshoe of contracting type with ``similar combinatorics''  whose center Lyapunov exponents are scaled by a positive constant (less than one) of the initial exponent.

\subsection{Cyclic relation between an $\csu$-cube and a blender-horseshoe}\label{ss.interactioncubeblender}

For what follows, we fix an $\csu$-cube $C$ and a blender-horseshoe $\fB =\fB(\bfC,f^S)$. We state our hypothesis, which involves some interaction between both.  Specifically, we assume that there exist uniformly bounded transition times from the cube $C$ to $\fB$ and back. 

Recall that the cube $C$ is endowed with its family of unstable disks $\cD(C)$, and the blender-horseshoe $\fB$ is equipped with the family $\mathcal{B}$ of $\uu$-disks in-between. Also, recall the subfamilies $\cD_{\rm c}$ of $\cD(C)$ in Remark \ref{r.safeness-of-cube} and $\mathcal{B}_\rmb$ of $\mathcal{B}$ in Lemma \ref{l.safedisk}.

\begin{definition}[Cyclic relations] \label{d.trelation}
An $\csu$-cube $C$ and an unstable blender-horseshoe $\fB$ are {\em{cyclically  related}} if there is $T_0\in\bN$ such that
\begin{itemize}[leftmargin=0.6cm ]
\item[(a)] for every $\Delta\in \mathcal{B}$ and $T\ge T_0$, the set $f^{T}(\Delta)$ contains a  $\uu$-disk in $\cD_{\rm c}$,
\item[(b)] for every $\Theta\in \cD(C)$ and every $T\ge T_0$ the set $f^{T}(\Theta)$ contains a $\uu$-disk in $\mathcal{B}_\rmb$.
\end{itemize}
In that case, we say that $T_0$ is the {\em{transition time}} and that  $C$ and $\fB$ are {\em{$T_0$-related}}.
\end{definition}

\begin{remark}\label{r.transitiontime}
Assume that the $\csu$-cube $C$ and  the blender-horseshoe $\fB$ are $T_0$-related. It is then a consequence of continuity of the strong unstable foliation that there are positive constants $\theta_{\rm c},\theta_{\rm b}$ such that the following holds:
\begin{itemize}[leftmargin=0.6cm ]
\item[(c)] If $C'$ is a $\uu$-complete subcube of $C$ with $\diam_{\rm H}(C')<\theta_{\rm b}$, then there is a $\cs$-complete subcube $\widetilde C$ of $f^{T_0}(C')$ which is a $\mathcal{B}$-cube fibered by disks in $\mathcal{B}_{\rm b}$ and $\diam_{\rm H}(\widetilde C)<\tau_{\rm b}$, where $\tau_{\rm b}$ is as in Lemma \ref{l.safedisk}.
\item[(d)] If $C''$ is a $\mathcal{B}$-cube with $\diam_{\rm H}(C'')<\theta_{\rm c}$, then there is a $\cs$-complete subcube of $f^{T_0}(C')$ whose fibers are in $\cD_{\rm c}$ and which  $\uu$-covers $C$.
\end{itemize}
Note that, after shrinking $\tau_{\rm b}$ and $\rho_{\rm b}$ defined in Lemma \ref{l.safeblender}, we can assume that $\tau_{\rm b}<\rho_{\rm b}<\min\{\theta_{\rm c},\theta_{\rm b}\}$.
\end{remark}

Using the uniform expansion of $\uu$-disks under $f$ and the minimality of the strong unstable foliation, the following is an immediate consequence of our above choices. 

\begin{lemma}\label{lemr.transitiontime}
	Assume that the strong unstable foliation of $f$ is minimal and that $f$ has an unstable blender-horseshoe $\fB$. Then any $\csu$-cube is cyclically related to any unstable blender-horseshoe.
\end{lemma}

\begin{remark}[Heterodimensional cycles]\label{r.heterocycle}
Consider an unstable blender-horseshoe $\fB$ and a $\csu$-horseshoe $\bfH$ of contracting type relative to an $\csu$-cube $C$. Suppose that $C$ and $\fB$ are cyclically related. Then the sets $\widetilde\Lambda_\bfH$ in \eqref{defLambdatilde} and $\widetilde\Lambda_{\fB}$ in \eqref{basicsettilde} form a heterodimensional cycle. Indeed, the fact that the unstable set of $\widetilde\Lambda_{\fB}$ and the stable set of $\widetilde\Lambda_\bfH$ intersect follows from item (a) in Definition~\ref{d.trelation}. The definition of a blender-horseshoe implies that its stable set intersects any disk in $\mathcal{B}$.  Since item (b) in Definition~\ref{d.trelation} implies that the unstable set of $\widetilde\Lambda_\bfH$ contains disks in $\mathcal{B}$,  the assertion follows.
\end{remark}

\subsection{Subordinated horseshoes with controlled exponents and entropy}
\label{ss.subordinatedhorseshoes}

We begin by introducing subordinated horseshoes. For that, recall Definition \ref{defhorseshoe} and Notation \ref{notCylinders} of cylinders.

\begin{definition}[Subordinated horseshoe]\label{defsubhorse}
	Consider a $\csu$-horseshoe relative to an $\csu$-cube $C$, $\bfH=(C,\{\bfK_i\}_{i=1}^N,f^n)$, and $m\in\bN$. We say that a $\csu$-horseshoe relative to $C$, $\bfH'=(C,\{\bfK_\ba'\}_{\ba\in\{1,\ldots,N\}^m},f^{n'})$,  is \emph{subordinated to $\bfH$} if
\begin{itemize}[leftmargin=0.6cm ]
\item[(1)] $n'\ge mn$, 
\item[(2)] for every $\ba\in\{1,\ldots,N\}^m$, it holds
\begin{itemize}[leftmargin=0.6cm ]
\item[(2a)] $\bfK_\ba'\subset\bfK_\ba$,
\item[(2b)] $\bfK_\ba'$ is a $\cs$-complete subcube of $C$,
\item[(2c)] $f^{n'}(\bfK_\ba')$ is a $\uu$-complete subcube of $C$.
\end{itemize}
\end{itemize}
\end{definition}

In the above definition, items (2b)--(2c) are redundant (they follow already from the fact that $\bfH'$ is  a $\csu$-horseshoe relative to $C$), but stated for clarity. 

For that recall $\chi_{\rm min}(f),\chi_{\rm max}(f)$ in \eqref{defminmax}, an $\csu$-cube $C^\csu(\gamma,r)$ in Definition~\ref{defGammas}, and $\varepsilon_1$ in Remark \ref{remweaint}.

\begin{theorem}\label{THEOp.key-estimate-center-Lyapunov-exponent}
	Let $f\in\PH_{\c=1}^1(M)$. Consider an $\csu$-cube  $C=C^\csu(\gamma,r)$, $r<\varepsilon_1/3$,
	and an unstable blender-horseshoe $\fB$ which are cyclically related. Then there exist numbers $\xi>0$, $\delta_0\in(0,1)$, $\chi_0<0$, and $\rho\in(0,1)$ such that the following holds. 
	
	Let $\delta\in(0,\delta_0)$,  $\widehat\chi\in(\chi_0,0)$, and $\chi\in(\chi_{\rm min}(f), 0]$. Let $\bfH=\big(C,\{\bfK_i\}_{i=1}^N,f^R\big)$ be a $\csu$-horseshoe relative to $C$ so that 
\begin{equation}\label{finitetimeLyapexp}
	-|\widehat\chi|\delta
	\le \frac1R\log\,\lVert D^\c f^R(x)\rVert -(\chi+\widehat\chi)
	\le |\widehat\chi|\delta,
	\quad\text{ for every } x\in\bigcup_{i=1}^N\bfK_i. 
\end{equation}
Then there exist $K=K(\bfH)>0$ and $m_0=m_0(\bfH)\in\bN$ such that the following holds. For every $m\ge m_0$, there exist $R'\in\bN$ satisfying
\begin{equation}\label{eqchecktimetheorem}
	mR< R' <\big(1+\xi|\widehat\chi|\big) mR
\end{equation}
and a $\csu$-horseshoe  $\bfH'= \big(C,\{\bfK_\ba'\}_{\ba\in\{1,\ldots,N\}^m},f^{R'}\big)$ subordinated to $\bfH$ such that 
\[
	-\rho|\widehat\chi|\delta
	\le \frac{1}{R'}\log\,\lVert D^\c f^{R'}(x)\rVert-(\chi+\rho\widehat\chi)
	\le \rho|\widehat\chi|\delta, \quad \text{for every $x\in\bigcup_{\ba\in\{1,\ldots,N\}^m} \bfK_\ba'$.}
\]
and
\[
	\log\|D^\c f^k(x)\|\le Ke^{k\delta\widehat\chi/2}
	\quad\text{ for every $	x\in\bfK_\ba', k=1,\ldots,R'$ }.
\]
\end{theorem}

In proving Theorem \ref{THEOp.key-estimate-center-Lyapunov-exponent}, we need to simultaneously control both the size of the images of the cubes of the subordinated horseshoes and also the Lyapunov exponents. These tasks are interconnected, as Lyapunov exponents depend on the itineraries.

\begin{remark}[Cascades of subordinated horseshoes and underlying symbolic spaces]\label{remCascade}
Theorem \ref{THEOp.key-estimate-center-Lyapunov-exponent} provides the horseshoe $\bfH'$ by ``$m$ times repeating'' the initial horseshoe $\bfH_0=\bfH$ and ``tailing in a blender-horseshoe''. Fix a sufficiently fast growing sequence $(m_k)_k$ of natural numbers such that 
\[
	m_k\ge m_0(\bfH_k).
\]
Applying this theorem inductively, we get a cascade of $\csu$-horseshoes $(\bfH_k)_{k\in\bN_0}$, where for every $k$ the horseshoe
\[
	\bfH_{k+1}
	\eqdef \bfH_k'.
\]
is obtained from the previous one $\bfH_k$ by ``repeating it'' $m_k$ times. 	
Given the ``initial alphabet'' $\cA_0=\cA\eqdef\{1,\ldots,N\}$, the underlying symbolic symbolic space of $\bfH_1$ is $\cA_1\eqdef(\cA_0)^{m_1}$. Accordingly, denoting by $\cA_k$ the underlying symbolic space of $\bfH_k$, we have the corresponding cascade of symbolic spaces
\[
	\cA_{k+1}
	\eqdef (\cA_k)^{m_k}.
\]
Finally, we have the associated repeating times $(R_k)_k$ and ``tailing times'' 
\[
	\bt_k
	\eqdef R_k-m_kR_{k-1}.
\]
The latter time is essentially the time the orbits spend in the blender-horseshoe (not counting the transition times ``connecting $\bfH_k$ to the blender-horseshoe and back'').
\end{remark}

The following is an intermediate result towards the proof of the above theorem: given any $\csu$-horseshoe $\bfH$ of contracting type, we construct a new subordinated horseshoe $\bfH'$ by ``connecting $\bfH$ to the blender-horseshoe $\fB$''. 

\begin{proposition}[Connecting $\bfH$ to itself through $\fB$]\label{pl.completetour}
Consider an $\csu$-cube  $C=C^\csu(\gamma,r)$, $r<\varepsilon_1/3$, and an unstable blender-horseshoe $\fB=\fB(\bC, f^S)$ which are cyclically related with transition time $T_0$. Then the following holds.

For every $\csu$-horseshoe relative to $C$ of contracting type, $\bfH=(C,\{\bfK_i\}_{i=1}^N,f^R)$,  
there is $m_1=m_1(\bfH)\in\bN$ such that for every $m\ge m_1$ and every $m$-word $\ba = (a_0,\ldots,a_{m-1})\in \{1, \ldots, N\}^m$ there are $\ell (\ba)\in\bN$, where $\ell(\ba)\to \infty$ as $|\ba|=m\to \infty$, and a nested family of $\cs$-complete subcubes $\{\bfK_{\ba}^{(j)}\}_{j=0}^{\ell (\ba)}$ of $\bfK_\ba$ such that: for every $j=0,\ldots,\ell(\ba)$
\begin{enumerate}[leftmargin=0.6cm ]
\item (iterates in the horseshoe) $f^{kR}  (\bfK_{\ba}^{(j)}) \subset \bfK_{a_k}$ for every $k=0, \ldots, m-1$,
\item (transition to blender and iterates there) $f^{mR+T_0+kS}(\bfK_{\ba}^{(j)})$ is contained
 in some $\mathcal{B}$-cube for every $k< j$ and $f^{mR+T_0+jS}(\bfK_{\ba}^{(j)})$ is a $\mathcal{B}$-cube,
\item (return to the horseshoe) $f^{mR+T_0+jS +T_0} (\bfK_{\ba}^{(j)})$ is a $\uu$-complete subcube of $C$, and hence $\uu$-covers $\bfK_i$ for every $i=1, \ldots, N$.
\end{enumerate}
\end{proposition}

In the remainder of this section, simultaneously for the proofs of Proposition \ref{pl.completetour} and Theorem \ref{THEOp.key-estimate-center-Lyapunov-exponent}, we fix an unstable blender-horseshoe $\fB=\fB(\bC, f^S)$ and an $\csu$-cube $C=C^\csu(\gamma,r)$ as in the Theorem \ref{THEOp.key-estimate-center-Lyapunov-exponent}. In Section \ref{ss.prelim}, we first collect some quantifiers associated to  $C$ and $\fB$. In Section \ref{ss.connecting}, we prove Proposition \ref{pl.completetour}. Note that the choice of the number of times $m$ the horseshoe $\bfH$ is ``repeated'' is still flexible. In Section \ref{sec:choicequanti}, we fix some quantifiers. In Section \ref{ss.topologicalandcontrol}, we specify $m$ to construct $\bfH'$ and to finalize the proof of Theorem \ref{THEOp.key-estimate-center-Lyapunov-exponent}.

\subsection{Transitions from $C$ to $\fB$ and \emph{vice versa}}\label{ss.prelim}

We fix some quantifiers associated with an $\csu$-cube and a blender-horseshoe.

\begin{remark}[Choice of quantifiers for $\fB$]\label{remquantiCBB}
Recall that $\fB$ comes with a family $\mathcal{B}$ of $\uu$-disks in-between, see Definition \ref{defBinbetween}. Let
\begin{itemize}[leftmargin=0.6cm]
\item  $\rho_{\rm b}$ be as in Lemma \ref{l.safeblender}, 
\item  $\tau_{\rmb}>0$ and $\mathcal{B}_\rmb\subset\mathcal{B}$ be a family of $\tau_{\rmb}$-$\mathcal{B}$-safe disks as in Lemma \ref{l.safedisk}, 
\item  $\nu_{\mathrm{b}}>0$ be as in Remark \ref{r.uniformsizeblender}, bounding the size of the disks of  $\mathcal{B}_\rmb$ from below, 
\item without loss of generality, we can assume $\tau_{\rm b}<\rho_{\rm b}$.
\end{itemize}
Consider numbers $\kappa_{\rm max}>\kappa_{\rm min}>0$ such that 
\begin{equation}\label{eq3}
	\kappa_{\rm min}
	\le \frac 1S \log\,\lVert D^\c f^S(x)\rVert
	\le \kappa_{\rm max}
	\quad\text{ for every }\quad
	x\in \bC_1\cup\bC_2.
\end{equation}
\end{remark}

\begin{remark}[Choice of quantifiers for $C$]\label{remquantiCBC}
Recall that $C$ comes with a family $\cD(C)$ of $\uu$-disks,  see \eqref{csucubebis}. Let
\begin{itemize}[leftmargin=0.6cm]
\item $\tau_{\rm c}>0$ and $\cD_{\rm c}\subset\cD(C)$ be the family of $\tau_{\rm c}$-$C$-safe $\uu$-disks, as in Remark \ref{r.safeness-of-cube},
\item $\nu_{\rm c}>0$ be as in Remark \ref{remdeftauc}, bounding the size of the $\uu$-disks of $C$ and hence $\cD_{\rm c}$ from below.
\end{itemize}
\end{remark}
\subsection{Proof of Proposition \ref{pl.completetour}}\label{ss.connecting}

Let $m\in\bN$. Note that for every $m$-word $\ba=(a_0, \ldots, a_{m-1}) \in \{1, \ldots, N\}^m$ the set $f^{mR}(\bfK_\ba)$ is a $\uu$-complete subcube of $C$. As $\bfH$ is of contracting type, there exists $m_1=m_1(\bfH)\in\bN$ such that 
\begin{equation}\label{smalldiam}
	\diam_{\rm H}\big(f^{mR}(\bfK_\ba)\big)<\theta_{\rm b}
	\quad\text{ for every }\quad
	m\ge m_1\text{ and }
	\ba \in \{1, \ldots, N\}^m.
\end{equation}

By Remark~\ref{r.transitiontime} (c), it follows from \eqref{smalldiam} that there exists a $\cs$-complete subcube $\widetilde\bfK_\ba$ of $f^{mR+T_0}(\bfK_\ba)$ which is a $\mathcal{B}$-cube fibered by disks in $\mathcal{B}_{\rm b}$. By Remark~\ref{r.transitiontime} (c), we have that 
\begin{equation}\label{eqsmall}
	\diam_{\rm H}(\widetilde \bfK_\ba)<\tau_{\rm b}<\rho_{\rm b}. 
\end{equation}	
For further reference, let us observe that indeed we have
\[
	\diam_{\rm H}(\widetilde \bfK_\ba)
	\le e^{mR(\chi+\widehat\chi(1-\delta))+T_0\tau},
\]
where $\tau=\max_{x\in M}\big|\log\,\lVert D^\c f(x)\rVert\big|.$
Recalling Notation \ref{notrectangle} and Lemma \ref{lemcubeit}, it follows that 
\[
	f^{-(mR+T_0)}(\widetilde\bfK_\ba)
	= \bigcup_{x \in D^\cs_{a_0}}  \widehat D^\uu_{\ba} (x),\qquad
		\widehat D^\uu_{\ba}(x) \subset D^\uu_{\ba} (x)
\]
is a $\cs$-complete subcube of $\bfK_{\ba}$ and for each $x \in D^\cs_{a_0}$, $f^{mR+T_0} (\widehat D^\uu_{\ba}(x))$ is a $\uu$-disk in $\mathcal{B}_\rmb$. 

Recall Definition \ref{d.c-herentitineraries} of higher-order successors of $\mathcal{B}$-cubes. Hence, by \eqref{eqsmall} and Remark \ref{r.lastsucessor}, there exists the associated successor of $\widetilde\bfK_\ba$ (with respect to the map $g=f^S$) as well as the maximal order of succession $\ell(\ba)\eqdef\ell(\widetilde\bfK_\ba)$. 

The following claim is a consequence of the definition of $\ell(\ba)$.

\begin{claim}\label{claimellba}
	For every $m\in\bN$ and every $m$-word $\ba$, it holds
\[
	\ell(\ba)
	\ge \frac1S\frac{1}{\kappa_{\rm max}}
		\Big(\log\rho_{\rm b}-T_0\tau+mR|\chi+\widehat\chi(1-\delta)|\Big).
\]	
Hence, in particular, $\ell (\ba)\to \infty$ as $|\ba|=m \to \infty$.
\end{claim}

For each $j=0,\ldots, \ell(\ba)$,  we consider the $j$th successor $\widetilde \bfK_{\ba}^{[j]}$ of $\widetilde \bfK_{\ba}$. Note that each $\widetilde \bfK_{\ba}^{[j]}$ is a $\mathcal{B}$-cube. This allows to define
the $\cs$-complete subcube $\bfK_{\ba}^{(j)}$ of $\bfK_\ba$,
\begin{equation}\label{e.hatKab}
	\bfK_{\ba}^{(j)}\eqdef f^{-(mR+T_0+jS)} (\widetilde \bfK_{\ba}^{[j]}).
\end{equation}
Since $\widetilde \bfK_\ba^{[j]}$ is a $\mathcal{B}$-cube, this proves item 2. By construction, $\bfK_{\ba}^{(j)}\subset \bfK_\ba$ and hence Corollary \ref{cl.cubeshorseshoe} implies item 1 of the proposition. 

Note that, by construction, the sequence is nested 
\[
	\bfK_{\ba}^{(j+1)} \subset \bfK_{\ba}^{(j)}.
\]
Moreover, 
\[
	\diam_{\rm H} (f^{mR+T_0+jS} ( \bfK_{\ba}^{(j)}))
	< \rho_{\rm b}
	< \theta_{\rm c} 
	\quad \mbox{for every $j=0,\ldots, \ell (\ba)$}. 
\]
Moreover, as $\rho_{\rm b}<\theta_\c$, applying Remark \ref{r.transitiontime} (d), we get that 
\[	
	f^{T_0} (\widetilde \bfK_\ba^{[j]})
	= f^{mR+T_0+jS+T_0}( \bfK_\ba^{(j)})
\]	 contains a $\cs$-complete subcube  whose fibers are in $\cD_{\rm c}$. Hence, in particular, this image $\uu$-covers $C$ and hence every $\bfK_i$, $i=1,\ldots,N$. This proves item 3  and finishes the proof of Proposition \ref{pl.completetour}.
\qed

\subsection{Choice of quantifiers in the proof of Theorem \ref{THEOp.key-estimate-center-Lyapunov-exponent}}\label{sec:choicequanti}

Fix $\chi\in(\chi_{\rm min}(f),0]$.
The constants $\xi$, $\chi_0$, $\rho$, and $\delta_0$ in the theorem are chosen as follows
\begin{align}
	&\xi  \eqdef\frac{1}{\kappa_{\rm max}},\label{d.defkappa}\\
	&\chi_0\eqdef-\frac{\kappa_{\rm max}}{6},\label{defchi0}\\	
	&\rho
	\eqdef 1 - \frac{\kappa_{\rm min}}{2\kappa_{\rm max}} \in (0,1),\label{fixgamma}\\
	&0
	<\delta_0
	<\min\big\{\frac{1}{12},\frac{3\kappa_{\rm min}}{28\kappa_{\rm max}} \big\}\le \frac{3}{14}(1-\rho). \label{eq:choice-of-delta}
\end{align}
Note that \eqref{defchi0} implies that for $\widehat\chi\in(\chi_0,0)$ one has
\begin{equation}\label{onehass}
	\frac{|\widehat\chi|}{\kappa_{\rm max}}
	<\frac16.
\end{equation}
Also observe that with these choices, for $\delta\in(0,\delta_0)$ one has 
\begin{equation}\label{chiineq}
	\chi+\widehat\chi(1-\delta)<\frac\delta2\widehat\chi.
\end{equation}	

\subsection{Control of Lyapunov exponents: fixing times in the blender and the horseshoe}
\label{ss.topologicalandcontrol}

In this section, we consider a $\csu$-horseshoe $\bfH=(C,\{\bfK_i\}_{i=1}^N,f^R)$ of contracting type such that 
\[
	-|\widehat\chi|\delta
	\le \frac1R\log\,\lVert D^\c f^R(x)\rVert-(\chi+\widehat\chi)
	\le |\widehat\chi|\delta,
	\quad\text{ for every } x\in\bigcup_{i=1}^N\bfK_i,\]
with   $\chi\in(\chi_{\rm min}(f),0)$, $\widehat\chi\in(\chi_0,0)$, and  $\delta\in(0,\delta_0)$. 
In particular, there exists $K=K(\bfH)>1$ such that 
\begin{equation}\label{eq:finite-time-LE-for-1st-horseshoe}
	\lVert D^\c f^k(x)\rVert
	\le K\cdot e^{k(\chi+\widehat\chi(1-\delta))} 
\quad\text{ for every $x\in\bigcup_{i=1}^N\bfK_i $ and $k=1,\ldots, R$}.
\end{equation}

For every $m\in\bN$ sufficiently large, define $R'=R'(m)\in\bN$ by
\begin{equation}\label{fixm}
	mR \frac{|\widehat\chi|(1-3\delta)+\kappa_{\rm max}}{\kappa_{\rm max}}
	\le R' 
	< mR \frac{|\widehat\chi|(1-3\delta)+\kappa_{\rm max}}{\kappa_{\rm max}}+1.
\end{equation}
The following is the main result of this section. Recall $\ell(\ba)$ in Proposition \ref{pl.completetour}.

\begin{lemma}[{Choice of $T_{\rm b}(\ba)$}]\label{l.proexistHprime} 
There are $K=K(\bfH)>0$ and $m_0=m_0(\bfH)\in\bN$ such that for every $m\ge m_0$ and every $\ba \in \{1, \dots, N\}^m$, there exist an  integer 
\[
	T_{\rm b}(\ba)<\frac 1S \min\big\{R'-mR-2T_0,\ell(\ba)S \big\} 
\]
and a $\cs$-complete subcube $\bfK'_{\ba}$ of $\bfK_\ba^{(T_{\rm b} (\ba))}$ such that 
\begin{itemize}[leftmargin=0.6cm ]
	\item[(1)] $f^{R'}(\bfK'_\ba)$ is a $\uu$-complete subcube of $C$;
	\item[(2)] for every $x\in\bfK_\ba'$, one has 
	\[ R'\rho\widehat\chi(1+\delta)< \log\|D^\c f^{R'}(x)\|-R'\chi\le R'\rho\widehat\chi(1-\delta).\]
	\item[(3)] for every $x\in\bfK_\ba'$ and every $k=1,\ldots,R'$, one has
	\[ \log\|D^\c f^k(x)\|\le K\exp(k\delta\widehat\chi/2).\]
\end{itemize}
\end{lemma}

Let us briefly describe the strategy to prove Lemma \ref{l.proexistHprime}. For every cube $\bfK_\ba$ indexed by some sufficiently long $m$-word $\ba$, we are going to choose the subcube $\bfK_\ba'$ as in Definition \ref{defsubhorse} (2a). For that, we choose the iteration times in the $\csu$-cube $C$, in the blender-horseshoe $\fB$, and back in $C$. The itinerary is split  as: 
\begin{itemize}[leftmargin=0.6cm ]
\item $mR$ iterations inside $C$, followed by 
\item the transition time $T_0$ from $C$ into $\fB$, followed by 
\item $ST_{\rm b}(\ba)$ iterations inside the superposition region of $\fB$, followed by  
\item the transition time $T_0$ from $\fB$ into $C$, followed by 
\item further $T_{\rm c}(\ba)$ iterations in $C$.
\end{itemize}
Compare also Figure \ref{fig.proof1}. These iterations are chosen to ensure that there is some auxiliary point $y\in\bfK_\ba$ which follows this itinerary and whose corresponding finite-time center Lyapunov exponent satisfies \eqref{finitetimeLyapexp}. This is done in such a way that the sum $R'=mR+T_0+ST_{\rm b}(\ba)+T_0+T_{\rm c}(\ba)$ is independent of $\ba$ and comparable to $mR$.

\begin{figure}[h] 
 \begin{overpic}[scale=.5]{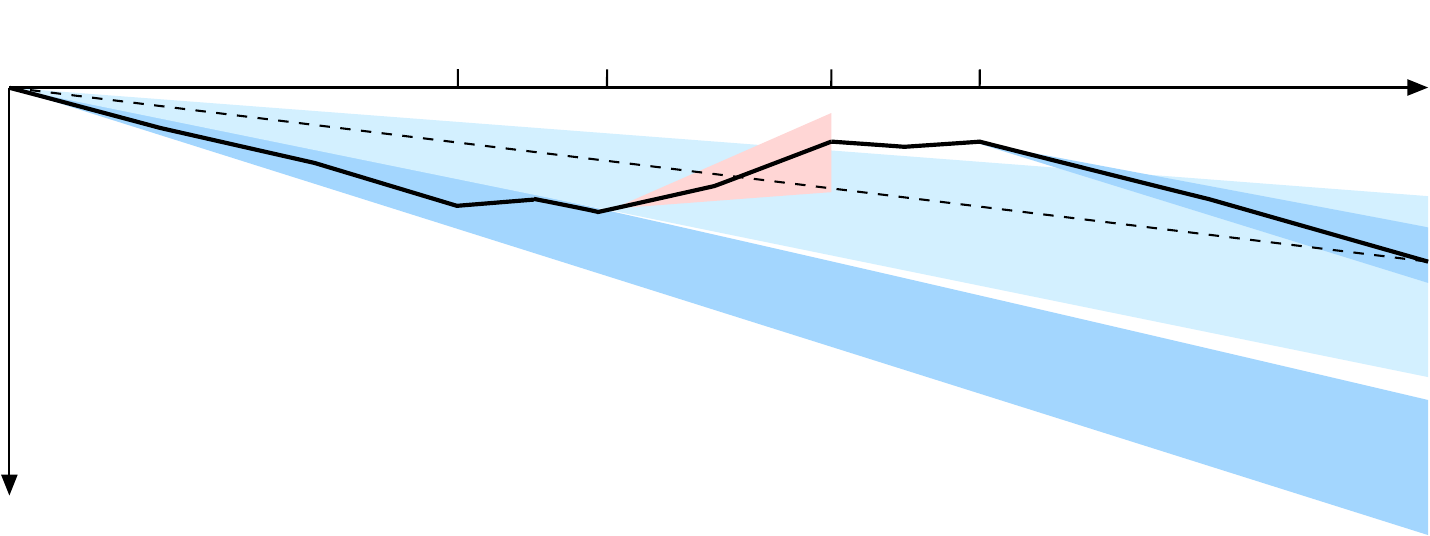}
	\put(98,28){$R'$}
	\put(14.5,35.5){$mR$}
	\put(35.5,35.5){$T_0$}
	\put(46,35.5){$ST_{\rm b}(\ba)$}
	\put(62,35.5){$T_0$}
	\put(75,35.5){$R'-2T_0-ST_{\rm b}(\ba)$}
	\put(0.5,33){$\overbrace{\hspace{3.8cm}}$}
	\put(32,33){$\overbrace{\hspace{1.23cm}}$}
	\put(42.4,33){$\overbrace{\hspace{1.85cm}}$}
	\put(58,33){$\overbrace{\hspace{1.26cm}}$}
	\put(68.6,33){$\overbrace{\hspace{3.7cm}}$}
	\put(63,12.5){{\rotatebox{-17}{\small{contraction by  $\exp\big(\chi+\widehat\chi(1\pm\delta)\big)$}}}}
	\put(60,20){{\rotatebox{-11}{\small{contraction by  $\exp\big(\chi+\rho\widehat\chi(1\pm\delta)\big)$}}}}
	\put(43.5,32){{\rotatebox{0}{\tiny{blender expansion}}}}
	\put(9,32){{\rotatebox{0}{\tiny{horseshoe contraction}}}}
	\put(75,32){{\rotatebox{0}{\tiny{horseshoe contraction}}}}
\end{overpic}
  \caption{Choice of $T_{\rm b}(\ba)$}
 \label{fig.proof1}
\end{figure}

\begin{proof}[Proof of Lemma \ref{l.proexistHprime}]	
We first fix some quantifiers and, in particular, the number $m_0$. 

By the uniform continuity of the center bundle, there exists $\eta>0$ such 	that
\begin{equation}\label{eq:choice-of-eta}
	\big|\log\|D^\c f(x)\|-\log\|D^\c f(y)\|\big|<\frac{\delta}{3}\rho|\widehat\chi|
	\quad\text{ for every }\quad
	d(x,y)<\eta. 
\end{equation}
Let $K=K(\bfH)$ be as in \eqref{eq:finite-time-LE-for-1st-horseshoe} and consider the constant $L=L(K, \delta\widehat\chi/2,\eta)\in\bN$ in Corollary \ref{c.distortion-reference-cube}. Hence, by \eqref{eq:finite-time-LE-for-1st-horseshoe} and \eqref{chiineq} applying Corollary \ref{c.distortion-reference-cube} we get  
\[
	\diam_{\rm H}(f^k(C))\le\eta \quad\text{ for every }\quad k=L,\ldots,R-L
\]
Let
\begin{equation}\label{defNormf}
	\tau 
	\eqdef \max_{x\in M}\big|\log\,\lVert D^\c f(x)\rVert\big|.
\end{equation}
Let $m_1=m_1(\bfH)$ be as in Proposition \ref{pl.completetour} applied to $\bfH$. 
Note that $1-\delta-\rho-\rho\delta/3>0$ due to \eqref{eq:choice-of-delta} and \eqref{fixgamma}. 
Choose an integer $m_0>m_1$ sufficiently large satisfying 
\begin{eqnarray}	
	\max\big\{\kappa_{\rm max},24L\tau,-\log\rho_{\rm b},18T_0\tau,6\log K\big\}
	&<&m_0R\rho|\widehat\chi|\delta<m_0R|\widehat\chi|\delta, 	\label{eqconditionsm01}\\
	S\tau+S\big(|\chi|+|\widehat\chi|(1-\delta)\big)+2T_0\tau
	&<&m_0R|\widehat\chi|\delta,  		\label{crazyS}\\
	6T_0\tau
	&<& m_0R|\widehat\chi|(1-\delta-\rho-\rho\delta/3),    \label{eqconditionsm01bis}\\
	2T_0\tau+2S\tau+S\kappa_{\rm max}&<&m_0R \frac{\delta^2|\widehat\chi|^2}{\kappa_{\rm max}}.	
	\label{eqconditionsm03}
\end{eqnarray}
Given $m\ge m_0\in\bN$, let $R'=R'(m)$ be as in \eqref{fixm}. Define $\ell_0\in \mathbb{N}$ as the largest integer such that 
\begin{equation}\label{e.defell0}
	mR\chi+mR\widehat\chi(1-\delta)+T_0\tau+\ell_0 S\kappa_{\rm max}\leq \log\rho_{\rm b}. 
\end{equation}

 For every $m\ge m_0$ and every $m$-word $\ba\in\{1,\ldots,N\}^m$, consider the number $\ell(\ba)\in\bN$ and the nested family of $\cs$-complete subcubes $\{\bfK_\ba^{(j)}\}_{j=0}^{\ell(\ba)}$  in Proposition \ref{pl.completetour}. Recall that by Claim \ref{claimellba}, we have 
 \[
 	\log\rho_{\rm b}\le mR\chi+mR\widehat\chi(1-\delta)+T_0\tau+\ell(\ba) S\kappa_{\rm max} . 	
 \]
 This immediately implies $\ell_0\le\ell(\ba)$.

\begin{claim}\label{cl.nmT0}
For every $\ba\in\{1,\ldots,N\}^m$, it holds $R'-mR-T_0\le\ell_0S \le \ell (\ba)S$.  
\end{claim}

\begin{proof}
It remains to prove the first inequality. By  \eqref{e.defell0}, it is enough to check that
\[
	mR\chi+	mR\widehat\chi(1-\delta)+T_0\tau+(R'-mR-T_0)\cdot\kappa_{\rm max}
	<\log\rho_{\rm b}.
\]
For that, recall that \eqref{eqconditionsm01} implies $T_0\tau< mR|\widehat\chi|\delta$. Hence,
\begin{align*}
mR\chi+	mR\widehat\chi(1-\delta)&+T_0\tau+(R'-mR-T_0)\kappa_{\rm max}\\
	\text{\tiny by \eqref{eqconditionsm01}}\quad
	&\leq mR\chi+mR\widehat\chi(1-2\delta)+(R'-mR-1)\kappa_{\rm max}\\
	\text{\tiny by \eqref{fixm} }\quad
	&\leq mR\chi+mR\widehat\chi(1-2\delta)
		+mR\frac{|\widehat\chi|(1-3\delta)}{\kappa_{\rm max}} \kappa_{\rm max}\\
    	&\leq  -mR|\widehat\chi|\delta \\
	\text{\tiny by \eqref{eqconditionsm01} }\quad
	&<\log\rho_{\rm b}.
\end{align*}
This proves the claim.
\end{proof}

Claim \ref{cl.nmT0} ensures that the sequence of $\cs$-complete subcubes $\big\{\bfK_{\ba}^{(i)}\}_{i=0}^{\ell_0}$ of $\bfK_{\ba}$ given by Proposition~\ref{pl.completetour} is well defined. We finally are ready to determine the time $ST_{\rm b}(\ba)$ corresponding to the iterates in the blender-horseshoe. For that we fix an auxiliary point $y=y_\ba\in \bfK_{\ba}^{(\ell_0)}$. One should not expect this point to belong to the horseshoe we construct, it will just help us describe a part of trajectories of the points that will form the horseshoe. Let  
\begin{equation}\label{defkappay}
	\kappa_i(y)\eqdef \log\,\lVert D^\c f^i(y)\rVert.
\end{equation}
By construction, the following inequalities hold
\begin{equation}\label{properties-y}
	\begin{split}
	-iR|\widehat\chi|(1+\delta)\leq \kappa_{iR}(y)-iR\chi\leq -i|\widehat\chi|(1-\delta) 
	\quad&\text{ for all }  i=1,\ldots, m,\\
	S\kappa_{\rm min}\leq \kappa_{mR + T_0 + (i+1)S}(y)-\kappa_{mR +T_0 +iS}(y)\leq S\kappa_{\rm max}
	\quad&\text{ for all }i=0,\ldots,  \ell(\ba).
\end{split}\end{equation}
The first inequality follows from the iterates spent in the sets of the $\csu$-horseshoe $\bfH$ (recall the hypothesis \eqref{finitetimeLyapexp}) and the second one from the iterates in the blender region (recall \eqref{eq3}). For $j=0,\ldots,  \ell(\ba)$, consider now the auxiliary numbers
\begin{equation}\label{defTpm}
\begin{split}
	\omega^-(j)
	&\eqdef \kappa_{mR+T_0+jS}(y)+\big(R'-mR-jS-2T_0\big)\big(\chi+\widehat\chi(1+\delta)\big)-T_0\,\tau,\\
	\omega^+(j)
	&\eqdef \kappa_{mR+T_0+jS}(y)+\big(R'-mR-jS-2T_0\big)\big(\chi+\widehat\chi(1-\delta)\big)+T_0\,\tau.
\end{split}
\end{equation}

The meaning of those numbers is as follows. Assume a point follows the trajectory of $y$ up to the time $mR + T_0 + jS$, that is it first stays in $\bfH$ for $m$ full cycles, then goes to the blender-horseshoe for $j$ full cycles, and hence  spends there a total time $jS$, with approximately the same blender center expansion as the trajectory of $y$ has. Assume then that from this point on its trajectory comes back to $\bfH$ as fast as possible and stays in $\bfH$ until time $R'$. Then $\omega^+(j)$ and $\omega^-(j)$ will be the upper and lower bounds for the Lyapunov exponent at time $R'$. The exact calculation will be done in the proof of Claim \ref{clsubclaim-allx}.
	
\begin{claim}\label{clc.find-time-in-blender}
	There exists an integer $T_{\rm b}(\ba)\in\bN$,
\begin{equation}\label{propTba}
	1\leq ST_{\rm b}(\ba)<R'-mR-2T_0 \le \ell_0,
\end{equation}	 
such that
\begin{equation}\label{limits-Tis}
	R'\chi-R'\rho|\widehat\chi|(1+\delta/3)
	\le \omega^-(T_{\rm b}(\ba))
	< \omega^+(T_{\rm b}(\ba))
	\le R'\chi-R'\rho|\widehat\chi|(1-\delta/3).
\end{equation}
\end{claim}

\begin{proof}
	The proof of this claim is split into several parts.
\begin{subclaim}\label{scl:increasing}
	 $\omega^+(j)$ increases with $j$. 
\end{subclaim}

\begin{proof}
By \eqref{properties-y} and \eqref{defTpm}, one has 
\begin{equation}\label{eq:growth-rate-of-T-plus}
	0
	<\kappa_{\rm min}-\big(\chi+\widehat\chi(1-\delta)\big)
	\leq\frac 1S ( \omega^+(j+1)-\omega^+(j))\leq \kappa_{\rm max}-\big(\chi+\widehat\chi(1-\delta)\big),
\end{equation}
proving that $\omega^+(j)$ is increasing.
\end{proof}

We continue by estimating $\omega^+(j)$ for $j=0$ and $j=j_{\rm max}\eqdef\lfloor(R'-mR-2T_0)/S\rfloor$. Note that $j_{\rm max}\leq\ell_0(\ba)$ due to Claim \ref{cl.nmT0}.

\begin{subclaim}\label{subclaim616}
$\omega^+(0)\leq   R'\chi-R'\rho|\widehat\chi|(1+\delta/3)$.
\end{subclaim}

\begin{proof}	
Note that
\[\begin{split}
\omega^+(0)
&=\kappa_{mR+T_0}(y)+
\big(R'-mR-2T_0\big)\big(\chi+\widehat\chi(1-\delta)\big)+T_0\,\tau
	\\
	\text{\tiny by \eqref{finitetimeLyapexp}}\quad
	&\leq mR\big(\chi+\widehat\chi(1-\delta)\big)+T_0\,\tau+
	\big(R'-mR-2T_0\big)\big(\chi+\widehat\chi(1-\delta)\big)+T_0\,\tau\\
	&= 2T_0\tau+(R'-2T_0)\big(\chi+\widehat\chi(1-\delta)\big)\\
	&\leq 6T_0\tau+R'\chi-R'|\widehat\chi|(1-\delta).
\end{split}\]
Check that \eqref{eqconditionsm01bis} implies
\[
	6T_0\tau - R'|\widehat\chi|(1-\delta)
	\le -R'\rho|\widehat\chi|(1+\delta/3)
\]
and hence Suclaim \ref{subclaim616} follows.
\end{proof}
Note that Subclaim \ref{subclaim616} already implies that, to get \eqref{limits-Tis}, we need $T_{\rm b}(\ba)\ge1$. 

\begin{subclaim} \label{sc.T+nmT0}
$\omega^+(j_{\rm max})>R'(\chi-\rho|\widehat\chi|(1-\delta/3)).$
\end{subclaim}

\begin{proof}
Check that $j_{\rm max}S=R'-mR-2T_0-s$ for some $s\in\{0,\ldots,S-1\}$.
Note that by \eqref{defTpm}
\[\begin{split}
	\omega^+(j_{\rm max})
	&=\kappa_{mR+T_0+j_{\rm max}S}(y)+(R'-mR-j_{\rm max}S-2T_0)(\chi+\widehat\chi(1-\delta))+T_0\,\tau\\
	\text{\tiny by \eqref{finitetimeLyapexp} and \eqref{defNormf}}\quad
	&\geq mR\big(\chi+\widehat\chi(1+\delta)\big)-T_0\,\tau+(R'-mR-2T_0)\, \kappa_{\rm min}-S\tau\\
	&\phantom{\ge\kappa_{R'-T_0}(y)}+S(\chi+\widehat\chi(1-\delta))+T_0\,\tau\\
	&\geq mR\big(\chi+\widehat\chi(1+\delta)\big)+(R'-mR)\, \kappa_{\rm min}-2T_0\tau\\
	&\phantom{\ge\kappa_{R'-T_0}(y)}-S\tau+S(\chi+\widehat\chi(1-\delta))\\
	\text{\tiny by \eqref{crazyS}}\quad
   	&\geq mR\big(\chi-|\widehat\chi|(1+2\delta)\big)+(R'-mR)\, \kappa_{\rm min}\\
  	\text{\tiny by \eqref{fixm}} \quad
	&\geq R'\chi-mR|\widehat\chi|(1+2\delta)+mR|\widehat\chi|\frac{(1-3\delta)\kappa_{\rm min}}{\kappa_{\rm max}}\\
   	&=R'\chi+mR|\widehat\chi| \bigg(\frac{(1-3\delta)\kappa_{\rm min}}{\kappa_{\rm max}}-(1+2\delta)\bigg).
\end{split}\]

Hence to prove the subclaim it is suffices to see that
\begin{equation}\label{todoso}
	mR|\widehat\chi|\, \bigg(\frac{(1-3\delta)\kappa_{\rm min}}{\kappa_{\rm max}}-(1+2\delta)\bigg)
	>-R'\rho|\widehat\chi|(1-\delta/3).
\end{equation}
To verify \eqref{todoso}, check that 
\begin{align*}
	&mR|\widehat\chi|\, \bigg(\frac{(1-3\delta)\kappa_{\rm min}}{\kappa_{\rm max}}-(1+2\delta)\bigg)
		+R'\cdot\rho|\widehat\chi|(1-\delta/3)\\
	\text{\tiny by \eqref{fixm}}\quad
	&\geq mR|\widehat\chi|\, \bigg(\frac{(1-3\delta)\kappa_{\rm min}}{\kappa_{\rm max}}-(1+2\delta)\bigg)
		+ mR\frac{|\widehat\chi|(1-3\delta)+\kappa_{\rm max}}{\kappa_{\rm max}}\cdot\rho|\widehat\chi|(1-\delta/3)\\
	&=\frac{mR|\widehat\chi|}{\kappa_{\rm max}}\bigg((1-3\delta)\kappa_{\rm min}-(1+2\delta)\kappa_{\rm max}
		+\big(\kappa_{\rm max}+|\widehat\chi|(1-3\delta)\big) \rho(1-\delta/3)\bigg)\\
	&>\frac{mR|\widehat\chi|}{\kappa_{\rm max}}\bigg((1-3\delta)\kappa_{\rm min}-(1+2\delta)\kappa_{\rm max}+\kappa_{\rm max} \rho(1-\delta/3)\bigg)
	\\
	\text{\tiny by \eqref{fixgamma}}\quad
	&\geq \frac{mR|\widehat\chi|}{\kappa_{\rm max}}\bigg((1-3\delta)\kappa_{\rm min}-(1+2\delta)\kappa_{\rm max}
		+(\kappa_{\rm max}-\kappa_{\rm min}/2)(1-\delta/3)\bigg)\\
	&> \frac{mR|\widehat\chi|}{\kappa_{\rm max}}\bigg( \kappa_{\rm min}(\frac{1}{2}-3\delta)- \kappa_{\rm max}\frac{7\delta}{3}\bigg)\\
\text{\tiny by \eqref{eq:choice-of-delta} and $\delta<\delta_0$}\quad	
&> \frac{mR|\widehat\chi|}{\kappa_{\rm max}}\bigg( \kappa_{\rm min}\frac{1}{4}- \kappa_{\rm max}\frac{7\delta}{3}\bigg)\\
	\text{\tiny by \eqref{eq:choice-of-delta}}\quad
	&>0.
\end{align*}
This proves that \eqref{todoso} and hence the subclaim.
\end{proof}

By Subclaim \ref{scl:increasing}, $\omega^+(j)$ is increasing with $j$. Hence, Subclaim~\ref{sc.T+nmT0} implies that there exists $1\le j<j_{\rm max}$  satisfying the last inequality in \eqref{limits-Tis}. Choose $T_{\rm b}(\ba)\in\bN$ so that
\begin{equation}\label{thisequation}
	\omega^+(T_{\rm b}(\ba))
	\leq R'\chi-R'\rho|\widehat\chi|(1-\delta/3)
	< \omega^+\big(T_{\rm b}(\ba)+1\big).
\end{equation}
As observed already, we have  $1\le T_{\rm b}(\ba)<(R'-mR-2T_0)/S$. Therefore, to finish the proof of Claim \ref{clc.find-time-in-blender}, it remains to show the following subclaim.

\begin{subclaim}\label{scl.endoftheend}
$\omega^-(T_{\rm b}(\ba))\geq R'\chi-R'\rho|\chi|(1+\delta/3).$
\end{subclaim}

\begin{proof}
By our choice in \eqref{thisequation}, we get 
\[\begin{split} 
	\omega^+(T_{\rm b}(\ba))
	&\leq R'\chi-R'\rho|\widehat\chi|(1-\delta/3)
	<\omega^+\big(T_{\rm b}(\ba)+1\big)\notag\\
	\text{\tiny  by \eqref{eq:growth-rate-of-T-plus} }\quad
	&\leq \omega^+(T_{\rm b}(\ba))+S\kappa_{\rm max}-S\big(\chi+\widehat\chi(1-\delta)\big)\notag,
\end{split}\]
which gives that 
\begin{equation}\label{eq:bounding-interval-for-T-ba}
	R'\chi-R'\rho|\widehat\chi|(1-\delta/3)-S\kappa_{\rm max}+S\big(\chi+\widehat\chi(1-\delta)\big)
	\leq \omega^+(T_{\rm b}(\ba))
	\le R'\chi-R'\rho|\widehat\chi|(1-\delta/3).
\end{equation}
By \eqref{defTpm}, one has 
\begin{equation}\label{eq:gap-of-Tplus-Tminus}
	\begin{split}
	\omega^+(T_{\rm b}(\ba))-\omega^-(T_{\rm b}(\ba))
	&=2T_0\tau+2\delta(R'-mR-T_{\rm b}(\ba)S-2T_0)|\widehat\chi|
	\\
	&<2T_0\tau+2\delta(R'-mR-1)|\widehat\chi|
	\\
	\text{\tiny by \eqref{fixm}}	
	&\leq 2T_0\tau+2\delta\frac{1-3\delta}{\kappa_{\rm max}} mR|\widehat\chi|^2.
\end{split}\end{equation}
Note that $T_0\ge1$, $\delta\in(0,1)$, and $\max\{|\chi|,|\widehat\chi|\}\le\tau$ imply
\begin{equation}\label{somethingtrivial}
	\chi+\widehat\chi(1-\delta)
	\ge -2\tau.
\end{equation}
We get
\[\begin{split}
	\omega^-(T_{\rm b}(\ba))
	&= -\Big(\omega^+(T_{\rm b}(\ba))-\omega^-(T_{\rm b}(\ba))\Big)+\omega^+(T_{\rm b}(\ba))\\
	\text{\tiny by \eqref{eq:gap-of-Tplus-Tminus}, \eqref{eq:bounding-interval-for-T-ba}}
	&> -\Big(2T_0\tau+2\delta\frac{1-3\delta}{\kappa_{\rm max}} mR|\widehat\chi|^2\Big) 
		+ \Big(R'\chi-R'\rho|\widehat\chi|(1-\delta/3)-S\kappa_{\rm max}+S\big(\chi+\widehat\chi(1-\delta)\big)\Big)\\
	\text{\tiny by \eqref{somethingtrivial}}\quad
	&\ge \Big(-2T_0\tau-2\delta\frac{1-3\delta}{\kappa_{\rm max}} mR|\widehat\chi|^2\Big) 
		+ \Big(R'\chi-R'\rho|\widehat\chi|\cdot(1-\delta/3)-S\kappa_{\rm max}-2S\tau\Big)\\
	&\ge R'\chi -R'\rho|\widehat\chi|\cdot(1+\delta/3) 
		+ \Big(- 2T_0\tau-2S\tau - 2\delta\frac{1-3\delta}{\kappa_{\rm max}} mR|\widehat\chi|^2\Big) 
		+ \Big( R'\rho|\widehat\chi|\cdot\frac23\delta-S\kappa_{\rm max}\Big).
\end{split}\]
Thus, to prove the subclaim, it is enough to check that the sum of latter two brackets is positive.
To see why this is so, observe first that \eqref{eqconditionsm03} together with $m\ge m_0$ implies
\[
	2T_0\tau+2S\tau+S\kappa_{\rm max}
	\le \frac{\delta^2}{\kappa_{\rm max}}mR|\widehat\chi|^2.
\]
Hence, we get
\begin{equation}\label{citefixm}
	0
	< \frac{5\delta^2}{\kappa_{\rm max}}mR|\widehat\chi|^2
	\le  \frac{6\delta^2}{\kappa_{\rm max}}mR|\widehat\chi|^2-2T_0\tau-2S\tau-S\kappa_{\rm max}. 
\end{equation}
Now we can check that the sum of the above two brackets is positive.
\[\begin{split}
	&\Big(-4T_0\tau-2\delta\frac{1-3\delta}{\kappa_{\rm max}} mR|\widehat\chi|^2\Big) 
		+ \Big(\frac23R'\cdot\rho|\widehat\chi|\delta-S\kappa_{\rm max}\Big)\\
	\text{\tiny  by \eqref{fixm}}\quad
	&> \frac{2}{3} mR \frac{|\widehat\chi|(1-3\delta)
		+\kappa_{\rm max}}{\kappa_{\rm max}}\cdot \rho|\widehat\chi|\delta
		-2\delta\frac{(1-3\delta)}{\kappa_{\rm max}} mR|\widehat\chi|^2
		-2T_0\tau-2S\tau-S\kappa_{\rm max} \\
		&>\frac{2}{3} mR \cdot \rho|\widehat\chi|\delta
		-\frac{2\delta}{\kappa_{\rm max}} mR|\widehat\chi|^2
		+\frac{6\delta^2}{\kappa_{\rm max}} mR|\widehat\chi|^2-2T_0\tau-2S\tau-S\kappa_{\rm max} \\
	\text{\tiny by \eqref{citefixm}}\quad
	&>mR|\widehat\chi|\bigg(\frac{2}{3}\rho \delta 
			-\frac{2\delta}{\kappa_{\rm max}} |\widehat\chi|\bigg) \\
	&= 2 \delta mR |\widehat\chi|\bigg(\frac{1}{3}\rho
			-\frac{1}{\kappa_{\rm max}} |\widehat\chi|\bigg)\\
	\text{\tiny by \eqref{defchi0} and \eqref{fixgamma}}\quad
	&>0.
\end{split}\]
This proves the subclaim.
\end{proof}
Subclaim \ref{scl.endoftheend} together with \eqref{thisequation} finishes the proof of Claim \ref{clc.find-time-in-blender}.	
\end{proof}

\begin{claim}\label{cl:choiceKaprime}
There exists a $\cs$-complete subcube  $\bfK'_\ba$ of $\bfK_\ba^{(T_{\rm b} (\ba))}$ such that
\begin{equation}\label{propKprime}
	f^k(\bfK'_\ba)\subset
	\begin{cases}
	\bigcup_{i=1}^N\bfK_i&\text{ for }0\leq k\le mR,\\
	\bfC_1\cup\bfC_2&\text{ for }mR+T_0\le k\le mR+T_0+ST_{\rm b}(\ba),\\
	\bigcup_{i=1}^N\bfK_i&\text{ for }mR+2T_0+ST_{\rm b}(\ba)\le k\le  R'.
	\end{cases}
\end{equation}
Moreover, $f^{R'}(\bfK'_\ba)$ is a $\uu$-complete subcube of $C$ and thus satisfies  (1)  in Lemma \ref{l.proexistHprime}.
\end{claim}

\begin{proof}
Consider the $\cs$-complete subcube $\bfK_\ba^{(T_{\rm b} (\ba))}$ of $\bfK_\ba$. It follows from item 3 in  Proposition \ref{pl.completetour} that its image $f^{mR+2T_0+ST_{\rm b}(\ba)}(\bfK_\ba^{(T_{\rm b} (\ba))})$ is a $\uu$-complete subcube of $C$. Since it holds $mR+2T_0+ST_{\rm b}(\ba)<R'$, iterating inside $C$  until completing $R'$ iterates, one finally gets a $\cs$-complete subcube  $\bfK'_\ba$ of $\bfK_\ba^{(T_{\rm b} (\ba))}$ such that $f^{R'}(\bfK'_\ba)$ is a $\uu$-complete subcube of $C$.
\end{proof}

\begin{claim}\label{cl.preli-estimate-finite-center-LE}
Let $K$ be as in \eqref{eq:finite-time-LE-for-1st-horseshoe}, then  assertion (3) in Lemma \ref{l.proexistHprime} holds, that is, for every $x\in \bfK'_\ba$ and $k=1,\ldots, R^\prime$,   
\[
	\|D^\c f^k(x)\|\leq K\cdot e^{k\delta\widehat\chi/2}  .
\]
\end{claim}
	
\begin{proof}
Note that by \eqref{fixm}, \eqref{defchi0}, and $\widehat\chi\in(\chi_0,0)$, one has 
\[
	R'
	< mR\Big(\frac{|\widehat\chi|(1-3\delta)}{\kappa_{\rm max}}+1\Big)+1
	= mR\Big(\frac{|\widehat\chi|(1-3\delta)}{6|\chi_0|}+1\Big)+1
	< mR\Big(\frac{1}{6}+\frac{1}{mR}+1\Big)
	< 2mR,
\]
where for the last inequality we used $1/(mR)<5/6$, which follows from \eqref{eqconditionsm01}. Using the above, let us first check the assertion for $k=1,\ldots, mR-1$. By \eqref{eq:finite-time-LE-for-1st-horseshoe}, \eqref{chiineq}, and \eqref{propKprime},  one has 
\[
	\|D^\c f^k(x)\|
	\leq K\cdot e^{k(\chi+\widehat\chi(1-\delta))}
	\leq K\cdot e^{k\delta\widehat\chi/2}. 
\]
We can then check that the assertion for $mR\le k\le R'$. Indeed, using \eqref{propKprime}, one has 
\[\begin{split}
	\log\|D^\c f^k(x)\|
 &	=\log \|D^\c f^{mR}(x)\| + \log \|D^\c f^{k-mR}(f^{mR}(x))\|\\
	\text{\tiny{by \eqref{finitetimeLyapexp}, \eqref{defNormf}, \eqref{eq3}, \eqref{eq:finite-time-LE-for-1st-horseshoe}}}\quad
	&\leq  mR(\chi+\widehat\chi(1-\delta))
	+2T_0\tau+ST_{\rm b}(\ba)\kappa_{\rm max}+\\
	&\phantom{\leq}+\max\{0,k-mR-ST_{\rm b}(\ba)-2T_0\}(\chi+\widehat\chi(1-\delta))+\log K\\
	\text{\tiny{by \eqref{propTba}}}	\quad
	&\leq  mR(\chi+\widehat\chi(1-\delta))+2T_0\tau+(R'-mR-2T_0)\kappa_{\rm max}+\log K\\
	\text{\tiny by \eqref{fixm}}\quad
	&\leq mR(\chi+\widehat\chi(1-\delta))+2T_0\tau
		+mR\frac{|\widehat\chi|(1-3\delta)}{\kappa_{\rm  max}}\kappa_{\rm max}+\log K\\
	&=mR(\chi+2\delta\widehat\chi)+2T_0\tau+\log K\\
	\text{\tiny by \eqref{eqconditionsm01}
	}\quad
	&<mR(\chi+2\delta\widehat\chi) - mR\widehat\chi\delta
	\leq mR\delta\widehat\chi
\end{split}\] 
Note that \eqref{onehass} and the choice of $R'$ in \eqref{fixm} implies that $R'<2mR$. Hence, we conclude 
\[
	\log \,\|D^\c f^k(x)\|
	\leq R^\prime\delta\widehat\chi/2\\
	\leq k\delta\widehat\chi/2.
\] 
This proves the claim.
\end{proof}

Recall the constant $L=L(K,\delta\widehat\chi/2,\eta)\in\bN$ in Corollary \ref{c.distortion-reference-cube}.

\begin{claim}\label{cl:etaclose}
	It holds $\diam(f^k(\bfK'_\ba))\le\eta$ for every $k=L,\ldots,R'-L$. Similarly, it holds $\diam (f^k(\bfK_\ba^{T_{\rm b}(\ba)})) \leq \eta$ for $k=L,\ldots,mR + T_0 + ST_{\rm b}(\ba)-L$.
\end{claim}	

\proof
We apply Corollary \ref{c.distortion-reference-cube} to the cube  $\bfK'_\ba$ and the iteration time $R'$. First note that the above construction of $\bfK'_\ba$ together with our hypothesis $\diam (C)<\varepsilon_1$ already imply that property (2) in Corollary \ref{c.distortion-reference-cube} holds,
\[
	\max\big\{\diam(\bfK'_\ba),\diam(f^{R'}(\bfK'_\ba))\big\}
	<\varepsilon_1.
\]
By taking the intersection $f^{R'}(\bfK'_\ba)$ with the base $\cW^\ss(\gamma,r)\subset C$, one gets a base of $f^{R'}(\bfK'_\ba)$ which is contained in $f^{R'}(\bfK'_\ba)$, and this verifies property (3) in Corollary  \ref{c.distortion-reference-cube}. Hence, by Claim \ref{cl.preli-estimate-finite-center-LE}, we get also property  (1) in  Corollary \ref{c.distortion-reference-cube}. This finishes the proof of the first part of the assertion, the second part is proven analogously.
\endproof 
	
\begin{claim}\label{clsubclaim-allx}
	Every $x\in \bfK_{\ba}'$ satisfies assertion (2) in Lemma \ref{l.proexistHprime}.
\end{claim}
\begin{proof}
Let $x\in\bfK_\ba'$. 
\[\begin{split}
	\log\|D^\c f^{R'}(x)\|
	&=\log\|D^\c f^{L}(x)\|+ \log\|D^\c f^{mR+T_0+T_{\rm b}(\ba)S-2L}(f^{L}(x))\|\\
	&\phantom{=}
		+\log\|D^\c f^{T_0+L}(f^{mR+T_0+T_{\rm b}(\ba)S-L}(x))\|\\
	&\phantom{=}		
		+\log\|D^\c f^{R'-mR-2T_0-T_{\rm b}(\ba)S}(f^{mR+2T_0+T_{\rm b}(\ba)S}(x))\|\\
	\text{\tiny{by \eqref{defNormf}, \eqref{eq:choice-of-eta}, \eqref{eq:finite-time-LE-for-1st-horseshoe} and Claim \ref{cl:etaclose}}}\quad
	&\leq L\tau
		+(mR+T_0+T_{\rm b}(\ba)S)\frac{\delta}{3}\rho|\widehat\chi|\\
	&\phantom{\leq L\tau}	
		+\log\|D^\c f^{mR+T_0+T_{\rm b}(\ba)S-2L}(f^{L}(y))\|\\
	&\phantom{\leq L\tau}+(T_0+L)\tau+(R'-mR-T_{\rm b}(\ba)S-2T_0)(\chi+\widehat\chi(1-\delta))+\log K\\
	&\leq 3L\tau
		+(mR+T_0+T_{\rm b}(\ba)S)\frac{\delta}{3}\rho|\widehat\chi|
		+\log\|D^\c f^{mR+T_0+T_{\rm b}(\ba)S}(y)\|\\
	&\phantom{\leq L\tau}+(T_0+L)\tau+(R'-mR-T_{\rm b}(\ba)S-2T_0)(\chi+\widehat\chi(1-\delta))+\log K\\
	\text{\tiny{by \eqref{defTpm}}}\quad
	&<4L\tau+R'\frac{\delta}{3}\rho|\widehat\chi|+\omega^+(T_{\rm b}(\ba))+\log K\\
	\text{\tiny by Claim \ref{clc.find-time-in-blender}}\quad
	&\leq 4L\tau+\log K+R'\frac{\delta}{3}\rho|\widehat\chi|+R'\chi+R'\rho\widehat\chi(1-\delta/3)\\
	\text{\tiny by \eqref{eqconditionsm01}}\quad
	&<R'\frac{2\delta}{3}\rho|\widehat\chi|+R'\chi+R'\rho\widehat\chi(1-\delta/3)\\
	&=R'\chi+R'\rho\widehat\chi(1-\delta).
\end{split}\]
The other inequality is analogous.
\end{proof}
This finishes the proof of Lemma \ref{l.proexistHprime}.
\end{proof} 
 
\subsection{End of the proof of Theorem \ref{THEOp.key-estimate-center-Lyapunov-exponent}} 
 
We are now ready to define the $\csu$-horseshoe $\bfH'$ subordinated to $\bfH=(C,\{\bfK_i\}_{i=1}^N,f^R)$.  We start by defining its rectangles. 

Let $K>0$ and $m_0\in\bN$ be as in Lemma \ref{l.proexistHprime}. Given $m\ge m_0$, let $R'=R'(m)$ be as in Lemma \ref{l.proexistHprime}. Consider any $\ba \in \{1, \dots, N\}^m$. By Lemma \ref{l.proexistHprime}, there are a number $T_{\rm b}(\ba)\in\bN$ and a $\cs$-complete subcube $\bfK^\prime_\ba$ of  $\bfK_{\ba}^{(T_{\rm b}(\ba))}$ (and hence a  $\cs$-complete subcube of $\bfK_\ba$) such that 
\begin{itemize}[leftmargin=0.6cm ]
	\item[(1)] $f^{R'}(\bfK'_\ba)$ is a $\uu$-complete subcube of $C$;
	\item[(2)] for every $x\in\bfK_\ba'$, one has 
	\[ R'\rho\widehat\chi(1+\delta)< \log\|D^\c f^{R'}(x)\|-R'\chi\le R'\rho\widehat\chi(1-\delta),\]
	\item[(3)] for every $x\in\bfK_\ba'$ and every $k=1,\ldots,R'$, one has
	\[ \log\|D^\c f^k(x)\|\le K\exp(k\delta\widehat\chi/2).\]
\end{itemize}
Let
\begin{equation}
	\bfH'
	\eqdef (C, \{\bfK_{\ba}'\}_{\ba \in \{1,\ldots,N\}^m}, f^{R'}).
\end{equation}
Our choices of $R'$ in \eqref{fixm},    $m_0$ provided by Lemma  \ref{l.proexistHprime}, and $\xi$ in \eqref{d.defkappa} imply 
\[
	1
	< \frac{R'}{mR}
	\leq \frac{\kappa_{\rm max}+|\widehat\chi|(1-3\delta)}{\kappa_{\rm max}}+\frac{1}{mR}
	\leq 1+\frac{|\widehat\chi|(1-3\delta)}{\kappa_{\rm max}}+\frac{1}{m_0R}
	<1+\frac{|\widehat\chi|}{\kappa_{\rm max}}=1+\xi|\widehat\chi|
\]
proving \eqref{eqchecktimetheorem}.
The proof of Theorem \ref{THEOp.key-estimate-center-Lyapunov-exponent} is now complete.
\qed 

\subsection{Asymptotic distribution of trajectories}\label{sec.remweakstar}
	
For further use, we collect some facts about our construction in this section. In Theorem \ref{THEOp.key-estimate-center-Lyapunov-exponent}, we construct a horseshoe $\bfH'$ such that every trajectory spends the fraction of time $ST_{\rm b}(\ba)/R'$ in the superposition region of the blender-horseshoe (where the logarithmic expansion rate is between $\kappa_{\rm min}$ and $\kappa_{\rm max}$). Ignoring the transition times of length $T_0$, in the remain of time the trajectory stays in $\bfH$ (where the logarithmic contraction rate is roughly $-(|\chi|+|\widehat\chi|)$). Using the Landau notation, we can estimate that
\[
	 \frac{|\widehat\chi| (1-\rho)} {\kappa_{\rm max} + |\chi| + |\widehat\chi|}
	 \Big(1-O(\delta)\Big)-O\Big(\frac{1}{R'}\Big)
	\le \frac{ST_{\rm b}(\ba)}{R'}
	\le \frac{|\widehat\chi| (1-\rho)} {\kappa_{\rm min} + |\chi| + |\widehat\chi|}
	 \Big(1+O(\delta)\Big)+O\Big(\frac{1}{R'}\Big).
\]
Applying Theorem \ref{THEOp.key-estimate-center-Lyapunov-exponent} repeatedly, as described in Remark \ref{remCascade}, as the result for large $k$ every trajectory in $\bfH_k$ spends at a fraction of time $s_k\in(0,1)$ in the blender-horseshoe, where 
\[
	\frac{|\widehat\chi|(1-\rho^k)} {\kappa_{\rm max} +|\chi|+ |\widehat\chi|}
	 \Big(1-O(\delta)\Big)-O\Big(\sum_{i=1}^k\frac{1}{R_i}\Big)
	\le s_k 
	\le \frac{|\widehat\chi|(1-\rho^k)} {\kappa_{\rm min} +|\chi|+ |\widehat\chi|}
	 \Big(1+O(\delta)\Big)+O\Big(\sum_{i=1}^k\frac{1}{R_i}\Big).
\]
The remain of time, except for some fraction $O_k\eqdef O(\sum_{i=1}^k\frac{1}{R_i})$ of iterations,  the trajectory stays close to $\bfH$. Note that, by choosing the sequence $(m_k)_k$ growing sufficiently fast, the term $O_k$ can be made arbitrarily small. 

\section{$\bar d$- and Feldman-Katok distances}\label{sedbarFK}

In this section we introduce the topologies which are the main ingredients for our analysis. This will be done for a shift space and also in the framework of a general continuous map on a compact metric space $T\colon X\to X$. We continue to use the corresponding notations of the spaces of invariant and of ergodic probability measures. Besides the weak$\ast$ topology,  in this section we define further topologies and discuss their properties. The results deal with dependence of entropy on these distances which we import from \cite{KwiLac:}. The main consequence for the maps considered in this paper is the continuity of entropy in the Feldman-Katok distance, see Corollary \ref{c.continuity-of-entropy-under-FK-for PH-1D}. 

To define distances between measures, we use generic points. Given $x\in X$, consider the Borel probability measure
\[
	m(x,n)
	\eqdef \frac1n\sum_{k=0}^{n-1}\delta_{T^k(x)},
\]
where $\delta_y$ denotes the Dirac measure at $y$. A measure $\mu$ is \emph{generated} by $x$ if it is the weak$\ast$-limit of some subsequence of $(m(x,n))_n$. Recall that, by Bogolyubov-Krylov argument, any measure generated this way is a $T$-invariant Borel probability measure. A point $x$ is a \emph{generic point} of a measure $\mu\in\cM(T)$ if $\mu$ is the only measure generated by it, and we denote by $G(\mu)$ the set of all such points. Recall that $G(\mu)$ is nonempty if $\mu$ is ergodic.

\subsection{$\bar d$-distance}\label{secdbardistance}

In this subsection, consider a finite alphabet $\cA$ and the sequence space $\cA^\bZ$ whose elements we denote by $\ua=(\ldots,a_{-1}|a_0,a_1,\ldots)$. A \emph{word} over $\cA$ is a finite sequence of symbols in $\cA$ and its \emph{length} is the number of symbols, denoted by $|\cdot|_\cA$. Denote by $\cA^+\eqdef\bigcup_{n\in\bN}\cA^n$ the set of all words over the alphabet $\cA$. 

Given two $n$-words $(a_0,\ldots, a_{n-1})$ and $(b_0,\ldots,b_{n-1})$ over $\cA$, define their \emph{Hamming distance} by 
\begin{equation}\label{defHamming}
	\bar d_n\big((a_0,\ldots, a_{n-1}),(b_0,\ldots,b_{n-1})\big)
	\eqdef \frac1n\card\{j\in\{0,\ldots,n-1\}\colon a_j\ne b_j\}.
\end{equation}
For two sequence $\ua,\ub\in\cA^\bZ$, let
\[
	\bar d(\ua,\ub)
	\eqdef \limsup_{n\to\infty}\bar d_n\big((a_0,\ldots,a_{n-1}),(b_0,\ldots,b_{n-1})\big). 
\]
Note that this defines a pseudo-metric on $\cA^\bZ$.

We consider the left shift $\sigma_{\cA}\colon\cA^\bZ\to\cA^\bZ$. Let us now recall the definition of $\bar d$-distance on $\cM_{\rm erg}(\sigma_\cA)$. 

\begin{definition}[$\bar d$-distance between ergodic measures]\label{def:dconmea}
Given $\nu,\nu'\in\cM_{\rm erg}(\sigma_\cA)$ and $\ua\in\cA^\bZ$, the number
\[
	\inf\{\bar d(\ua,\ub)\colon\ub\in G(\nu')\}
\]
is equal to some constant $\nu$-almost surely; and we call this constant the \emph{$\bar d$-distance} between $\nu$ and $\nu'$, denoted by $\bar d(\nu,\nu')$. 
\end{definition}

See \cite[Section I.9]{Shi:96} for an equivalent definition of $\bar d$-distance.
The following fact is an immediate consequence of Definition \ref{def:dconmea}, see also \cite[Section I.9]{Shi:96}.

\begin{fact}\label{factdbar}
	For every $\varepsilon>0$ and $\nu_1,\nu_2\in\cM_{\rm erg}(\sigma_\cA)$ satisfying $\bar d(\nu_1,\nu_2)<\varepsilon$, there exist a $\nu_1$-generic point $\ua$ and a $\nu_2$-generic point $\ub$ such that $\bar d(\ua,\ub)<\varepsilon$.
\end{fact} 

\begin{remark}
	Note that the $\bar d$-distance on $\cA^\bZ$ is only a pseudo-metric. However, the $\bar d$-distance on $\cM_{\rm erg}(\sigma_\cA)$ is a metric. 
\end{remark}

As observed in the last paragraph in \cite[Section 2]{Orn:73}, one has the following fact. 

\begin{remark}[$\bar d$ \emph{versus} weak$\ast$]\label{r.space-of-Bernoulli-measures}
	On the simplex of all Bernoulli measures on $\cA^\bZ$, the $\bar d$-topology coincides with the weak$\ast$ topology. In particular, the simplex of all Bernoulli measures on $\cA^\bZ$ under the $\bar d$-topology is path-connected. 
\end{remark}

\begin{theorem}[{\cite[Theorems I.9.15 and I.9.16]{Shi:96}}]\label{thm.convergence-of-ergodic-measures-in-d-bar-metric}
Let $(\mu_n)_{n\in\mathbb{N}}$ be a sequence of ergodic measures in $\cM_{\rm erg}(\sigma_{\cA})$ which is a Cauchy sequence under the $\bar d$-metric. Then there exists an ergodic measure $\mu\in\cM_{\rm erg}(\sigma_\cA)$ such that 
\begin{itemize}[leftmargin=0.6cm ]
	\item $\lim_{n\rightarrow\infty}\bar d(\mu_n,\mu)=0;$
	\item $\lim_{n\rightarrow\infty} h_{\mu_n}(\sigma_\cA)=h_\mu(\sigma_\cA).$
\end{itemize} 
\end{theorem}

The following result is a consequence of \cite[Lemma I.9.11 items (a) and (f)]{Shi:96}.

\begin{lemma}
Consider  two Bernoulli measures $\mu_1,\mu_2$ of  $(\cA^\bZ,\sigma_\cA)$ given by probability vectors $(p^1_a)_{a\in\cA}$ and $(p^2_a)_{a\in\cA}$ respectively, then 
\[\bar d(\mu_1,\mu_2)=\frac{1}{2}\sum_{a\in\cA}|p^1_a-p^2_a|. \]
\end{lemma}

\begin{remark}[$\bar d$ \emph{versus} $\bar f$]
In view of the study of the Feldman-Katok pseudometric (see Definition \ref{defFK}) which we introduced in Section \ref{secFKdistance}, one may think that instead of the $\bar d$-pseudometric it would be more natural to consider the $\bar f$-pseudometric%
\footnote{The $\bar f$-pseudometric between sequences $\ua,\ub\in\cA^\bZ$ is essentially defined by replacing the Hamming distance $\bar d_n$ in \eqref{defHamming} by the weaker \emph{edit metric} $\bar f_n$ given by $k/n$, where $k$ is the minimal number of symbols which need to be removed from each $n$-word so that the words formed by the remaining symbols coincide. See \cite{ORW:82} for full details.}
 on the symbolic space $\cA^\bZ$. We use, however, the $\bar d$-pseudometric, for the following reason. Sequence spaces will be one of our main modeling tools. Starting with an alphabet $\cA$, we consider the symbolic spaces with alphabets $\cA_k = \cA^{M_k}$ for some increasing sequence of natural numbers $(M_k)_k$. By the natural identification of elements of $\cA_k^\bZ$ and $\cA^\bZ$ (we will denote this identification by $\Sub_k^0$), for a sequence $(\ua^{(n)})_n\subset \cA_k^\bZ$ we have
	\[
	\ua^{(n)}\to \ua \quad\text{ \big(in the space }(\cA_k^\bZ, \bar d)\text{\big)} 
	\quad\Longleftrightarrow\quad
	 \Sub_k^0 (\ua^{(n)}) \to \Sub_k^0(\ua) \quad\text{\big(in the space } (\cA^\bZ, \bar d) \text{\big)}.
	\]
It is straightforward to check that a similar statement for the $\bar f$-topology is not true.  
\end{remark}

\subsection{Feldman-Katok pseudometric}\label{secFKdistance}

In this subsection, we consider a compact metric space $(X,\rho)$ and a continuous map $T\colon X\to X$ and we want to introduce a concept analogous to the $\bar d$-distance (which can only be defined on the symbolic spaces and measures living there). The object we define, following \cite{KwiLac:}, resembles rather the $\bar f$-distance than the $\bar d$-distance. Like the $\bar d$-distance, we define it in two settings: we define the Feldman-Katok pseudometric on general sequences of points (in particular, on orbit) from $X$, and the Feldman-Katok metric on the space of ergodic measures $\cM_{\rm erg}(f)$. As in \cite{KwiLac:}, by a slight abuse of notation, we use the same notation $\bar F_\FK(\cdot,\cdot)$. 

\begin{definition}[Matches and the Feldman-Katok pseudometric]\label{defFK}
Consider two sequences $\ux=(x_j)_{j\in\bN_0}$ and $\uy=(y_j)_{j\in\bN_0}$ in $X^{\bN_0}$. Given $n\in\bN$ and $\delta>0$, an \emph{$(n,\delta)$-match} of $\ux,\uy$ is an order preserving bijection $\theta\colon \cD(\theta)\to\cR(\theta)$ such that
\begin{itemize}[leftmargin=0.6cm ]
	\item $\cD(\theta),\cR(\theta)\subset\{0,1,\ldots,n-1\}$;
	\item    for every $i\in\cD(\theta)$, it holds $\rho(x_i,y_{\theta(i)})<\delta$.
\end{itemize}  
In this case, we also call $(x_i)_{i\in\cD(\theta)}$ an \emph{$(n,\delta)$-match of} $(y_{\theta(i)})_{i\in\cD(\theta)}$.

Given an $(n,\delta)$-match $\theta$, we define it quality as 
\[q(\theta)\eqdef\frac{\card \cD(\theta)}{n}. \]
Let
\[
	\bar f_{n,\delta}(\ux,\uy)
	\eqdef \min\big\{1-q(\theta)\colon \theta\text{ is an }(n,\delta)\text{-match of }\ux\text{ with }\uy\big\}.
\]
The \emph{$\bar f_\delta$-distance} between $\ux$ and $\uy$ is defined by
\[
	\bar f_\delta(\ux,\uy)
	\eqdef \limsup_{n\to\infty}\bar f_{n,\delta}(\ux,\uy).
\]
The \emph{Feldman-Katok distance} on $X^{\bN_0}$ is defined by
\[
	\bar F_\FK(\ux,\uy)
	\eqdef \inf\big\{\delta\colon \bar f_\delta(\ux,\uy)<\delta\big\}.
\]

Having defined the distance on the space of all infinite sequences of points from $X$, we define a distance on $X$ by
\[
	\bar F_\FK(x,y)
	\eqdef \bar F_\FK\big((x,T(x),\ldots,T^{n-1}(x),\ldots),(y,T(y),\ldots,T^{n-1}(y),\ldots)\big),
\]
that is, the Feldman-Katok distance between two points in $X$ is defined as the Feldman-Katok distance between their forward orbits.
\end{definition}

Analogously to Definition \ref{def:dconmea}, we define the $\bar F_\FK$-distance between ergodic measures.
Recall that $G(\nu)$ denotes the set of generic points of a $T$-invariant measure $\nu$.

\begin{definition}[$\bar F_\FK$-distance between ergodic measures]\label{defFKerg}
Given two ergodic measures $\mu,\nu\in\cM_{\rm erg}(f)$, their $\bar F_\FK$-distance is defined by 
\[
	\bar F_\FK(\mu,\nu)
	\eqdef\inf\big\{\bar F_\FK(x,y) \colon x\in G(\mu), ~y\in G(\nu)\big\}.  
\]
\end{definition}

\begin{remark}
	Note that the $\bar F_\FK$-distances on $X^{\bN_0}$ and on $X$ are only pseudometrics. However, the $\bar F_\FK$-distance on $\cM_{\rm erg}(T)$ is a metric. 
\end{remark}

We now provide a powerful criterion from \cite{KwiLac:} for convergence of a sequence of ergodic measures in the $\bar F_\FK$-metric. 

\begin{remark}[Criterion for $\bar F_\FK$-convergence]\label{remCauchyFK}
A sequence of measures $(\mu_n)_n\subset\cM(f)$ \emph{$\bar F_\FK$-converges} to some $\mu\in\cM(f)$ if
there exist a  $\mu$-generic orbit $\ux$ and a $\mu_n$-generic orbit $\ux^{(n)}$ for each $n\in\bN$ such that
\[\lim_{n\rightarrow+\infty} \bar F_\FK(\ux^{(n)},\ux)=0.\] 
\end{remark}

\begin{remark}[Completeness]
In general, it is not clear whether the $\bar F_\FK$-distance induces a complete topology on $X$. But, on the space of ergodic measures, the $\bar F_\FK$-distance induces a very strong {\it complete} topology: a Feldman-Katok Cauchy sequence of ergodic measures converges to an ergodic measure, see Theorem \ref{thepro:ergentcon}. Moreover, it is a consequence of \cite[Theorem 33 and Corollary 34]{KwiLac:} that to prove $\bar F_\FK$-con\-vergence of a sequence of ergodic measures it is enough to prove that there are
sequences of generic orbits that form a $\bar F_\FK$-Cauchy sequence.
\end{remark}

Let us present some consequences of $\bar F_\FK$-convergence of measures. 

\begin{theorem}[{\cite[Theorems 33, 40, and 41]{KwiLac:}}]\label{thepro:ergentcon}
Let $T\colon X\to X$ be a continuous map on a compact metric space $X$. Let $(\mu_n)_n\subset\cM_{\rm erg}(T)$ be a $\bar F_\FK$-Cauchy sequence. Then there exists a unique  ergodic measure $\mu\in\cM_{\rm erg}(T)$ such that  
\begin{itemize}[leftmargin=0.6cm ]
	\item  $\lim_{n\to\infty}\bar F_{\FK}(\mu_n,\mu)=0$,
	\item $\mu_n$ converges to $\mu$ in the weak$\ast$-topology.
\end{itemize}
Moreover, 
\[
h_\mu(f)
\le \liminf_{n\to\infty}h_{\mu_n}(T).
\]	
\end{theorem}

We close this section returning to the partially hyperbolic context of this paper. 

\begin{remark}\label{rem:entropyexpansive}
Note that under the hypotheses of entropy expansiveness, the entropy map $\mu\mapsto h_\mu(f)$ is upper semi-continuous in the weak$\ast$ topology \cite{Bow:72}. Hence, Theorem \ref{thepro:ergentcon} implies convergence in entropy for entropy expansive maps. It has been  shown that $f\in\PH_{\c=1}^1(M)$ is entropy expansive \cite{LiViYa:13,DiFiPaVi:12}.
\end{remark}

Remark \ref{rem:entropyexpansive} and Theorem \ref{thepro:ergentcon} together imply the following.

\begin{corollary}\label{c.continuity-of-entropy-under-FK-for PH-1D}
	Let $f\in\PH_{\c=1}^1(M)$, and $(\mu_n)_n\subset\cM_{\rm erg}(f)$ be a sequence of ergodic measures converging to $\mu$ under the $\bar F_{\FK}$-metric. Then 
	\[h_\mu(f)=\lim_{n\rightarrow+\infty}h_{\mu_n}(f). \] 
\end{corollary}

\subsection{Factors of horseshoes and $\bar d$-/$\bar F_{\rm FK}$-distances}

Let us return to considering the factor $\Pi$ associated to a $\csu$-horseshoe $\bfH=(C,\{\bfK_i\}_{i=1}^N,f^R)$ introduced in Section \ref{ss.symbolic-horseshoe}. As in this section, we deal with several metrics, for clearness, in what is below we write $\dist_M$ for the metric on the manifold $M$.
Consider the ``modulus of continuity''-function associated to $\Pi$ given by
\begin{equation}\label{eqModcon}
	\Mod\colon\bR_{>0}\to\bN,
\end{equation}
where $\Mod(\varepsilon)$ is the minimal $k\in\bN$ such that for every $\ua,\ub\in\cA^\bZ$ satisfying that $a_i=b_i$ for all  $i=-k,\ldots,k$,  we have $\dist_M(\Pi(\ua,s),\Pi(\ub,s))<\varepsilon$, for all $s=0,\ldots,R-1$.

The following lemma deals with the regularity of $\Pi$ relative to the $\bar d$-  and the $\bar F_{\rm FK}$-pseudometric, respectively.

\begin{lemma} \label{suspmetr}
	For every $\varepsilon>0$ and $\ua,\ub\in\cA^\bZ$ satisfying 
\[
	\bar d(\ua,\ub)< \frac{\varepsilon}{2\Mod(\varepsilon)+1},
\]
one has 
\[
	\bar F_\FK\big(\Pi(\ua,0), \Pi(\ub,0)\big) \leq \varepsilon .
\]
\end{lemma}

\begin{proof}
Given $\ua,\ub$ as in the statement and let $\delta\eqdef\bar d(\ua,\ub)$. By definition, $\delta$  is  the upper density of the set $\{j\in\bN_0\colon a_j\neq b_j\}$. Given $n\in\mathbb{N}$,  the upper density of the set 
$$
	D_n(\ua, \ub) \eqdef\{j\in \bN_0\colon \exists k\in \{-n,\ldots,n\}, a_{j+k}\neq b_{j+k}\}
$$ 
is asymptotically at most $(2n+1)\delta$. 

Substituting $n\eqdef \Mod(\varepsilon)$ and applying the definition of $\Mod$, we see that if $j\notin D_n(\ua,\ub)$ then 
\[
	\dist_M\big(\Pi(\sigma_\cA^j (\ua),s),\Pi(\sigma^j(\ub),s)\big) 
	\leq \varepsilon, \quad \mbox{for every $s\in\{0,\ldots, R-1 \}$.} 
\]	 
This gives rise to 
\[
	\dist_M(f^{jR+s}(\Pi(\ua),0)),f^{jR+s}(\Pi((\ub),0)))\leq\varepsilon.
\]
By definition of the $\bar F_{\rm FK}$-pseudometric, one has 
\[ 
	\bar F_\FK\big(\Pi(\ua,0), \Pi(\ub,0)\big)
	\leq (2\Mod(\varepsilon)+1)\delta\le  \varepsilon.
\]
This proves the lemma.
\end{proof}

Given $\nu\in\cM_{\rm erg}(\sigma_{\cA})$, consider the measure $\lambda_{\cA,R,\nu}\in\cM_{\rm erg}(\cS_{\cA,R})$ defined in \eqref{eqtildelambda}. 

\begin{proposition}\label{proepsdel}
	For every $\nu,\nu'\in\cM_{\rm erg}(\sigma_\cA)$ and $\varepsilon>0$ such that
\[
	\bar d(\nu,\nu')
	< \delta\eqdef \frac{\varepsilon}{2\Mod(\varepsilon)+1},
\]
there exist generic points $x,x'$ for the measures $\Pi_\ast\lambda_{\cA,R,\nu}$ and $\Pi_\ast\lambda_{\cA,R,\nu'}$, respectively, such that $\bar F_\FK (x,x') \leq \varepsilon$. In particular, one has 
\[
	\bar F_\FK \big(\Pi_\ast\lambda_{\cA,R,\nu},\Pi_\ast\lambda_{\cA,R,\nu'}\big) 
	< \varepsilon.
\]
\end{proposition} 

\begin{proof}
	It follows from Fact \ref{factdbar} that there are generic sequences $\ua$ and $\ua'$ for $\nu$ and $\nu'$, respectively, with $\bar d(\ua,\ua')<\delta$. Hence, by Lemma \ref{suspmetr}, $\bar F_\FK(\Pi(\ua,0),\Pi(\ua',0))\le \varepsilon$. Notice that $(\ua,0)$ and $(\ua',0)$ are generic points for $\lambda_{\cA,R,\nu}$ and $\lambda_{\cA,R,\nu'}$,   and thus 
	$\Pi(\ua,0)$ and $\Pi(\ua',0)$  are generic points for $\Pi_\ast\lambda_{\cA,R,\nu}$ and $\Pi_\ast\lambda_{\cA,R,\nu'}$, respectively. 
\end{proof}

\section{Ergodic measures with prescribed Lyapunov exponents}\label{sec:existence}

In this section, we construct arcs of ergodic measures with a prescribed center Lyapunov exponent. 
We continue to denote by $\dist_M$ the metric on the manifold $M$.

Recall our choice of $\varepsilon_1$ in Remark \ref{remweaint}. 

\begin{theorem}\label{thm.horseshoe-to-path-of-measure}
	Let $f\in\PH_{\c=1}^1(M)$. Assume that the strong unstable foliation of $f$ is minimal and $f$ has an unstable blender-horseshoe. Then there exist numbers $c_0>0$, $\delta_0\in(0,1)$, and $\chi_0<0$ such that the following holds.  
	
	Let $\delta\in(0,\delta_0)$, $\widehat\chi\in(\chi_0,0)$, and $\chi\in(\chi_{\rm min}(f),0]$. Let  $C=C^\csu(\gamma,r)$ be an $\csu$-cube with $r<\varepsilon_1/3$. Let $\bfH=(C,\{\bfK_i\}_{i=1}^N,f^R)$ be a $\csu$-horseshoe relative to $C$ such that 
	\begin{equation}\label{hypotheseses}
		\widehat\chi(1+\delta)
		\le \frac{1}{R}\log\,\lVert D^\c f^{R}(x)\rVert-\chi
		\le \widehat\chi(1-\delta),
		\quad\text{ for every } x\in\bigcup_{i=1}^N\bfK_i. 
	\end{equation}
	Then there exists a continuous path $\{\mu_t\}_{t\in[0,1]}\subset \cM_{\rm erg}(f)$ such that 
	\begin{itemize}[leftmargin=0.6cm ]
		\item[(1)] $\chi^\c(\mu_t)=\chi$ for any $t\in[0,1]$,
		\item[(2)] $t\mapsto h_{\mu_t}(f)$ is continuous,
		\item[(3)] $h_{\mu_0}(f)=0$, and 
		\[ h_{\mu_1}(f)\geq \big(1+c_0\widehat\chi\big) \cdot \frac{\log N}{R}.\]
	\end{itemize}
\end{theorem}
  
To prove the above theorem, we will apply inductively Theorem~\ref{THEOp.key-estimate-center-Lyapunov-exponent} getting a sequence of $\csu$-horseshoes $(\bfH_k)_{k\in\bN_0}$ as in Remark \ref{remCascade}. Here each horseshoe $\bfH_{k+1}$  has center Lyapunov exponents which are each time closer to the ``target exponent $\chi$'' by some factor. 

In Section \ref{s.cascades-of-horseshoes}, we first derive a cascade of horseshoes $(\bfH_k)_k$ and collect their quantifiers. In Sections \ref{news.symbolic-space} and \ref{s.cascade-of-suspension-space-symbolic}, we fix the corresponding cascade of alphabets and suspension spaces, respectively, where we pay particular attention to the control of added tails.  In Section \ref{secFK-Cauchy}, we prove the existence of Feldman-Katok Cauchy sequences. In Section \ref{ssecUniformFK}, we show that our construction provides uniform $\FK$-convergence across all horseshoes. Finally, in Section \ref{secfinal}, we prove Theorem \ref{thm.horseshoe-to-path-of-measure}.

\subsection{Cascades of horseshoes}\label{s.cascades-of-horseshoes}

Let $f\in\PH^1_{\c=1}(M)$ be as in the assumption of Theorem~\ref{thm.horseshoe-to-path-of-measure}. 
Let $\xi>0$, $\delta_0\in(0,1)$, $\chi_0<0$, and $\rho\in(0,1)$  be the constants provided by Theorem \ref{THEOp.key-estimate-center-Lyapunov-exponent}. 
Let  $C=C^\csu(\gamma,r)$ be an $\csu$-cube with $r<\varepsilon_1/3$.

Let $\delta\in(0,\delta_0)$, $\widehat\chi\in(\chi_0,0)$, and $\chi\in(\chi_{\rm min}(f), 0]$. Let us consider a $\csu$-horseshoe $\bfH=(C,\{\bfK_i\}_{i=1}^N,f^R)$ relative to  $C$ and satisfying \eqref{hypotheseses}. Consider the number $m_0(\bfH)$ provided by Theorem \ref{THEOp.key-estimate-center-Lyapunov-exponent} applied to $\bfH$.

Invoke now Remark \ref{remCascade}. Letting $\bfH_0=\bfH$, consider the cascades of horseshoes $(\bfH_k)_{k\in\bN_0}$, alphabets $(\cA_k)_{k\in\bN_0}$, constants $(K_k)_{k\in\bN_0}$, repeating times $(m_k)_{k\in\bN}$, and tailing times $(\bt_k)_{k\in\bN}$. For each horseshoe $\bfH_k=(C,\{\bfK_\ba'\}_{\ba\in\cA_k},f^{R_k})$, for $k\in\bN$ we get
\begin{itemize}[leftmargin=0.6cm ]
\item $\cA_0=\cA$ and $\cA_{k}=\cA_{k-1}^{m_{k}}$;
\item $m_{k} R_{k-1}\le R_{k}< \big(1+\xi\cdot \rho^{k}|\widehat\chi|\big) m_{k} R_{k-1}$ and hence 
\begin{equation}\label{eqdefRk}
	\bt_{k}<\xi |\widehat\chi|\cdot \rho^{k} \cdot m_{k} R_{k-1},
\end{equation}
\item for every $\ba\in\cA_k$ and every $x\in \bfK_\ba'$,   
\begin{equation}\label{eq:bound-center-LE-for-kth-horseshoe} 
		- \rho^k|\widehat\chi|(1+\delta)
		\le \frac{1}{R_k} \log\,\lVert D^\c f^{R_k}(x)\rVert- \chi
		\le - \rho^k|\widehat\chi|(1-\delta).
\end{equation}
\item for every $\ba\in\cA_k$, every $x\in \bfK_\ba'$, and every $i=1,\ldots,R_k$,
\begin{equation}
	 \log\|D^\c f^i(x)\|\le K_{k-1}e^{i\delta\rho^{k-1}\widehat\chi/2}.	
\end{equation}  	
\end{itemize}

The following result is Corollary \ref{c.distortion-reference-cube} applied to the horseshoe $\bfH_k$.

\begin{corollary}\label{corinitial}
For every $k\in\bN$, there exists $L_k\in\bN$ such that for every $\ba\in \cA_k$ and every $x,y\in \bfK_\ba$ we have
\[
	\dist_M(f^i(x), f^i(y)) < 2^{-k},
	\quad\text{ for every }\quad
	i=L_{k},\ldots, R_k  - L_{k},
\]
where the number $L_{k}=L(K_{k-1},\delta\rho^{k-1}\widehat\chi/2,2^{-k})$ is as in Corollary \ref{c.distortion-reference-cube} applied to the horseshoe $\bfH_k$, which hence does not depend on $m_k$ and hence on $R_k$.
\end{corollary}

For the following, we assume that the sequence $(m_k)_k$ was chosen to grow sufficiently fast such that 
\begin{equation}\label{extrassumption}
	L_{k}
	\le 2^{-k}m_kR_{k-1}\quad\textrm{for every $k$}.
\end{equation}
 
With the above, letting $c=\xi|\widehat\chi|$, the following assumption is satisfied. 
\begin{assumption}[Control of tail lengths]\label{asstailen}
There exists a constant  $c>0$ such that for the sequences $(m_k)_{k\in\bN}$, $(R_k)_{k\in\bN_0}$, and $(\bt_k)_{k\in\bN}$ the following holds. For every $k\in\bN$ 
\[
	R_{k} =  {m_{k}} R_{k-1} +\bt_{k} ,
	\quad\text{ where }\quad
	0
	\le \bt_{k}
	\le c\cdot \rho^{k} \cdot {m_{k}} R_{k-1} .
\]	
\end{assumption}

\subsection{Cascade of alphabets}\label{news.symbolic-space}

Recall that
\[
	\cA_0\eqdef\cA,\quad
	\cA_k\eqdef\cA_{k-1}^{m_k}.
\]
Note that each $\cA_k$ is a finite collection of $(m_1\cdot m_2\cdots m_k)$-words over the initial alphabet $\cA$. However, we want to look at it as a new abstract alphabet. This point of view comes with naturally associated ``respelling maps'' for any $\ell>k$
\[\begin{split}
	\Sub_\ell^{\ell-1}&\colon\cA_\ell\to\cA_{\ell-1}^{m_\ell},\\
	\Sub_\ell^{k}&\colon\cA_\ell\to\cA_k^{m_{k+1}\cdots m_\ell},\quad 
	\Sub_\ell^{k}\eqdef\Sub_{k+1}^{k}\circ\Sub_{k+2}^{k+1}\circ\cdots\circ\Sub_\ell^{\ell-1}.
\end{split}\]
We extend these maps to bijections between the associated sequence spaces. For that, denote by $(\ldots|a_0^{(\ell)},a_1^{(\ell)},\ldots)$ the elements in $\cA_\ell^\bZ$. Let
\begin{equation}\label{eqSub}\begin{split}
	&\underline\Sub_\ell^k\colon \cA_\ell^\bZ\to \cA_k^\bZ,\\
	&\underline\Sub_\ell^k\big(\ldots|a_0^{(\ell)},a_1^{(\ell)},\ldots\big)
	\eqdef (\ldots|\Sub_\ell^k(a_0^{(\ell)}),\Sub_\ell^k(a_1^{(\ell)}),\ldots).
\end{split}\end{equation}
Note that the map $\underline\Sub_\ell^k$  is a homeomorphism. 

Denote by 
\begin{equation}\label{defsigmak}
	\sigma_k\eqdef\sigma_{\cA_k}\colon\cA_k^\bZ\to\cA_k^\bZ
\end{equation}
the left shift on $\cA_k^\bZ$. 

\begin{remark}\label{remtopconjugate}
The maps $(\cA_k^\bZ,\sigma_k)$ and $(\cA^\bZ,\sigma_\cA^{m_1\cdots m_k})$ are topologically conjugate by $\underline\Sub_k^0$, that is, the following diagram commutes.
\begin{displaymath}
	\xymatrix@C=6em@R=3em{
		\cA_{k}^\bZ\ar[r]^{\sigma_k} \ar[d]_{\underline\cL_k^0} & 
		\cA_k^\bZ \ar[d]^{\underline\cL_k^0} \\
		\cA^\bZ \ar[r]^{\sigma_\cA^{m_1\cdots m_k}}   &  \cA^\bZ }
\end{displaymath}
\end{remark}

\begin{remark}
Recall the definition of $\bar d$-distance in Section \ref{secdbardistance} and denote by $\bar d_k$ the corresponding distance on $\cA_k^\bZ$. Then for every $k\in\bN$ and every $\ua,\ub\in\cA^\bZ$, it holds
\[
	m_1\cdots m_k\cdot \bar d(\ua,\ub)
	\ge \bar d_k\big((\underline\Sub_k^0)^{-1}(\ua),(\underline\Sub_k^0)^{-1}(\ub)\big)
	\ge \bar d (\ua,\ub).
\]
\end{remark}

\subsection{Cascade of suspension spaces}\label{s.cascade-of-suspension-space-symbolic}

Consider now the associated sequence of suspension spaces and the maps
\[
	\cS_k\eqdef\cS_{\cA_k,R_k}, \quad \Phi_k\eqdef\Phi_{\cA_k,R_k}
\]
with base $\cA_k^\bZ$ and constant roof function $R_k$. 

Given $\ell>k$, we now introduce inductively homeomorphisms between the suspension space $\cS_k$ and certain subsets of $\cS_\ell$ which we will call ``strips''. We start with some notations.

\begin{notation}[Representation with varying base]\label{notaddress}
Given $\ell>k$, we call 
\[	
	\vr=(\mathrm r_k,\ldots,\mathrm r_{\ell-1})
	\quad\text{ with }\quad \mathrm r_i\in\{0,\ldots,m_{i+1}-1\}, \quad i=k,\ldots,\ell-1,
\]	 
a \emph{$(k,\ell)$-admissible tuple}. As we only consider admissible tuples, for simplicity we will drop this adjective. In the following, we use the indices ${}_\ell$ and ${}^{k,\vr}$, where $\ell\in\bN$, $k\in\{0,\ldots,\ell-1\}$, and $\vr$ is as above, to ``localize the intermediate floors'' in higher-level suspension spaces. The index ${}_\ell$ indicates that the object is contained in the suspension space $\cS_\ell$. The index ${}^{k,\vr}$ indicates which previous levels are taken into account (from the suspension space $\cS_k$). The length of the tuple $\vr$ indicates the difference of levels. Compare Figure \ref{fig.1}.
\end{notation}

We now proceed with the inductive definition. 

\begin{definition}[$k$-strips and intermediate floors]\label{defintfloors}
Given $k\in\bN$, for every $\mathrm r=0,\ldots, m_{k+1}-1$, let
\[
	P_k^{(k+1,\mathrm r)}\colon\cS_k \to \cS_{k+1},\quad
	P_k^{(k+1,\mathrm r)}(\ua^{(k)},s) 
	\eqdef \Big(\big(\underline\Sub_{k+1}^k\big)^{-1}
		\big(\sigma_k^{-\mathrm r}( \ua^{(k)})\big), s+\mathrm   r\cdot R_k\Big)
\]
(check that for $s\in\{0,\ldots,R_k-1\}$ both arguments are indeed in their canonical presentations in $\cS_k$ and $\cS_{k+1}$, respectively).

Given $\ell>k$ and any $(k,\ell)$-tuple $\vr=(\mathrm r_k,\ldots, \mathrm r_{\ell-1})$, we define inductively the map: 
\begin{equation}\label{eqproj}
	P_{k}^{(\ell,\vr)} \colon\cS_k \to \cS_\ell,\quad
	P_{k}^{(\ell,\vr)} 
	\eqdef P_{\ell-1}^{(\ell,\mathrm r_{\ell-1})} \circ \cdots \circ P_{k}^{(k+1,\mathrm r_k)}.
\end{equation}
The sets $P_k^{(\ell,\vr)}(\cS_k)$ are called \emph{$k$-strips} and $P_k^{(\ell,\vr)}(\cA_k^\bZ\times\{0\})$ their \emph{intermediate floors}.
\end{definition}
\begin{fact}\label{fac:commuting-P-and-Phi}
Given $\ell>k$ and  any $(k,\ell)$-tuple $\vr$, one has 
\[ \Phi_\ell\circ 	P_{k}^{(\ell,\vr)}=	P_{k}^{(\ell,\vr)}\circ \Phi_k. \]
\end{fact}
\begin{fact}\label{facPs}
	It is an immediate consequence of its definition that each map $P_k^{(k+1,\mathrm r)}$, $\mathrm r=0,\ldots, m_{k+1}-1$, is an injection of $\cS_k$ into $\cS_{k+1}$. The images of those maps, for different $\mathrm r$, are pairwise disjoint.
\end{fact}

\begin{notation}\label{not-jr}
Given any $\ell>k$ and any $(k,\ell)$-tuple $\vr=(\mathrm r_k,\ldots, \mathrm r_{\ell-1})$, let
\[
j
= j_k^\ell(\vr) 
\eqdef \mathrm r_k+\sum_{i=k+1}^{\ell-1} m_{k+1} \cdots m_i \mathrm r_i 
\in\Big\{0,\ldots,\sum_{i=k}^{\ell-1}m_{k+1}\cdots m_{i+1}\Big\}
\]
(when $k$ and $\ell$ are clear from the context, then we will omit them in the notation of $j$).
\end{notation}

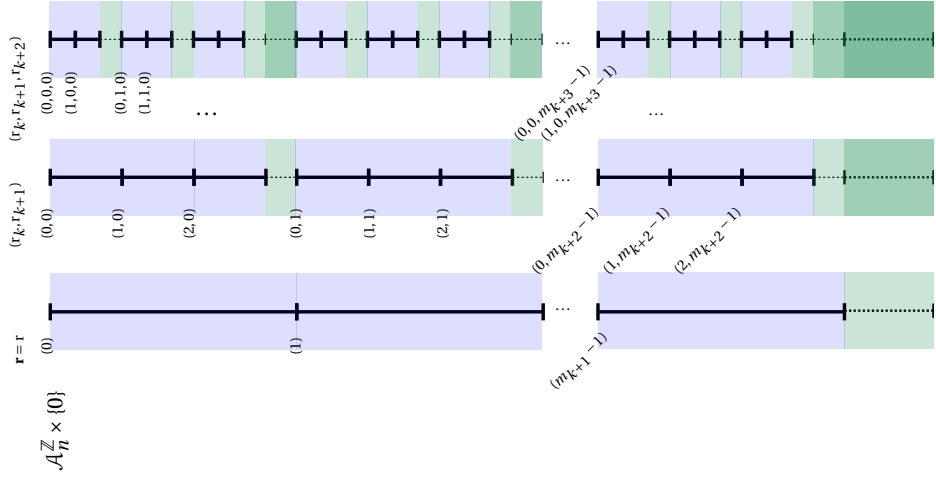
\begin{figure}[h] 
 \begin{overpic}[scale=.43]{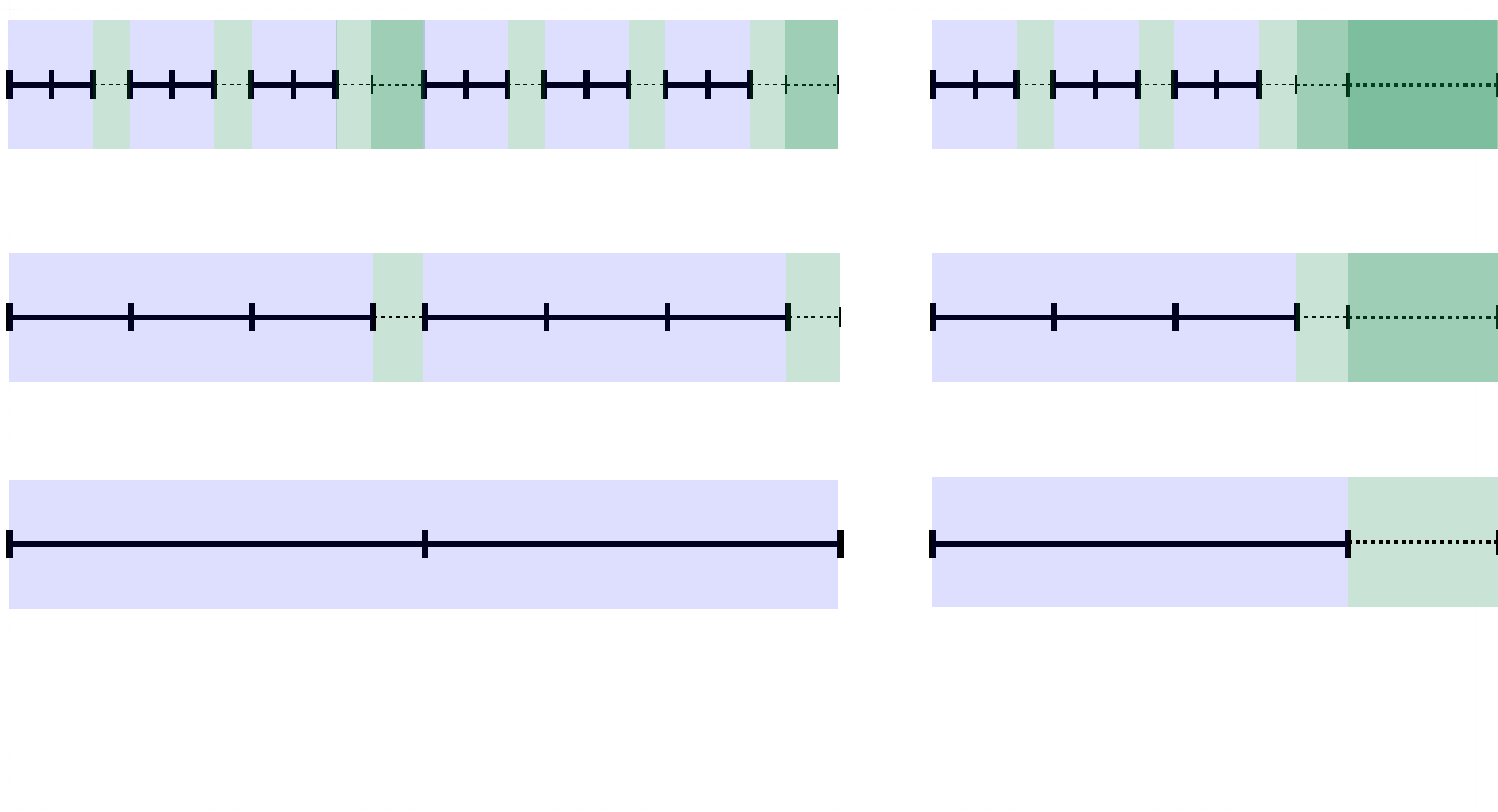}
 	\put(-1,0){ {\rotatebox{90}{\small$\cA_n^\bZ\times\{0\}$}}}
 	\put(-1,26){ {\rotatebox{90}{\tiny$(0,0)$}}}
 	\put(7,26){ {\rotatebox{90}{\tiny$(1,0)$}}}
 	\put(15,26){ {\rotatebox{90}{\tiny$(2,0)$}}}
 	\put(27,26){ {\rotatebox{90}{\tiny$(0,1)$}}}
 	\put(35.3,26){ {\rotatebox{90}{\tiny$(1,1)$}}}
 	\put(43.5,26){ {\rotatebox{90}{\tiny$(2,1)$}}}
 	\put(54,22){ {\rotatebox{45}{\tiny$(0,m_{k+2}-1)$}}}
 	\put(62,22){ {\rotatebox{45}{\tiny$(1,m_{k+2}-1)$}}}
 	\put(70,22){ {\rotatebox{45}{\tiny$(2,m_{k+2}-1)$}}}
 	\put(-1,13){ {\rotatebox{90}{\tiny$(0)$}}}
 	\put(27,13){ {\rotatebox{90}{\tiny$(1)$}}}
 	\put(56,9){ {\rotatebox{45}{\tiny$(m_{k+1}-1)$}}}	
	\put(-3.2,12){{\rotatebox{90}{\tiny$\vr=\mathrm r$}}}
	\put(-4,25){{\rotatebox{90}{\tiny$(\mathrm r_{k},\mathrm r_{k+1})$}}}
	\put(-4,37){{\rotatebox{90}{\tiny$(\mathrm r_{k},\mathrm r_{k+1},\mathrm r_{k+2})$}}}
 	\put(-1,40){ {\rotatebox{90}{\tiny$(0,0,0)$}}}	
 	\put(1.7,40){ {\rotatebox{90}{\tiny$(1,0,0)$}}}	
	\put(7.3,40){ {\rotatebox{90}{\tiny$(0,1,0)$}}}	
 	\put(10,40){ {\rotatebox{90}{\tiny$(1,1,0)$}}}	
 	\put(52,37){ {\rotatebox{45}{\tiny$(0,0,m_{k+3}-1)$}}}	
 	\put(55,37){ {\rotatebox{45}{\tiny$(1,0,m_{k+3}-1)$}}}	
	\put(17,40){{\small$\ldots$}}
	\put(57.5,48){{\tiny$\ldots$}}
	\put(57.5,32.6){{\tiny$\ldots$}}
	\put(57.5,18){{\tiny$\ldots$}}
	\put(68,40){{\tiny$\ldots$}}
 \end{overpic}
  \caption{The internal structure of $\cS_{k+1}$ (bottom), $\cS_{k+2}$ (middle), and $\cS_{k+3}$ (top) and the $k$-strips inside them. Light green is $t_{k+1}$, medium green $t_{k+2}$, dark green $t_{k+3}$.}
 \label{fig.1}
\end{figure}

It turns out that, like in Fact \ref{facPs}, the images $P_k^{(\ell, \vr)}(\cS_k)$ are also pairwise disjoint for any $\ell>k$; moreover, their union covers a large subset of $\cS_\ell$:

\begin{lemma}\label{lemcovmost}
For every $\ell>k$, and every  $\ua^{(\ell)}\in\cA_\ell^\bZ$, there are at most 
\[c\cdot \frac{\rho^{k+1}-\rho^{\ell+1}}{1-\rho}\cdot  R_\ell\]
points in the $\cS_\ell$-orbit segment $\{(\ua^{(\ell)},i)\colon i=0,\ldots,R_\ell-1\}$ that are not contained in 
\[
	\bigcup_{\vr}P_k^{(\ell,\vr)}\Big(\big\{\sigma_k^{j(\vr)}(\underline\Sub_\ell^k(\ua^{(\ell)}))\big\} \times \big\{0,\ldots,  R_k-1\big\}\Big),
\]
where the union is taken over all $(k,\ell)$-tuples $\vr$.
\end{lemma}

\begin{proof}
We argue by induction. Consider first the case $\ell=k+1$ and let $\mathrm r\in \{0,\ldots,m_{k+1}-1\}$. By Fact \ref{facPs}, $P_k^{(k+1, \mathrm r)}(\cS_k)$ are disjoint for different $\mathrm r$'s. Thus, their union covers exactly $m_{k+1}\cdot R_k$ elements of the sequence $(\ua^{(k+1)},0),\ldots, (\ua^{(k+1)},  R_{k+1}-1)$. By Assumption \ref{asstailen}, there are at most $c\rho^{k+1}  m_{k+1}R_k$ and hence at most $c\rho^{k+1}R_{k+1}$ elements which are not covered. This proves the assertion for $\ell=k+1$.

Consider now any $\ell>k+1$ and assume that the assertion of the lemma holds for every $k+1,\ldots,\ell-1$ and every $(k,\ell)$-tuple $\vr=(\mathrm r_k,\ldots,\mathrm r_{\ell-2},\mathrm r_{\ell-1})$. Note that $\vr$ can be considered as the juxtaposition of the $(k,\ell-1)$-tuple $\vr_1=(\mathrm r_k,\ldots, \mathrm r_{\ell-2})$ and the $(\ell-1,\ell)$-tuple $\vr_2=(\mathrm r_{\ell-1})$. By the definition of $P_{\ell-1}^{(\ell, \vr_2)}$ and the argument above, after removing $c\rho^\ell R_\ell$- points, the orbit segment $\{(\ua^{(\ell)},i)\colon i=0,\ldots, R_\ell-1\}$ is covered by the set   
\[
	\Big\{ P_{\ell-1}^{(\ell, \vr_2)}
	\big(\sigma_{\ell-1}^{\vr_2}\circ\underline \Sub_\ell^{\ell-1}(\ua^{(\ell)}),i\big)\colon
		i=0,\ldots, R_{\ell-1}-1, ~\vr_2=0,\ldots, m_\ell-1\Big\}.
\]	 
By the induction assumption, after removing $c(\rho^{k+1}+\cdots+\rho^{\ell-1})R_{\ell-1}$-many points, each of those orbit segments, is in turn covered by images under $P_k^{(\ell-1, \vr_1)}$ of orbit segments from 
\[
	\Big\{\big(\sigma_k^{j_k^{\ell-1}(\vr_1)} \circ\big(\underline\Sub_\ell^k\big)^{-1} (\ua^{(\ell)}),i\big)\colon 
	i=0,\ldots, R_{\ell-1}\Big\}.
\]	 
Observing that
\[
P_k^{(\ell, \vr)} = P_{\ell-1}^{(\ell, \vr_2)}\circ P_k^{(\ell-1, \vr_1)}  
\]
and that
\[
	c(\rho^{k+1}+\cdots+\rho^{\ell-1})R_{\ell-1}m_{\ell}
	+c\rho^{\ell}R_\ell
	\leq  cR_\ell\frac{\rho^{k+1}-\rho^{\ell+1}}{1-\rho},
\]
this finishes the proof.	
\end{proof}

\subsection{Occurrence of $\bar F_{\FK}$-Cauchy sequence}\label{secFK-Cauchy}

For each horseshoe $\bfH_k=(C,\{\bfK_\ba'\}_{\ba\in\cA_k},f^{R_k})$ given in Section \ref{s.cascades-of-horseshoes}, the associated set 
\[
	\Lambda_k
	=\Lambda_{\bfH_k}
	\eqdef\bigcap_{n\in\bZ}f^{nR_k}\big(\bigcup_{\ba\in\cA_k}\bfK_\ba'\big)
\]
is $f^{R_k}$-invariant and $(\Lambda_k,f^{R_k})$ is topologically conjugate to $(\cA_k^\bZ,\sigma_k)$ (recall Lemma \ref{r.factor}). Let us now invoke the suspension system $(\cS_k,\Phi_k)$ of $(\cA_k^\bZ,\sigma_k)$ with constant roof function $R_k$ given in Section \ref{s.cascade-of-suspension-space-symbolic} and the projection map given in Definition \ref{d.projeciton-of-suspension-space},
\[
	\Pi_{k}\colon\cS_{k}\to\bigcup_{i=0}^{R_k-1}f^i(\Lambda_k).
\]
Note that $\Pi_k$ gives the semi-conjugacy between $(\cS_k,\Phi_k)$ and $(\bigcup_{i=0}^{R_k-1}f^i(\Lambda_k),f)$ (recall item 1. in Lemma \ref{r.factor}). Recall also the maps $P_k^{(\ell,\vr)}$ across suspension spaces in \eqref{eqproj} and the ``respelling'' maps $\underline\Sub_\ell^k$ in \eqref{eqSub}.

Given $\underline{a}^{(k+1)}\in\cA_{k+1}^{\bZ}$, consider the subcube $\bfK(\underline{a}^{(k+1)})$ defined in \eqref{defKai}. One has 
\begin{equation}\label{eq:iterating-K-a-by-predecessor}
	f^{\mathrm r R_k}(\bfK(\underline{a}^{(k+1)}))
	\subset \bfK\big(\sigma_k^{\mathrm r}(\underline{\Sub}_{k+1}^k(\underline{a}^{(k+1)}))\big)
	\quad\text{ for every }\quad
	\mathrm r=0,\ldots, m_{k+1}-1.
\end{equation}
Informally speaking, by construction, the iterates of points in $\bfK(\ua^{(k+1)})$ which correspond to a visit to the rectangle $\bfK(\sigma_k^i \circ\underline\Sub_{k+1}^k (\ua^{(k+1)}))$ are the images of the corresponding intermediate floors. This is made precise in the following lemma. Recall Notation \ref{not-jr}.

\begin{lemma}\label{lemwhatever}
Let $\ell > k$. For every $(k,\ell)$-tuple $\vr$ and every $\ua^{(\ell)}\in \cA_\ell^\bZ$,  we have
\[
\Pi_\ell \circ P_k^{(\ell, \vr)}\big(\sigma_k^{j(\vr)}\circ\underline\Sub_\ell^k( \ua^{(\ell)}),0\big)
\in \bfK(\sigma_k^{j(\vr)}\circ \underline\Sub_\ell^k( \ua^{(\ell)})). 
\]
\end{lemma}

\begin{proof}
Let $\ell>k$. Let $\vr=(\mathrm r_k,\ldots, \mathrm r_{\ell-1})$ be a $(k,\ell)$-tuple. Applying Definition \ref{defintfloors} inductively,  one obtains that 
\[
	P_k^{(\ell, \vr)}\big(\sigma_k^{j(\vr)}\circ\underline\Sub_\ell^k( \ua^{(\ell)}),0\big)
	=\big(\ua^{(\ell)}, \sum_{i=k}^{\ell-1}\mathrm r_i\cdot R_{i}\big).
\] 
Note that 
\[
	\big(\ua^{(\ell)}, \sum_{i=k}^{\ell-1}\mathrm r_i\cdot R_{i}\big)
	=\Phi_k^{\sum_{i=k}^{\ell-1}\mathrm r_i\cdot R_{i}}(\ua^{(\ell)}, 0).
\]
By Lemma \ref{r.factor} for  $\Pi_{\ell}$, one has  
\[\begin{split}
	\Pi_{\ell}\big(\ua^{(\ell)}, \sum_{i=k}^{\ell-1}\mathrm r_i\cdot R_{i}\big)
	&=\Pi_{\ell}\circ \Phi_k^{\sum_{i=k}^{\ell-1}\mathrm r_i\cdot R_{i}}(\ua^{(\ell)}, 0)\\
	&= f^{\sum_{i=k}^{\ell-1}\mathrm r_i\cdot R_{i}}\circ \Pi_{\ell}(\ua^{(\ell)}, 0)
		\in f^{\sum_{i=k}^{\ell-1}\mathrm r_i\cdot R_{i}}(\bfK(\ua^{(\ell)})) .
\end{split}\]
By \eqref{eq:iterating-K-a-by-predecessor},  one has 
\[
	f^{r_{\ell-1}R_{\ell-1}}(\bfK(\ua^{(\ell)}))
	\subset \bfK(\sigma_{\ell-1}^{\mathrm r_{\ell-1}}\circ\underline\Sub_\ell^{\ell-1}(\ua^{(\ell)})).
\]
Then 
\[
	\Pi_{\ell}\big(\ua^{(\ell)}, \sum_{i=k}^{\ell-1}\mathrm r_i\cdot R_{i}\big)
	\in f^{\sum_{i=k}^{\ell-1}\mathrm r_i\cdot R_{i}}
		\big(\bfK(\sigma_{\ell-1}^{\mathrm r_{\ell-1}}\circ \underline\Sub_\ell^{\ell-1}(\ua^{(\ell)}))\big).
\]
Once again, using \eqref{eq:iterating-K-a-by-predecessor},  one gets  that 
\[
	f^{\mathrm r_{\ell-2}\cdot R_{\ell-2}}
		\big(\bfK\big(\sigma_{\ell-1}^{\mathrm r_{\ell-2}}\circ\underline\Sub_\ell^{\ell-1}(\ua^{(\ell)})\big)\big)
	\in \bfK\big(\sigma_{\ell-2}^{\mathrm r_{\ell-2}}\circ\underline\Sub_{\ell-1}^{\ell-2}\circ\sigma_{\ell-1}^{\mathrm r_{\ell-1}}\circ\underline\Sub_\ell^{\ell-1}(\ua^{(\ell)})\big), 
\]
and  note that 
\[\begin{split}
	\sigma_{\ell-2}^{\mathrm r_{\ell-2}}\circ\underline\Sub_{\ell-1}^{\ell-2}\circ\sigma_{\ell-1}^{\mathrm r_{\ell-2}}\circ\underline\Sub_\ell^{\ell-1}(\ua^{(\ell)})
	&=\sigma_{\ell-2}^{\mathrm r_{\ell-2}}\circ\sigma_{\ell-2}^{\mathrm r_{\ell-1}m_{\ell-1}}\circ\underline\Sub_{\ell-1}^{\ell-2}\circ\underline\Sub_\ell^{\ell-1}(\ua^{(\ell)})\\
	&=\sigma_{\ell-2}^{\mathrm r_{\ell-2}+\mathrm r_{\ell-1}m_{\ell-1}}\circ\underline\Sub_\ell^{\ell-2}(\ua^{(\ell)}).
\end{split} \]
Inductively apply the arguments above, one gets that 
\[
	f^{\sum_{i=k}^{\ell-1}\mathrm r_i\cdot R_{i}}(\bfK(\ua^{(\ell)}))
	\subset \bfK\big(\sigma_{k}^{\mathrm r_k+\sum_{i=k+1}^{\ell-1}m_{k+1}\cdots m_i\mathrm r_i}	
		\circ\underline\Sub_\ell^k(\ua^{(\ell)})\big)
	= \bfK\big(\sigma_{k}^{j(\vr)}\circ\underline\Sub_\ell^k(\ua^{(\ell)})\big),
\]
where the last equality comes from   the definition of $j(\vr)$ in Notation \ref{not-jr}.

Summarize above and one gets  that 
\[
\begin{split}
\Pi_{\ell}\circ P_k^{(\ell, \vr)}\big(\sigma_k^{j(\vr)}\circ\underline\Sub_\ell^k( \ua^{(\ell)}),0\big)
&= \Pi_{\ell}\big(\ua^{(\ell)}, \sum_{i=k}^{\ell-1}\mathrm r_i\cdot R_{i}\big)
\\
&\in f^{\sum_{i=k}^{\ell-1}\mathrm r_i\cdot R_{i}}(\bfK(\ua^{(\ell)}))\subset \bfK\big(\sigma_{k}^{j(\vr)}\circ\underline\Sub_\ell^k(\ua^{(\ell)})\big).
\end{split}\]  
This proves the lemma.	
\end{proof}

For the following, recall the concept of matches in Definition \ref{defFK}. We obtain that the orbits of the projection of a point in $\cS_\ell$ and the orbits of the projection of its respelling in $\cS_{k}$, $\ell>k$, are matched in large quality. 

\begin{proposition}\label{p.Cauchy}
There exists a constant $c^\prime>0$ such that for  every $k\in\bN$, $\ell>k$, and $\ua^{(\ell)}\in \cA_\ell^\bZ$, one has 
\[\bar F_{\FK}\big(\Pi_\ell(\ua^{(\ell)},0),\Pi_k\big(\underline\Sub_\ell^k(\ua^{(\ell)}),0\big)\big)<c^\prime (\rho^k+2^{-k}). \]
\end{proposition}

\begin{proof}
Consider two integers  $\ell > k$ and fix any $\ua^{(\ell)}\in \cA_\ell^\bZ$. To shorten notation, write 
\[
	\ua=\ua^{(\ell)}, \quad
	\ub=\underline\Sub_\ell^k(\ua)\in\cA_k^\bZ, \quad\text{ and }\quad  
	\varepsilon_k=2^{-k}.
\]	 
To prove the assertion, we need to construct a high quality match between the sequences of points $\big(f^i\circ \Pi_\ell (\ua,0)\big)_i$ and $\big(f^i\circ\Pi_k (\ub,0)\big)_i$. For that it is enough to only consider $0\leq i < R_\ell$ and then repeat the argument for $\sigma_\ell (\ua)$.

Recall that  by item 1. in Lemma \ref{r.factor}, $f^i\circ\Pi_\ell=\Pi_{\ell}\circ\Phi_{\ell}^i$ and $f^i\circ\Pi_k=\Pi_{k}\circ\Phi_{k}^i$, then one has 
\[\big\{ f^i\circ \Pi_\ell (\ua,0): 0\le i\le R_\ell-1\big\}= \big\{  \Pi_\ell (\ua,i): 0\le i\le R_\ell-1\big\} \]
and 
\[\big\{ f^i\circ \Pi_\ell (\ub,0): 0\le i\le R_\ell-1\big\}\supset\bigcup_{\vr} \big\{  \Pi_\ell (\sigma_k^{j(\vr)}(\ub),i): 0\le i\le R_k-1\big\},  \]
where the union is taken over all $(k,\ell)$-tuple $\vr$ and $j(\vr)$ is given in Notation \ref{not-jr}.
By Lemma \ref{lemcovmost}, the subset
\[
\bigcup_\vr  P_k^{(\ell, \vr)} \left(\bigcup_{i=0}^{R_k-1} (\sigma_k^{-j(\vr)}(\ub),i)\right) 
\subset\bigcup_{i=0}^{R_\ell-1}(\ua,i)
\]
has cardinality at least 
\[ R_\ell- c\cdot\frac{\rho^{k+1}-\rho^{\ell+1}}{1-\rho}\cdot R_\ell. \]
In the following, we will find a match between 
\[\bigcup_\vr \Pi_\ell\circ P_k^{(\ell, \vr)} \left(\bigcup_{i=0}^{R_k-1} (\sigma_k^{-j(\vr)}(\ub),i)\right) \]
and 
\[ \bigcup_{\vr} \big\{  \Pi_\ell (\sigma_k^{j(\vr)}(\ub),i): 0\le i\le R_k-1\big\}\]
with large quality, and it suffices to compare them for each  $(k,\ell)$-tuple $\vr$.

Fix a $(k,\ell)$-tuple $\vr$. By Lemma \ref{lemwhatever}, one has 
\[
	\Pi_\ell\circ   P_k^{(\ell, \vr)}\big(\sigma_k^{-j(\vr)}(\ub),0\big) \in \bfK\big(\sigma_k^{j(\vr)} (\ub)\big). 
\]
By item 2 in Lemma \ref{r.factor}, one also has  
\[ \Pi_k  \big(\sigma_k^{j(\vr)} (\ub),0\big)\big)\in \bfK\big(\sigma_k^{j(\vr)} (\ub)\big).\] 
Thus, by Corollary \ref{corinitial}, for every $i=L_{k},\ldots,R_k-L_{k}$, we have
\[
	\dist_M\big(f^i\circ\Pi_\ell\circ  P_k^{(\ell, \vr)} 
		\big(\sigma_k^{-j(\vr)}(\ub),0\big), f^i \circ\Pi_k  \big(\sigma_k^{j(\vr)} (\ub),0\big)\big)< \varepsilon_k.
\]
Once again, by item 1. in Lemma \ref{r.factor}, $f^i\circ\Pi_\ell=\Pi_{\ell}\circ\Phi_{\ell}^i$ and $f^i\circ\Pi_k=\Pi_{k}\circ\Phi_{k}^i$. Then using Fact \ref{fac:commuting-P-and-Phi}, for every $i=L_{k},\ldots,R_k-L_{k}$, one has  
\[
\dist_M\big(\Pi_\ell\circ  P_k^{(\ell, \vr)} 
\big(\sigma_k^{-j(\vr)}(\ub),i\big), \Pi_k  \big(\sigma_k^{j(\vr)} (\ub),i\big)\big)< \varepsilon_k.
\]
This gives us a $(R_k,\varepsilon_k)$-match between 
\[
	 \Pi_\ell \circ P_k^{(\ell, \vr)} \left(\bigcup_{i=0}^{R_k-1} \big(\sigma_k^{-j(\vr)}(\ub),i\big)\right)
	\quad\text{ and }\quad
	 \Pi_k   \left(\bigcup_{i=0}^{R_k-1} \big(\sigma_k^{j(\vr)}(\ub),i\big)\right) 
\]
with quality at least $1-2L_{k}/R_k>1-2^{-k+1}$ due to \eqref{extrassumption}. 

 Summing over all the $(k,\ell)$-tuples, we obtain a $(R_\ell,2^{-k})$-match of $\Pi_\ell(\ua^{(\ell)},0)$ and $\Pi_k\big(\underline\Sub_\ell^k(\ua^{(\ell)}),0\big)$ with quality at least 
\begin{align*}
	 \big(R_\ell- c\cdot\frac{\rho^{k+1}-\rho^{\ell+1}}{1-\rho}\cdot R_\ell\big)\cdot R_\ell^{-1}  \cdot (1-2^{-k+1})
	&>\big(1- c\cdot\frac{\rho^{k+1}-\rho^{\ell+1}}{1-\rho}\big) \cdot (1-2^{-k+1})
	\\
	&>1-c\frac{\rho^{k+1}}{1-\rho}-2^{-k+1}.
	\end{align*}
Taking $c^\prime=c\rho/(1-\rho)+2$, this proves the proposition.
\end{proof}

\begin{proposition}\label{p.uniform-continuity-for-lift-orbit}
For every $\varepsilon>0$, there exist $\delta>0$ and $k\in\bN$ such that for every $\ua,\ub\in\cA^\bZ$ with 
$\bar d(\ua,\ub)<\delta$ and $\ell>k$, one has 
\[
	\bar F_\FK\Big(\Pi_\ell\Big(\big(\underline\Sub_\ell^0\big)^{-1}(\ua),0\Big),\Pi_\ell\Big(\big(\underline\Sub_\ell^0\big)^{-1}(\ub),0\Big)\Big)
	<\varepsilon.
	\]
\end{proposition}

\begin{proof}
Given $\varepsilon>0$, take $k\in\bN$ such that $c^\prime(\rho^k+2^{-k})<\varepsilon/3$, where $c^\prime$ is given by Proposition \ref{p.Cauchy}. Take $\delta>0$ small enough such that 
\begin{equation}\label{someest}
	k\cdot\delta
	<\frac{\varepsilon/3}{2\Mod_k(\varepsilon/3)+1},
\end{equation}
where $\Mod_k$ is the modulus continuity function for $\Pi_k$ defined in \eqref{eqModcon}.

By the triangle inequality, we can estimate
\begin{equation}\label{summands}
	\begin{split}
	&\bar F_\FK\big(\Pi_\ell\big((\underline\Sub_\ell^0)^{-1}(\ua),0\big),\Pi_{\ell}\big((\underline\Sub_\ell^0)^{-1}(\ub),0)\big)\\
	&\leq \bar F_\FK\big(\Pi_\ell\big((\underline\Sub_\ell^0)^{-1}(\ua),0\big),
		\Pi_k\big((\underline\Sub_k^0)^{-1}(\ua),0\big)\big)
		+\bar F_\FK\big(\Pi_k\big((\underline\Sub_k^0)^{-1}(\ua),0\big),\Pi_k\big((\underline\Sub_k^0)^{-1}(\ub),0\big)\big)\\
	&\phantom{\leq}+\bar F_\FK\big(\Pi_k\big((\underline\Sub_k^0)^{-1}(\ub),0\big),\Pi_{\ell}\big((\underline\Sub_\ell^0)^{-1}(\ub),0\big)\big).
\end{split}
\end{equation}

Let us estimate the first and the third terms in \eqref{summands}. Note that  for every $\ell>k$ 
\[
	\underline\Sub_\ell^k\circ  \big(\underline\Sub_\ell^0\big)^{-1}(\underline{e})
	=\big(\underline\Sub_k^0\big)^{-1}(\underline{e}) \quad\text{ for }\underline{e}\in\{\ua,\ub\}.
\]
By Proposition~\ref{p.Cauchy}, one has that 
\begin{equation}\label{eqfirsst}
	\bar F_{\FK} \big(\Pi_\ell\big((\underline\Sub_\ell^0)^{-1}(\underline{e}),0\big), 
		\Pi_k\big((\underline\Sub_k^0)^{-1}(\underline{e}),0\big)\big)
	\leq c^\prime(\rho^k+2^{-k})
	<\frac\varepsilon3
	\quad\text{ for }\underline{e}\in\{\ua,\ub\}.
\end{equation}

Let us now estimate the second term in \eqref{summands}. 

\begin{claim}
	If $\bar d(\ua,\ub)<\delta$, then $\bar d\big((\underline\Sub_k^0)^{-1}(\ua),(\underline\Sub_k^0)^{-1}(\ub)\big)<k\delta$. 
\end{claim}

\begin{proof}
Consider the set $D(\ua,\ub)\eqdef\{i\in\bN_0\colon a_i\neq b_i \}$. Assume that $\bar d(\ua,\ub)<\delta$. Hence, there exists $N\in\mathbb{N}$ such that for every $n\geq N$, 
\[
	\frac{1}{n}\card \big([0,n-1]\cap D(\ua,\ub)\big)<\delta. 
\]
Consider the cardinality of intervals of the form $[ik,(i+1)k-1]$ which intersect the set $D(\ua,\ub)$. For every $n\geq N$, one has 
\[	
	\frac{1}{n}\card\big\{i\in\{0,\ldots,n-1\} 
		\colon [ik,(i+1)k-1]\cap D(\ua,\ub)\neq\emptyset \big\} 
	<\frac{nk\delta}{n}=k\delta.
\]
In particular, together with \eqref{someest}, one has 
\[
	\bar d\big((\underline\Sub_k^0)^{-1}(\ua),(\underline\Sub_k^0)^{-1}(\ub)\big)
	<k\delta
	<\frac{\varepsilon/3}{2\Mod_k(\varepsilon/3)+1},
\]
proving the claim.
\end{proof}

Finally, applying Lemma~\ref{suspmetr} to the alphabet $\cA_k$, the factor $\Pi_k$, and the sequences $\big(\underline\Sub_k^0\big)^{-1}(\ua),\big(\underline\Sub_k^0\big)^{-1}(\ub)$, we get  
\begin{equation}\label{eqfirsstbis}
	\bar F_\FK\big(\Pi_k\big((\underline\Sub_k^0)^{-1}(\ua),0\big),\Pi_{k}\big((\underline\Sub_k^0)^{-1}(\ub),0\big)\big)
	<\frac{\varepsilon}{3}.
\end{equation}
Hence, by \eqref{eqfirsst} and \eqref{eqfirsstbis} together with \eqref{summands}, one can conclude.
\end{proof}

\subsection{Uniform $\FK$-convergence across the cascade of horseshoes}\label{ssecUniformFK}

Using the cascades introduced above, we now ``push'' a totally ergodic measures for $(\cA^\bZ,\sigma_\cA)$ to an $f$-ergodic measure by a map $\bH$ which has very convenient properties.  Recall that a measure $\nu\in\cM_{\rm erg}(\sigma_\cA)$ is \emph{totally ergodic} if it is $\sigma_\cA^k$-ergodic for any $k\in\bN$;  denote by $\cM_{\rm TE}(\sigma_\cA)$ the space of all such measures.

\subsubsection{The map $\bH$}

Consider a measure $\nu\in\cM_{\rm TE}(\sigma_\cA)$. Given any $k\in\bN$, by Remark \ref{remtopconjugate}, $(\cA_k^{\bZ},\sigma_k)$ is topologically conjugate to $(\cA^\bZ,\sigma_\cA^{m_1\cdots m_k})$ via $\underline\Sub_k^0$. Consider the measure
\begin{equation}\label{deflambdaketc}
 	\nu_k\eqdef\big((\underline\Sub_k^0)^{-1}\big)_\ast\nu,
	\quad
	\lambda_k\eqdef\lambda_{\cA_k,R_k, \nu_k },
	\quad\text{ and }\quad
	\mu_k
	\eqdef(\Pi_k)_\ast(\lambda_k).
\end{equation}
As $\nu\in\cM_{\rm erg}(\sigma_\cA)$ is totally ergodic, the measure $\nu_k$ is ergodic for $\sigma_k$.  Recall the definition of $\lambda_k$ in \eqref{eqtildelambda}. By Remark \ref{rem:some}, $\lambda_k$ is an ergodic measure for $(\cS_{\cA_k,R_k},\Phi_{\cA_k, R_k})$. Hence, the measure $\mu_k$ is $f$-ergodic.

\begin{claim}\label{cl.mu-k-Cauchy}
The sequence $(\mu_k)_{k\in\bN}$ of $f$-ergodic measures is $\bar F_\FK$-Cauchy. Moreover, for every $k\in\bN$,
\begin{equation}\label{eqexponent}
	\rho^k(1+\delta)\widehat\chi
	<\chi^\c(\mu_k)-\chi
	\le \rho^k(1-\delta)\widehat\chi.
\end{equation}
\end{claim}

\proof 
Take a generic point $\ua\in\cA^\bZ$ of $\nu$.  For every $k\in\bN$, let 
\[
	x_k\eqdef \Pi_k\big((\underline\Sub_k^0)^{-1}(\ua),0\big).
\]	 
By the definition of $\mu_k$,  the point $x_k$ is $\mu_k$-generic. By Proposition \ref{p.Cauchy}, the sequence $(x_k)_k$ is $\bar F_\FK$-Cauchy. Hence by Definition \ref{defFKerg}, the sequence of measures $(\mu_k)_k$ is a $\bar F_\FK$-Cauchy sequence. 

As $\mu_k$ is an ergodic measure supported on $\bigcup_{i=0}^{R_k-1}f^i(\Lambda_k)$, from \eqref{eq:bound-center-LE-for-kth-horseshoe}, we get \eqref{eqexponent}.
\endproof	

By Theorem \ref{thepro:ergentcon}, there exists a unique ergodic measure $\mu\in\cM_{\rm erg}(f)$ which is the $\bar F_\FK$-limit of the sequence $(\mu_k)_k$. As, also by Theorem \ref{thepro:ergentcon}, the $\bar F_\FK$-convergence implies convergence in the weak$\ast$ topology, together with the continuity of the center bundle,   the center Lyapunov exponent of $\mu$ is $\chi$.   Letting $\bH(\nu)\eqdef\mu$, this defines a map
\begin{equation}\label{defbH}
	\bH\colon \big(\cM_{\rm TE}(\sigma_\cA),\bar d\big)\to \big(\cM_{\rm erg,\chi}(f),\bar F_\FK\big).
\end{equation}
 
\begin{proposition}\label{p.existence-of-uniform-continuous-map-between-measure-spaces}
	The map $\bH$ in \eqref{defbH} is uniformly continuous.
\end{proposition}
 
 \begin{proof}
Given $\varepsilon>0$, let $\delta>0$ be given by Proposition~\ref{p.uniform-continuity-for-lift-orbit}.
Consider two totally ergodic measures $\nu,\nu^\prime\in\cM_{\rm TE}(\sigma_\cA)$ with $\bar d(\nu,\nu^\prime)<\delta$. By definition of $\bar d(\cdot,\cdot)$, there exist $\ua\in G(\nu)$ and $\ua^\prime\in G(\nu^\prime)$ such that $\bar d(\ua,\ua^\prime)<\delta$.
Consider  the sequence of $f$-ergodic measures $(\mu_k)_k$ and $(\mu_k')_k$ in the definitions of $\bH(\nu)$ and $\bH(\nu')$. Then $\Pi_k\big((\underline\Sub_k^0)^{-1}(\ua),0\big)$ and $\Pi_k\big((\underline\Sub_k^0)^{-1}(\ua^\prime),0\big)$ are $\mu_k$-generic and $\mu_k^\prime$-generic, respectively.
By Proposition \ref{p.uniform-continuity-for-lift-orbit}, there exists $k_0\in\bN$ such that for each $k\geq k_0$, one has 
\[
	\bar F_\FK\bigg(\Pi_k\big((\underline\Sub_k^0)^{-1}(\ua),0\big),
				\Pi_k\big((\underline\Sub_k^0)^{-1}(\ua^\prime),0\big)\bigg)
	<\varepsilon
\]
This implies $\bar F_\FK\big(\mu_k,\mu_k^\prime\big)<\varepsilon$. By the definition of $\bH(\cdot)$, we get that $\bar F_\FK\big(\bH(\nu),\bH(\nu^\prime)\big)\leq\varepsilon$, ending the proof of the proposition.
\end{proof}

\subsubsection{Comparing entropies}
We now compare the entropy of $\nu$ with the one of $\bH(\nu)$. 

\begin{proposition}\label{p.entropy-bound-for-map-bH}
There exists a constant $c_0>0$ such that for every $\nu\in\cM_{\rm TE}(\sigma_\cA)$, one has 
\[
	\frac{h_\nu(\sigma_\cA)}{R_0} 
	\geq 
	h_{\bH(\nu)}(f)\geq (1+c_0\widehat\chi) \frac{h_\nu(\sigma_\cA)}{R_0}.  
\]
\end{proposition}

\proof 
Consider the measures $\nu_k$, $\lambda_k$, and $\mu_k$ in \eqref{deflambdaketc}.   
As, by Remark \ref{remtopconjugate}, $\underline\Sub_k^0$ is a conjugacy between $(\cA^\bZ,\sigma_\cA^{m_1\cdots m_k})$ and $(\cA_k^\bZ,\sigma_k)$, one gets 
\[
	h_{\nu_k}(\sigma_k)
	= h_{\nu}(\sigma_\cA^{m_1\cdots m_k})
	= m_1\cdots m_k\cdot h_\nu(\sigma_\cA).  
\]
By the Abramov's formula \cite{Abr:59}, one gets 
\[
	h_{\lambda_k}(\Phi_{k})
	=\frac{1}{R_k}h_{\nu_k}(\sigma_k)
	=\frac{m_1\cdots m_k}{R_k}h_\nu(\sigma_\cA). 
\]
As, by item 3 in Lemma \ref{r.factor}, the map $\Pi_k$ preserves the entropy and hence we get
\begin{equation}\label{somenumber}
	h_{\lambda_k}(\Phi_{k})
	= h_{(\Pi_k)_\ast\lambda_k}(f)
	= h_{\mu_k}(f)
	=\frac{m_1\cdots m_k}{R_k}h_\nu(\sigma_\cA).
\end{equation}

Recall that $R_k=m_k\cdot R_{k-1}+t_k$.
By \eqref{eqdefRk}, for each $i\in\bN$ one has 
\[
	0\le t_i\leq \xi \cdot m_i\cdot R_{i-1}\cdot \rho^i\cdot |\widehat\chi| ,
\]
which implies that 
\begin{equation}\label{someothernumberbis}
	R_k
	=m_k\cdot R_{k-1}+t_k
	\leq m_k\cdot R_{k-1}\cdot (1+\xi  \cdot \rho^k\cdot |\widehat\chi|)
	\leq m_1\cdots m_k\cdot R_0 \prod_{i=1}^k(1+\xi  \cdot \rho^i\cdot |\widehat\chi|)
\end{equation}
and
\[
	R_k=m_k\cdot R_{k-1}+t_k\geq m_k\cdot R_{k-1}\geq m_1\cdots m_k\cdot R_0.
\]
This together with \eqref{somenumber} implies 
\[
	\frac{h_\nu(\sigma_\cA)}{R_0}
	\geq h_{\mu_k}(f).
\]
Moreover, we get
\begin{align*}
	h_{\mu_k}(f)
	&= \frac{m_1\cdots m_k}{R_k}\cdot h_\nu(\sigma_\cA)\\
	\text{\tiny by \eqref{someothernumberbis}}\quad
	&\geq \frac{1}{\prod_{i=1}^k(1+\xi  \cdot \rho^i\cdot |\widehat\chi|)}\cdot \frac{h_\nu(\sigma_\cA)}{R_0}
	\ge \Big(e^{-\sum_{i=1}^k \xi  \cdot \rho^i\cdot |\widehat\chi|}\Big)\frac{h_\nu(\sigma_\cA)}{R_0}\\
	&>\big(e^{-\xi \rho\cdot |\widehat\chi|/(1-\rho)}\big) \frac{h_\nu(\sigma_\cA)}{R_0}
	>\big(1+\frac{\xi \rho}{1-\rho}\widehat\chi\big) \frac{h_\nu(\sigma_\cA)}{R_0}.
\end{align*}	
Finally, letting $c_0=\xi \rho/(1-\rho)$, we get 
\[
	\frac{h_\nu(\sigma_\cA)}{R_0}
	\geq h_{\mu_k}(f)
	>\big(1+c_0\widehat\chi\big) \frac{h_\nu(\sigma_\cA)}{R_0}.
\]
As, by definition, $\bfH(\nu)$ is the $\bar F_\FK$-limit of $\mu_k$, by  Corollary \ref{c.continuity-of-entropy-under-FK-for PH-1D} the proposition follows.
\endproof 

\subsection{Proof of Theorem \ref{thm.horseshoe-to-path-of-measure}}

By Remark \ref{r.space-of-Bernoulli-measures}, there is a continuous path $\{\nu_t \}_{t\in[0,1]}$ in the space of Bernoulli measures on $(\cA^\bZ,\sigma_\cA)$ (with respect to the topology $\bar d$) such that 
\[
	h_{\nu_0}(\sigma_\cA)=0,
	\quad 
	h_{\nu_1}(\sigma_\cA)=\log\card(\cA). 
\]
Recall that any Bernoulli measure is totally ergodic. By Proposition \ref{p.existence-of-uniform-continuous-map-between-measure-spaces}, $\big\{\bH(\nu_t) \big\}_{t\in[0,1]}$ is a continuous path in $\cM_{\rm erg, \chi}(f)$ under the $\bar F_\FK$-metric. By Corollary  \ref{c.continuity-of-entropy-under-FK-for PH-1D}, the map $t\mapsto h_{\bH(\nu_t)}(f)$ is continuous. 
Let $c_0>0$ be the constant given by Proposition \ref{p.entropy-bound-for-map-bH}. Then one has
\[ 
	\frac{h_{\nu_t}(\sigma_\cA)}{R_0}
	\geq h_{\bH(\nu_t)}(f)\geq (1+c_0\widehat\chi)  \frac{h_{\nu_t}(\sigma_\cA)}{R_0}
\]
which implies that 
\[ 
	h_{\bH(\nu_0)}(f)=0, \quad
	h_{\bH(\nu_1)}(f)\geq (1+c_0\widehat\chi) \frac{\log\card(\cA)}{R_0}. 
\]
Recall that $\card\cA=N, R_0=R$, and this ends the proof of the theorem.
\qed

\section{Proofs of the main results}\label{secfinal}

In this section, we complete the proofs of Theorems \ref{thm.maingeneral} and \ref{thm.main}. 
Indeed, we start by proving the following slightly more technical version of Theorem  \ref{thm.maingeneral}.

\begin{theorem}\label{rthm.ergodic-measure-to-path-of-nonhyperbolic-ergodic-measure}
	Let $f\in\PH_{\c=1}^1(M)$. Assume that
\begin{itemize}[leftmargin=0.6cm ]
\item[(1)] $f$ has an unstable blender-horseshoe,
\item[(2)] the foliation $\cW^\uu$ is  minimal,
\item[(3)] $f$ has a saddle of contracting type.
\end{itemize} 
	 Then  there exist numbers  $\chi_0<0$ and $c_0>0$  such that for any $\chi\in(\chi_{\rm min}(f),0]$, any   $\nu\in\cM_{\rm erg}(f)$ with $\chi^\c(\nu)\in(\chi_0+\chi,\chi)$ and any $\varepsilon>0$, there exists a continuous path $\{\mu_t\}_{t\in[0,1]}\subset \cM_{\rm erg}(f)$ such that 
	\begin{itemize}[leftmargin=0.6cm ]
		\item $\chi^\c(\mu_t)=\chi$ for any $t\in[0,1]$;
		\item $t\mapsto h_{\mu_t}(f)$ is continuous;
		\item $h_{\mu_0}(f)=0$, and 
		\[ h_{\mu_1}(f)\geq \big(1+c_0(\chi^\c(\nu)-\chi)\big) (h_\nu(f)-\varepsilon).\]
	\end{itemize}
\end{theorem}

\begin{proof}
Let $c_0>0$, $\delta_0>0$, and $\chi_0<0$ be given by Theorem \ref{thm.horseshoe-to-path-of-measure}. 
Fix $\chi\in(\chi_{\rm min}(f),0]$, $\varepsilon>0$ and an ergodic measure $\nu$ with $\chi^\c(\nu)\in(\chi_0+\chi,\chi)$.

Define $\widehat\chi=\chi^\c(\nu)-\chi\in(\chi_0,0).$
Take 
\[\delta<\min\{\delta_0,1\} \quad \textrm{and} \quad\varepsilon_{\rm L}\in\big(0,\min\{\delta|\widehat\chi|,\varepsilon\}\big) .\]
Then one has 
\[\chi-(1+\delta)|\widehat\chi| <\chi^\c(\nu)-\varepsilon_{\rm L}<\chi^\c(\nu)+\varepsilon_{\rm L}<\chi-(1-\delta)|\widehat\chi|. \]
 
Applying Theorem \ref{t.smarthorseshoes} to the ergodic measure $\nu$,  $\varepsilon_{\rm L}$ and $r<\varepsilon_1/3$, one obtains   $R\in\bN$, a center curve $\gamma$ and  a $\csu$-horseshoe $\bfH=(C,\{\bfK_i \}_{i=1}^N, f^R)$ relative to an $\csu$-cube $C=C^\csu(\gamma,r)$ such that 
\begin{equation}\label{TheoremBentropy}
	N
	\geq \exp\big(R(h_\nu(f)-\varepsilon_{\rm L})\big) > \exp\big(R(h_\nu(f)-\varepsilon)\big) 
\end{equation}
and for each $i=1,\ldots, N$, and every $x\in\bfK_i$, 
\[\begin{split}
	-(1+\delta)|\widehat\chi|
	<  \chi^\c(\nu)-\varepsilon_{\rm L}-\chi
	<\frac{1}{R}\|D^\c f^R(x)\|-\chi
	< \chi^\c(\nu)+\varepsilon_{\rm L}-\chi
	<-(1-\delta)|\widehat\chi|.
\end{split}\]
Applying Theorem \ref{thm.horseshoe-to-path-of-measure} to the horseshoe $\bfH=(C,\{\bfK_i \}_{i=1}^N, f^R)$,  one gets  a continuous path $\{\mu_t \}_{t\in[0,1]}\subset\cM_{\rm erg}(f)$ such that 
\begin{itemize}[leftmargin=0.6cm ]
	\item $\chi^\c(\mu_t)=\chi$ for any $t\in[0,1]$,
	\item $t\mapsto h_{\mu_t}(f)$ is continuous,
	\item $h_{\mu_0}(f)=0$, and using also \eqref{TheoremBentropy} we get
	\[ 
	h_{\mu_1}(f)
	\geq \big(1+c_0\widehat\chi\big) \cdot \frac{\log N}{R}
	\geq \big(1+c_0(\chi^\c(\nu)-\chi)\big) \cdot \big(h_{\nu}(f)-\varepsilon\big).
	\]
\end{itemize}
This finished the proof of the theorem.
\end{proof}

\begin{proof}[{Proof of Theorem~\ref{thm.maingeneral}}]
	This is now an immediate consequence of Theorem \ref{rthm.ergodic-measure-to-path-of-nonhyperbolic-ergodic-measure}.
\end{proof}

\begin{proof}[{Proof of Theorem~\ref{thm.main}}]
	The assertion is an immediate consequence of Theorem \ref{thm.maingeneral} together with Remark \ref{r.DGZ}.
\end{proof}

\section{Nonergodic limit measures: Proof of Theorem \ref{nonerglims}}\label{secnonerglims}

Recall the definition of the \emph{Wasserstein metric} $W(\cdot,\cdot)$ on the space $\cM(f)$. For that, let
\[
	\cL
	\eqdef \big\{\varphi\colon M\to\bR\colon	|\varphi|\leq 1,\Lip(\varphi)\leq 1\big\}
\]
and define	
\begin{equation}\label{eqWasser}
	W(\mu, \mu') 
	\eqdef \sup\left\{ \left| \int \varphi \,d\mu - \int \varphi \,d\mu' \right|\colon \varphi\in\cL\right\}.
\end{equation}

\begin{remark}\label{remWasser}
The above indeed is a metric and induces the weak$\ast$ topology on $\cM(f)$. Observe that the space of functions $\cL$ is compact. Hence, for any $\varepsilon$ there is a finite subcollection $\varphi_1,\ldots, \varphi_{N(\varepsilon)}$ which allow us to estimate $W(\mu_1, \mu_2)$ with an error less than $\varepsilon$ for {\it any} pair of measures $\mu_1, \mu_2\in\cM(f)$.
\end{remark}

\begin{proof}[Proof of Theorem \ref{nonerglims}]
	Let $\vartheta$ be as in the theorem and consider some neighborhood $U$ of $\vartheta$. First observe that, by \cite[Theorem 1]{DiaGelSan:20} or \cite[Theorem 1.1]{YanZha:20}, there exists $\vartheta'\in\cM_{\rm erg,<0}(f)\cap U$ which is arbitrarily close to $\vartheta$ in  weak$\ast$-topology and in entropy. By Remark \ref{remWasser}, there is a finite collection $\varphi_1,\ldots,\varphi_\ell\in\cL$ and $\delta>0$ such that the (open) set
\begin{equation}\label{eqexponvartheta}
	\Big\{\mu\colon \mu\in\cM(f), \Big| \int \varphi_i\,d\mu - \int \varphi_i \,d\vartheta' \Big|<\delta \text{ for every } i=1,\ldots,\ell\Big\}
	\subset U.
\end{equation}
By Theorem \ref{t.smarthorseshoes}, there exist an $\csu$-cube $C$ and a $\csu$-horseshoe $\bfH=(C,\{\bfK_i\}_{i=1}^N,f^R)$ of contracting type such that  
\[
	\cM(f|_{\Lambda_\bfH})
	\subset U.
\]
Moreover, using \eqref{eqexponvartheta}, we get that for every $i=1,\ldots,N$ and $x\in\bfK_i$,
\[
	e^{R(\chi^\c(\vartheta')-\delta)}< \|D^\c f^R(x)\|<e^{R(\chi^\c(\vartheta')+\delta)}.	
\]
Assume, for simplicity, that we have the same quantifiers $T_0$, $\rho_{\rm b}$, $\kappa_{\rm min}$, and $\kappa_{\rm max}$ associated to the cube $C$ and to the unstable blender-horseshoe $\fB =\fB(\bfC,f^S)$ as in Section \ref{s.connecting}.

Let $\mu_\bfH$ be the measure of maximal entropy of $f|_{\widetilde\Lambda_\bfH}$. To conclude the proof of the theorem, it is enough to prove the following lemma.

\begin{lemma}
	For every neighborhood $V$ of $\mu_\bfH$ in the weak$\ast$ topology and entropy, there are $\mu_\infty\in\cM(f)\cap V$ and a sequence $(\mu_k)_k\subset\cM_{\rm erg}(f)\cap V$ having the following properties
\begin{enumerate}
\item $(\mu_k)_k$  converges to $\mu_\infty$ in entropy and in the weak$\ast$ topology,
\item $\mu_\infty$ is a nontrivial convex combination of $\mu_\bfH$ and some measure $\mu^+\in\cM(f|_{\widetilde\Lambda_\fB})$.
\end{enumerate}
\end{lemma}

\begin{proof}
We start by following the steps of the proof of Proposition \ref{ss.connecting} until Equation \eqref{e.hatKab}. We choose a sequence of subordinated horseshoes $(\bfH_m)_m$. Fix first $\varepsilon_0>0$ small so that
\begin{equation}\label{followsfrom}
	\frac{\varepsilon_0}{1-\varepsilon_0}
	< \frac{1}{\kappa_{\rm max}}|\chi^\c(\vartheta')-\delta|.
\end{equation}
Note that every $\varepsilon\in(0,\varepsilon_0)$ also satisfies this inequality. Fix $\varepsilon\in(0,\varepsilon_0)$. 

For every $m\in\bN$, every $m$-word $\ba=(a_0, \ldots, a_{m-1}) \in \{1, \ldots, N\}^m$, and $j=0,\ldots,\ell(\ba)$, we get a $\cs$-complete subcube $\bfK_{\ba}^{(j)}$ of $\bfK_\ba$. Recall that, by Claim \ref{claimellba}, we get the estimate 
\[
	\ell(m)
	\eqdef \min_{\ba\in\{1,\ldots,N\}^m}\ell(\ba)
	\ge \frac{1}{\kappa_{\rm max}}
		\Big(\log\rho_{\rm b}-\log\kappa_{\rm max}+mR|\chi^\c(\vartheta')-\delta|\Big).
\]
As a consequence of \eqref{followsfrom}, there is $m_0\in\bN$ so that for every $m\ge m_0$, it holds
\begin{equation}\label{usedd}
	\frac{\varepsilon}{1-\varepsilon}mR+1
	<\ell(m).
\end{equation}

Given $m\ge m_0$, let 
\begin{equation}\label{choiceeps0}
	R_m\eqdef 
	mR+T_{\rm b}(m,\varepsilon),
	\quad\text{ where }\quad
	T_{\rm b}(m,\varepsilon)
	\eqdef \Big\lfloor \frac{\varepsilon}{1-\varepsilon}mR\Big\rfloor
	<\ell(m),
\end{equation}
where we used \eqref{usedd}.
The proof of the following result is analogous to the one of Claim \ref{cl:choiceKaprime}, taking into account that with \eqref{choiceeps0}, we indeed stay in the blender domain $\bfC_1\cup\bfC_2$.

\begin{claim}
For every $m\ge m_0$ and every $m$-word $\ba\in\{1,\ldots,N\}^m$, there exists a $\cs$-complete subcube  $\bfK'_\ba=\bfK'_\ba(m)$ of $\bfK_\ba$ such that
\[	f^k(\bfK'_\ba)\subset
	\begin{cases}
	\bigcup_{i=1}^N\bfK_i&\text{ for }k=0,\ldots, mR,\\
	\bfC_1\cup\bfC_2&\text{ for }k=mR+T_0,\ldots, mR+T_{\rm b}(m,\varepsilon)-T_0,\\
	\bigcup_{i=1}^N\bfK_i&\text{ for } k=R_m.
	\end{cases}
\]
Moreover, $f^{R_m}(\bfK'_\ba)$ is a $\uu$-complete subcube of $C$.
\end{claim}

Let
\begin{equation}\label{defHm}
	\bfH_m
	\eqdef (C,\{\bfK_\ba'\}_{\ba\in\{1,\ldots,N\}^m},f^{R_m}).
\end{equation}
As in \eqref{defLambdatilde}, let 
\[
	\Lambda_m
	\eqdef \Lambda_{\bfH_m},
	\quad
	\widetilde\Lambda_m
	\eqdef \widetilde\Lambda_{\bfH_m}.
\]
Let $\mu_m$ be the measure of maximal entropy of $f^{R_m}|_{\Lambda_m}$. Recall that $\mu_m$ is the Bernoulli measure in which every rectangle $\bfK_\ba'$ has probability $p=N^{-m}$. 
Analogously, as in \eqref{basicset}, consider $\Lambda_\fB$ and $\widetilde\Lambda_\fB$ associated to the unstable blender-horse\-shoe.

Consider now the full shift $\sigma_m\colon\cA_m^\bZ\to\cA_m^\bZ$, where $\cA_m\eqdef\{1,\ldots,N\}^m$.
Applying Lemma \ref{r.factor} to the $\csu$-horseshoe $\bfH_m$ in \eqref{defHm}, the roof function $R_m$ in \eqref{choiceeps0}, and the alphabet $\cA_m$, we get a topological conjugacy between the discrete-time suspension $(\cS_m,\Phi_m)$  of $(\cA_m^\bZ,\sigma_m)$ with roof $R_m$ and the map $f|_{\widetilde\Lambda_m}$. Like in the proof of Lemma \ref{r.factor}, using the ``tower structure'' of the suspension space and the topological conjugation, any $f$-invariant probability measure $\mu$ on $\widetilde\Lambda_m$ can be written as the sum
\begin{equation}\label{decompos}
	\mu
	= \sum_{s=0}^{R_m-1}\mu^{(s)},
\end{equation}
where each measure $\mu^{(s)}$ is $f^{R_m}$-invariant, satisfies $f_\ast\mu^{(s)}=\mu^{(s+1)}$, for $s=0,\ldots,R_m-2$, and $|\mu^{(0)}| = \mu^{(0)}(\bfK_\ba')=1/R_m$. 

We apply now \eqref{decompos} to $\mu_m$ to get the family of measures $\{\mu_m^{(s)}\}_{s=0}^{R_m-1}$. Define
\begin{equation}
	\mu_m^-
	\eqdef \sum_{s=0}^{mR-1}\mu_m^{(s)},
	\quad
	\mu_m^+
	\eqdef \sum_{s=mR}^{R_m-1}\mu_m^{(s)}
	= \mu_m-\mu_m^-.
\end{equation}

Check  that by \eqref{choiceeps0}, we have
\begin{equation}\label{relRmmR}\begin{split}
	\frac{mR}{R_m}
	= \frac{mR}{mR+T_{\rm b}(m,\varepsilon)}
	&= 1-\varepsilon_m
	\quad\text{ and }\,\\
	\frac{T_{\rm b}(m,\varepsilon)}{mR+T_{\rm b}(m,\varepsilon)}
	&=\varepsilon_m,
	\quad\text{ where }
	\varepsilon_m\to\varepsilon
	\text{ as }m\to\infty.
\end{split}\end{equation}
Noting that $T_{\rm b}(m,\varepsilon)\to\infty$ as $m\to\infty$, we get that $(\Lambda_m)_m$ accumulates on  $\Lambda_\bfH$ in the Hausdorff distance, as $m\to\infty$. Hence, the weak$\ast$ limit of $\mu_m^-$ is some measure supported on $\Lambda_\bfH$. By the Bogolyubov-Krylov argument, this limit measure is $f$-invariant. Using Corollary \ref{c.distortion-reference-cube} the same way as in Section \ref{s.connecting}, we can choose some $\delta_m\searrow 0$ and $L_m/R_m \searrow 0$ such that almost every trajectory of  
\[
\mu_m^{- -} \eqdef \sum_{s=L_m}^{mR-L_m-1} \mu^{(s)}
\]
is $\delta_m$-close to some generic trajectory of $\mu_\bfH$, and that $|\mu_m^-|-|\mu_m^{- -}| = 2L_m/R_m$. Thus,   
\[
	\mu_m^-\to(1-\varepsilon)\mu_{\bfH}
		\quad\text{ as }\quad
	m\to\infty
\]
in the weak$\ast$ topology.

Note that the sequence $(\mu_m^+)_m$ may fail to converge. However, up to passing to a subsequence, we can assume that it indeed converges, and we denote by $\mu^+$ this limit measure. By the above arguments, this weak$\ast$ limit is an $f$-invariant Borel measure. Arguing as above, we get $\mu^+(\widetilde \Lambda_\fB)=\varepsilon$.

It remains to check the assertion about the entropy. By construction, by Abramov's formula, we have
\[
	h_{\mu_m}(f)=(1-\varepsilon) h_{\mu_\bfH}(f).
\]
As
\[
	\mu_\infty
	\eqdef \lim_{m\to\infty} \mu_m 
	= \lim_{m\to\infty} \mu_m^- + \lim_{m\to\infty} \mu_m^+ = (1-\varepsilon) \mu_\bfH + \varepsilon \mu^+,
\]
by affinity of entropy we have
\[
(1-\varepsilon) h_{\mu_\bfH}(f) \leq h_{\mu_\infty}(f) \leq (1-\varepsilon) h_{\mu_\bfH}(f) + \varepsilon h_{\rm top}(f,\widetilde\Lambda_\fB),
\]
and the assertion follows.
\end{proof}
This proves the theorem.
\end{proof}

\end{document}